\documentclass[letterpaper,reqno,10pt]{amsart}
\usepackage{amsmath,amsthm,amsfonts,amssymb,euscript,mathrsfs,graphics,color,latexsym,marginnote}
\usepackage[dvips]{graphicx}
\usepackage[margin=0.8in]{geometry}
\usepackage{hyperref}
\usepackage[toc,page]{appendix}
\usepackage{relsize}
\usepackage[shortlabels]{enumitem}

\usepackage[dvipsnames]{xcolor}

\usepackage{hyperref}
\hypersetup{colorlinks=true, pdfstartview=FitV,linkcolor=blue!70!black,citecolor=red!70!black, urlcolor=green!60!black}
\definecolor{labelkey}{rgb}{0.6,0,0} 

\newtheorem{theorem}{Theorem}[section]

\newtheorem{lemma}[theorem]{Lemma}
\newtheorem{proposition}[theorem]{Proposition}

\theoremstyle{definition}
\newtheorem{definition}{Definition}[section]
\theoremstyle{remark}
\newtheorem{remark}{Remark}[section]

\numberwithin{equation}{section}

\providecommand{\abs}[1]{\left\lvert #1 \right\rvert}
\providecommand{\babs}[1]{\big\lvert #1 \big\rvert}
\providecommand{\Babs}[1]{\Big\lvert #1 \Big\rvert}
\providecommand{\nm}[1]{\left\lVert #1 \right\rVert}
\providecommand{\bnm}[1]{\big\lVert #1 \big\rVert}
\providecommand{\Bnm}[1]{\Big\lVert #1 \Big\rVert}
\providecommand{\br}[1]{\left\langle #1 \right\rangle}

\providecommand{\tnm}[1]{\left\lVert #1 \right\rVert_{L^2}}
\providecommand{\btnm}[1]{\big\lVert #1 \big\rVert_{L^2}}
\providecommand{\Btnm}[1]{\Big\lVert #1 \Big\rVert_{L^2}}
\providecommand{\lnm}[1]{\left\lVert #1 \right\rVert_{L^{\infty}}}
\providecommand{\blnm}[1]{\big\lVert #1 \big\rVert_{L^{\infty}}}
\providecommand{\Blnm}[1]{\Big\lVert #1 \Big\rVert_{L^{\infty}}}

\providecommand{\lnmx}[1]{\left\lVert #1 \right\rVert_{L^{\infty}_{\xi}}}
\providecommand{\blnmx}[1]{\big\lVert #1 \big\rVert_{L^{\infty}_{\xi}}}
\providecommand{\Blnmx}[1]{\Big\lVert #1 \Big\rVert_{L^{\infty}_{\xi}}}

\providecommand{\lvwaa}[1]{\widehat{V^{wa,#1}_{\infty}}}
\providecommand{\lvkgg}[1]{\widehat{V^{kg,#1}_{\infty}}}
\providecommand{\hgwaa}[1]{\widehat{G^{wa,#1}}}
\providecommand{\hgkgg}[1]{\widehat{G^{kg,#1}}}

\def\dd{\mathrm{d}}
\def\Im{\mbox{Im}}
\def\Re{\mbox{Re}}
\def\Z{{\mathbb Z}}
\def\id{\,\mbox{id} \,}
\def\d{\delta}

\def\ud{\mathrm{d}}
\def\dt{\partial_t}
\def\p{\partial}
\def\ls{\lesssim}
\def\gs{\gtrsim}

\def\rt{\rightarrow}
\def\r{\mathbb{R}}
\def\no{\nonumber}
\def\ue{\mathrm{e}}
\def\ui{\mathrm{i}}

\def\na{\nabla}

\def\la{\Lambda}

\def\ep{\varepsilon}
\def\e{\varepsilon}
\def\lll{\mathcal{L}}
\def\vvv{\mathcal{V}}
\def\vv{\mathscr{V}}
\def\D{\mathbf{D}}

\def\uwa{U^{wa}}
\def\ukg{U^{kg}}
\def\vwa{V^{wa}}
\def\vkg{V^{kg}}
\def\hvwa{\widehat{V^{wa}}}
\def\hvkg{\widehat{V^{kg}}}
\def\lvwa{\widehat{V^{wa}_{\infty}}}
\def\lvkg{\widehat{V^{kg}_{\infty}}}

\def\gwa{G^{wa}}
\def\gkg{G^{kg}}
\def\hgwa{\widehat{G^{wa}}}
\def\hgkg{\widehat{G^{kg}}}
\def\llgwa{{G^{wa}_{\lll}}}
\def\llgkg{{G^{kg}_{\lll}}}
\def\hlgwa{\widehat{G^{wa}_{\lll}}}
\def\hlgkg{\widehat{G^{kg}_{\lll}}}

\def\lawa{\Lambda_{wa}}
\def\lakg{\Lambda_{kg}}

\def\rwa{R^{wa}}
\def\rkg{R^{kg}}

\def\ss{\mathcal{S}}

\def\rr{\mathcal{R}}

\def\bbi{\text{\Large$\mathbf{I}$}}

\def\bbii{\mathcal{I}}

\def\bbrr{\mathcal{R}}
\def\bbee{\mathcal{E}}
\def\bbkk{\mathcal{K}}

\def\ccc{\mathcal{C}}
\def\ddd{\mathcal{D}}
\def\hhh{h}
\def\hhf{\mathcal{H}}
\def\cci{\mathcal{C}_{\infty}}
\def\ddi{\mathcal{D}_{\infty}}
\def\hhi{h_{\infty}}
\def\hhfi{\hhf_{\infty}}
\def\bbf{\mathfrak{B}}
\def\bbb{\mathfrak{b}}
\def\bbfi{\mathfrak{B}_{\infty}}
\def\bii{\mathfrak{b}_{\infty}}

\def\zz{\mathbb{Z}}
\def\jj{\mathcal{J}}
\def\qq{\mathscr{Q}}
\def\rr{\mathscr{R}}

\def\ss{\mathscr{S}}

\def\g{{\bf g}}

\begin{document}

\title{Modified Wave Operators for the Wave-Klein-Gordon System}

\author[Z. Ouyang]{Zhimeng Ouyang}
\address[Z. Ouyang]{
   \newline\indent Department of Mathematics, University of Chicago}
\email{ouyangzm9386@uchicago.edu}
\thanks{Z. Ouyang is supported by NSF Grant DMS-2202824.}


\maketitle
\vspace*{-5ex}
\begin{abstract}
We consider a coupled Wave-Klein-Gordon system in 3D, which is a simplified model for the global nonlinear stability of the Minkowski space-time for self-gravitating massive fields. 
In this paper we study the large-time asymptotic behavior of solutions to such systems, and prove modified wave operators for small and smooth data with mild decay at infinity. 
The key novelty comes from a crucial observation that the asymptotic dynamics are dictated by the resonant interactions. 
As a consequence, our main results include the derivation of a resonant system with good error bounds, and a detailed description of the asymptotic dynamics of such quasilinear evolution system of hyperbolic and dispersive type.
\end{abstract}

\pagestyle{myheadings} \thispagestyle{plain} \markboth{Z. OUYANG}{ASYMPTOTICS OF WAVE-KLEIN-GORDON}

\setcounter{tocdepth}{1}
\tableofcontents

\section{Introduction}

\smallskip
\subsection{The Wave-Klein-Gordon System}

In this paper we are interested in the global behavior of the Wave-Klein-Gordon (W-KG) system in \,$3+1$ space-time dimensions:
\begin{equation} \label{Sys:W-KG-1}
\begin{split}
-\square u &\,=\, A^{\alpha\beta}\partial_\alpha v\partial_\beta v + Dv^2 \,,\\
(-\square +1)v &\,=\, uB^{\alpha\beta}\partial_\alpha\partial_\beta v + Euv \,,
\end{split}
\end{equation}
where \,$\square := -\partial_t^{\,2} + \Delta$\, is the d'Alembert operator. 
The unknowns \,$u,v: \r_t^+\times\r_x^3 \rightarrow \r$\, are real-valued functions, 
and $A^{\alpha\beta}$, $B^{\alpha\beta}$, $D$ and $E$ are real constants. 
Without loss of generality we may assume that \,$A^{\alpha\beta}=A^{\beta\alpha}$\, and \,$B^{\alpha\beta}=B^{\beta\alpha}$, \,$\alpha,\beta\!\in\!\{0,1,2,3\}$\,. 
For convenience we will also assume that \,$B^{00}=0$\,.\footnote{\,This can be achieved by adding some higher order terms in the second equation in \eqref{Sys:W-KG-1}, which do not change the analysis.}

Our main focus is the description of asymptotic behavior for small solutions to a prototype of the above system:
\begin{equation} \label{Sys:W-KG-2}
\begin{split}
(\partial_t^{\,2} - \Delta) u &\,=\, |\nabla_{\!t,x}v|^2 + v^2 \,,\\
(\partial_t^{\,2} - \Delta +1)v &\,=\, u \Delta v \,.
\end{split}
\end{equation}
This coupled system consists of a semi-linear wave equation for $u$\, and a quasi-linear Klein-Gordon equation for $v$\,.

\smallskip
\subsection{Background and Motivation}

The system \eqref{Sys:W-KG-1} was derived by LeFloch-Ma
\cite{LeFloch.Ma2016} as a simplified model for the Einstein-Klein-Gordon (E-KG) system, which describes the coupled evolution of the Lorentzian metric and a self-gravitating  massive scalar field. We refer to LeFloch-Ma
\cite{LeFloch.Ma2016, LeFloch.Ma2018}, Ionescu-Pausader \cite{Ionescu.Pausader2019, Ionescu.Pausader2020, Ionescu.Pausader2020(=)} and Wang \cite{Wang2020} for more details.

Consider an Einstein field equation for an unknown space-time $(M,\g)$
\begin{align}
    G_{\alpha\beta}:=R_{\alpha\beta}-\frac{1}{2}R\,\g_{\alpha\beta}=T_{\alpha\beta},
\end{align}
where $G_{\alpha\beta}$ is the Einstein tensor, $\g$ is a Lorentzian metric, $R_{\alpha\beta}$ is the Ricci tensor and $R$ is the scalar curvature, and $T_{\alpha\beta}$ is the energy-momentum tensor of matter in the space. In the presence of massive field $\psi$, $T_{\alpha\beta}$ is given by
\begin{align}
    T_{\alpha\beta}:=\D_{\alpha}\psi\D_{\beta}\psi-\frac{1}{2}\g_{\alpha\beta}\left(\D_{\mu}\psi\D^{\mu}\psi+\psi^2\right),
\end{align}
where $\D$ denotes covariant derivatives. Using Bianchi identities $\D^{\alpha}G_{\alpha\beta}=0$, we may derive the evolution of the massive scalar field $\psi$ as
\begin{align}
    \square_{\g}\psi-\psi=0,
\end{align}
where $\square_{\g}:=\g^{\alpha\beta}\D_{\alpha}\D_{\beta}$.
If we choose the wave coordinates satisfying $-\square_{\g}x^{\alpha}=0\ \ \text{for}\ \ \alpha=0,1,2,3$, then $(\g,\psi)$ satisfies the Einstein-Klein-Gordon system
\begin{align}
    R_{\alpha\beta}-\D_{\alpha}\psi\D_{\beta}\psi&=\frac{\psi^2}{2}\g_{\alpha\beta},\\
    \square_{\g}\psi-\psi&=0.
\end{align}
Further, under the harmonic gauge, the above system reduces to 
\begin{align}
    -\widetilde{\square}_{\g\,}\g_{\alpha\beta}=2\p_{\alpha}\psi\p_{\beta}\psi+\psi^2\g_{\alpha\beta}-F_{\alpha\beta}^{\geq2}(\g,\p \g)&=0,\\
    -\widetilde{\square}_{\g}\psi+\psi&=0,
\end{align}
where $\widetilde{\square}_{\g}:=\g^{\alpha\beta}\p_{\alpha}\p_{\beta}$, and $F_{\alpha\beta}^{\geq2}(\g,\p \g)$ contains all higher-order terms.

The global stability and asymptotic behavior of the Minkowski space-time are central topics in general relativity. The small data global well-posedness, regularity and stability of the Einstein-Klein-Gordon system were studied in LeFloch-Ma
\cite{LeFloch.Ma2018}, Ionescu-Pausader \cite{Ionescu.Pausader2020} and Wang \cite{Wang2020}.

Heuristically, in the Einstein-Klein-Gordon system, if we replace 
the deviation of the Lorentzian metric $\g$ from the Minkowski metric by a scalar function $u$, replace the massive scalar field $\psi$ by $v$, and neglect the higher-order interactions involving $\psi$ (i.e. ignoring $F_{\alpha\beta}^{\geq2}(\g,\p \g)$), we arrive at the Wave-Klein-Gordon system \eqref{Sys:W-KG-1}. 

The small data global well-posedness, regularity and stability of the Wave-Klein-Gordon system were studied in LeFloch-Ma
\cite{LeFloch.Ma2016}, Ionescu-Pausader \cite{Ionescu.Pausader2019}, Dong-Wyatt \cite{Dong.Wyatt2020}. The lower-dimensional results can be found in Ma \cite{Ma2017, Ma2017(=)}, Ifrim-Stingo \cite{Ifrim.Stingo2019} and Stingo \cite{Stingo2018}. Some earlier results on wave and Klein-Gordon type equations can be found in Georgiev \cite{Georgiev1990} and Katayama \cite{Katayama2012}.

Roughly speaking, LeFloch-Ma \cite{LeFloch.Ma2016, LeFloch.Ma2018}, Wang \cite{Wang2020}, and Dong-Wyatt \cite{Dong.Wyatt2020} used a refined hyperbolic foliation method (see also LeFloch-Ma \cite{LeFloch.Ma2014}), which requires restricted initial data (e.g. compactly supported). This restriction can be slightly lifted to treat more general data (see Ma \cite{Ma2020} for a model equation). 

On the other hand, Ionescu-Pausader \cite{Ionescu.Pausader2019, Ionescu.Pausader2020, Ionescu.Pausader2020(=)} relies on the so-called space-time resonance method. It combines the energy method and Fourier analysis to study the global well-posedness of quasi-linear dispersive equations. Generally speaking, energy method combined with vector fields may handle the higher-order regularity terms, and Fourier analysis with normal forms could justify the dispersive decay of lower-regularity terms. This method has been proven to be very effective in many kinds of nonlinear dispersive equations. We refer to Delort-Fang \cite{Delort.Fang2000}, Delort-Fang-Xue \cite{Delort.Fang.Xue2004}, Germain-Masmoudi-Shatah \cite{Germain.Masmoudi.Shatah2009, Germain.Masmoudi.Shatah2012}, Gustafson-Nakanishi-Tsai \cite{Gustafson.Nakanishi.Tsai2009}, Deng \cite{Deng2018}, Deng-Ionescu-Pausader \cite{Deng.Ionescu.Pausader2017}, Hani-Pausader-Tzvetkov-Visciglia \cite{Hani.Pausader.Tzvetkov.Visciglia2015}, Deng-Ionescu-Pausader-Pusateri \cite{Deng.Ionescu.Pausader.Pusateri2017}, Guo-Ionescu-Pausader \cite{Guo.Ionescu.Pausader2016}, Guo-Pausader \cite{Guo.Pausader2011}, Ionescu-Pausader \cite{Ionescu.Pausader2013, Ionescu.Pausader2014}, Deng-Pusateri \cite{Deng.Pusateri2020}, and Ionescu-Pusateri \cite{Ionescu.Pusateri2015}. In this paper, we will follow this path to study the long-time behavior of the Wave-Klein-Gordon system.

\smallskip
\subsection{Main Result}

The Wave-Klein-Gordon system does not have a null structure, but it presents an intricate resonant pattern, which is reflected in the asymptotic behavior. Ionescu-Pausader \cite[Remark 1.3]{Ionescu.Pausader2019} pointed out 
the modification to linear scattering for the Wave-Klein-Gordon system.
We interpret their result as follows.

\begin{theorem}[Modified Scattering]\label{modified scattering}
Let $(u,v)$ be a global solution of the system \eqref{Sys:W-KG-2}.
Then there exists a scattering profile $\big(V^{wa}_{\infty},V^{kg}_{\infty}\big)$ such that
\begin{align}
    \lim_{t\rt\infty}\nm{\big(\dt-\ui\abs{\nabla}\big)u(t)-\ue^{-\ui t\abs{\nabla}}V^{wa}_{\infty}}_{L^2}=0,
\end{align}
and
\begin{align}
    \lim_{t\rt\infty}\nm{\big(\dt-\ui\br{\nabla}\big)v(t)- \ue^{-\ui t\br{\nabla}+\ui \Theta(t,\nabla)} V^{kg}_{\infty}}_{L^2}=0,
\end{align}
where $\Theta(t,\nabla)$ is a phase-correction operator depending on $u$. 
\end{theorem}

\begin{remark}
This theorem indicates that the wave component $u$ scatters linearly and the Klein-Gordon component $v$ undergoes nonlinear scattering. 
\end{remark}

In this paper, we will focus on the converse statement -- existence of wave operators, namely that every possible asymptotic behavior is achieved. 
More precisely,
we intend to justify that 
any asymptotic dynamic arises in a unique way as the limit of a solution to the W-KG system.

\begin{theorem}[Main Theorem - Modified Wave Operator, informal version]\label{wave operator formal}
Assume the scattering data $\big(V^{wa}_{\infty}(x),V^{kg}_{\infty}(x)\big)$ is sufficiently small under suitable weighted Sobolev-type norms.
Then there exists a unique global solution $(u,v)$ of the system \eqref{Sys:W-KG-2} such that 
\begin{align}
    \lim_{t\rt\infty}\nm{\big(\dt-\ui\abs{\nabla}\big)u(t)-\ue^{-\ui t\abs{\nabla}}\left(V^{wa}_{\infty}+\check{\hhf}_{\infty}(t)\right)}_{L^2}=0,
\end{align}
and
\begin{align}
    \lim_{t\rt\infty}\nm{\big(\dt-\ui\br{\nabla}\big)v(t)- \ue^{-\ui t\br{\nabla}+\ui \ddd_{\infty}(t,\nabla)} V^{kg}_{\infty}}_{L^2}=0.
\end{align}
Here $\check{\hhf}_{\infty}(t,x)$ is a function depending on $V^{kg}_{\infty}$ and $\ddd_{\infty}(t,\nabla)$ is a phase-correction operator that explicitly depends on $(V^{wa}_{\infty},V^{kg}_{\infty})$. 
\end{theorem}

\begin{remark}
We would like to make a few remarks:
\begin{enumerate}
\item[(1)] 
This theorem is just an informal version and the precise statement of the theorem (Theorem \ref{main theorem}) is in Section\;\ref{Sec:Main-Thm}. $\check{\hhf}_{\infty}$ is the inverse Fourier transform of $\hhfi$ defined in \eqref{H_infty} and $\ddd_{\infty}(t,\nabla)$ is the operator corresponding to a modification $\ddi(t,\xi)$ of the free Klein-Gordon dispersion (see \eqref{D_infty} for the definition). 
\item[(2)] 
In the Fourier space, $\ddi(t,\xi)$ is a real-valued nonlinear function of $\widehat{V^{kg}_\infty}$ which takes into account the low-frequency part of $\widehat{V^{wa}_\infty}$.
In particular, the fact that $\abs{\ddi(t,\xi)}\rightarrow\infty$ as $t\rightarrow\infty$ indicates a genuine nonlinear scattering for the KG component.
\item[(3)] 
Compared with Theorem \ref{modified scattering}, we may notice that Theorem \ref{wave operator formal} has an extra term $\check{\hhf}_{\infty}$ (which is $\sim O(1)$ as $t\rightarrow\infty$). This term actually provides a more precise description of the asymptotic behavior for the Wave-Klein-Gordon system. Based on Lemma \ref{vtt lemma 10}, $\hhfi$ does not decay in time. However, Theorem \ref{wave operator formal} shows that the difference between the resonant system and the W-KG system decays to zero as $t\rt\infty$. Hence, $\hhfi$ indeed captures the asymptotic behavior of the W-KG system. 
\item[(4)]
Our result also applies to the original W-KG system \eqref{Sys:W-KG-1}.
For simplicity we will work with the model system \eqref{Sys:W-KG-2}, but all our arguments can be carried out for the general case \eqref{Sys:W-KG-1} as well.
\end{enumerate}
\end{remark}

\smallskip
\subsection{Difficulties and Methods}

The energy structure for the Wave-Klein-Gordon system can be explained heuristically as follows (see Ionescu-Pausader \cite[Section 1.3.1]{Ionescu.Pausader2019}): 
The structure of the W-KG system is characterized by the bilinear interactions 
\begin{align*}
    \rm{Wave} &\,\leftarrow\, \rm{KG} \times \rm{KG} \\
    \rm{KG} &\,\leftarrow\, \rm{Wave} \times \rm{KG}
\end{align*}
    Let $\lll$ be a commutative vector field and $\p$ be either time or space derivative. Then applying $\lll$ to the Wave-Klein-Gordon system \eqref{Sys:W-KG-2} yields
    \begin{align}
        \square\big(\lll[u]\big)&\simeq \p v\cdot\lll[\p v],\\
        (\square+1)\big(\lll[v]\big)&\simeq \lll[u]\cdot \p^2v+u\cdot\lll[\p^2v].
    \end{align}
    It is clear that the order of derivatives for the nonlinear terms are imbalanced. Since the standard energy estimate may only control
    \begin{align}
        \tnm{\lll[\p u]}\ \ \text{and}\ \ \nm{\lll[v]}_{H^1},
    \end{align}
    we have difficulty in bounding $\lll[u]\cdot \p^2v$ (in the sense that $\lll[u]$ is not included in the energy and $\p^2v$ cannot be handled via integration by parts). Notice that $u\cdot\lll[\p^2v]$ may be handled using integration by parts in the energy estimate, so we will not discuss it in the following.
    To overcome the above difficulty, Ionescu-Pausader \cite{Ionescu.Pausader2019} introduced $\abs{\nabla}^{-\frac{1}{2}}$ term in the energy structure of the wave equation. 
    
    Ideally, the wave component $\nm{u}_{W^{k,\infty}}\ls \br{t}^{-1}$ and the Klein-Gordon component $\nm{v}_{W^{k,\infty}}\ls \br{t}^{-\frac{3}{2}}$ based on the standard dispersive estimates.
    For the wave equation, energy estimate implies
    \begin{align}
        \dt\btnm{\abs{\nabla}^{-\frac{1}{2}}\lll[\p u]}&\ls \btnm{\abs{\nabla}^{-\frac{1}{2}}\big(\p v\cdot\lll[\p v]\big)}
        \ls \br{t}^{\frac{1}{2}}\lnm{\p v}\tnm{\lll[\p v]}\ls \br{t}^{-1}\tnm{\lll[\p v]},
    \end{align}
    for the frequency $\abs{\xi}\gs\br{t}^{-1}$. On the other hand, for the Klein-Gordon equation, energy estimate implies
    \begin{align}
        \dt\nm{\lll[v]}_{H^1}&\ls \tnm{\lll[u]\cdot\p^2v}\ls \btnm{\abs{\nabla}^{-\frac{1}{2}}\big(\abs{\nabla}^{-\frac{1}{2}}\lll[\p u]\big)\cdot\p^2v}\\
        &\ls \br{t}^{\frac{1}{2}}\lnm{\p^2 v}\btnm{\abs{\nabla}^{-\frac{1}{2}}\lll[\p u]}\ls \br{t}^{-1}\btnm{\abs{\nabla}^{-\frac{1}{2}}\lll[\p u]},\no
    \end{align}
    for the frequency $\abs{\xi}\gs\br{t}^{-1}$. Now if we can control the energy 
    \begin{align}
        \btnm{\abs{\nabla}^{-\frac{1}{2}}\lll[\p u]}\ \ \text{and}\ \ \nm{\lll[v]}_{H^1},
    \end{align}
    we may close the bootstrapping.
    
    In the above analysis, we leave out the part with frequency $\abs{\xi}\ls\br{t}^{-1}$. This part cannot be handled by the energy structure with dispersive estimate, and it will not decay to zero as $t\rt\infty$. In other words, in the long run, only the part with $\abs{\xi}\ls\br{t}^{-1}$ will play a role in the asymptotic behavior. In the Wave-Klein-Gordon system, we will extract this part and call it a resonant system for $(u^{RS},v^{RS})$.
    
    For the modified wave operator, we assume that the initial data for the resonant system has been prescribed. We intend to construct a solution $(u,v)$ to the Wave-Klein-Gordon system \eqref{Sys:W-KG-2} such that $u\rt u^{RS}$ and $v\rt v^{RS}$ as $t\rt\infty$.
    Denote $g^{wa}=u-u^{RS}$ and $g^{kg}=v-v^{RS}$. We insert it into the Wave-Klein-Gordon system \eqref{Sys:W-KG-2} to get an equation for $(g^{wa},g^{kg})$. This equation has similar structure as \eqref{Sys:W-KG-2} with some extra terms: 
    \begin{equation} \label{Sys:W-KG-2'}
    \begin{split}
   \square g^{wa} &=-\square u^{RS}+\abs{\p g^{wa}+\p u^{RS}}^2 + \abs{g^{wa}+u^{RS}}^2 \,,\\
    (\square+1) g^{kg}&\,=\, -(\square+1) v^{RS}+\big(g^{wa}+u^{RS}\big) \Delta\big(g^{kg}+v^{RS}\big) \,.
    \end{split}
    \end{equation}
    We will design a fixed-point argument such that $(g^{wa},g^{kg})$ is well-posed for $t\in[T,\infty)$ and satisfies 
    \begin{align}
        \lim_{t\rt\infty}g^{wa}(t)=\lim_{t\rt\infty}g^{kg}(t)=0.
    \end{align}
    The nonlinear terms can be further divided into three categories: quadratic terms (two components are both $g^{wa}$ or $g^{kg}$), linear terms (one component is $g^{wa}$ or $g^{kg}$ and the other one is $u^{RS}$ or $v^{RS}$), and forcing terms (two components are both $u^{RS}$ or $v^{RS}$). In particular, the most difficult estimate lies in the forcing terms, and we have to delicately utilize the cancellation with the resonant system.
    
    Compared with Ionescu-Pausader \cite{Ionescu.Pausader2019}, where they allow the solutions to grow mildly in energy norms as $t\rt\infty$ (i.e. nonlinear estimate is like $\br{t}^{-1+\delta}$), our case is more challenging, in the sense that we must prove $(g^{wa},g^{kg})$ decays as $t\rt\infty$ (i.e. nonlinear estimate is like $\br{t}^{-1-\delta}$). The faster decay rate is extremely crucial for our proof. This is particularly difficult for the forcing terms since the resonant system at most decays at $\br{t}^{-1}$, so we must delicately analyze the difference between the forcing terms and the resonant part $-\square u^{RS}$ and $-(\square+1) v^{RS}$ in \eqref{Sys:W-KG-2'}, and eliminate the low-frequency contributions. 

To summarize, the key to our proof is the energy estimate and analysis of the three types of nonlinear terms. The most difficult terms are forcing terms related to the resonant system. 
In particular, we focus on controlling the perturbation globally using a combination of tools from Fourier and harmonic analysis (Fourier transform, dyadic decomposition), oscillatory integrals (non-stationary/stationary phase analysis), bilinear estimates (pseudo-products operators with singular multipliers), ODE and geometric methods (normal forms, vector fields), and delicately designed function spaces and decay rate.

\bigskip
\section{Setup of the Problem} \label{Sec:Setup}

\smallskip
\subsection{Reformulation of the Equations: Duhamel's Formula} \label{SubSec: Duhamel}

Throughout this paper the Fourier transform is defined as 
\begin{equation}
\widehat{f}(\xi) = \mathscr{F}[f](\xi) = (2\pi)^{-\frac{3}{2}}\! \int_{\r^3} \ue^{-\ui x\cdot\xi} f(x) \,\ud x \,.    
\end{equation}

Define the operators on $\mathbb{R}^3$ 
\begin{equation} \label{propagators}
\Lambda_{wa}:=|\nabla| = \mathscr{F}^{-1}|\xi|\mathscr{F} \,,
\qquad\, \Lambda_{kg}:=\langle\nabla\rangle=\sqrt{|\nabla|^2+1} 
= \mathscr{F}^{-1}\langle\xi\rangle\mathscr{F} \,,
\end{equation}
which relate to the dispersion relations for the wave and Klein-Gordon equation, respectively. 

We denote also the corresponding multipliers in frequency space by
\begin{equation} \label{dispersion}
\begin{array}{ll}
\Lambda_{wa,+}(\xi)=\Lambda_{wa}(\xi):=|\xi| \,, 
&\qquad \Lambda_{kg,+}(\xi)=\Lambda_{kg}(\xi):=\langle\xi\rangle=\sqrt{|\xi|^2+1} \,,\\[5pt]
\Lambda_{wa,-}(\xi)=-\Lambda_{wa,+}(\xi) \,, 
&\qquad \Lambda_{kg,-}(\xi)=-\Lambda_{kg,+}(\xi) \,.\\[5pt]
\end{array}
\end{equation}

The (complex-valued) {\it{normalized solutions}} \,$U^{wa}, U^{kg}$ of the system \eqref{Sys:W-KG-2} are defined by
\begin{equation} \label{normalized-sol}
U^{wa}(t):=[\partial_t - \ui\Lambda_{wa}]\, u(t) \,,\qquad\, U^{kg}(t):=[\partial_t - \ui\Lambda_{kg}]\, v(t) \,,
\end{equation}
along with
\begin{equation} \label{normalized-sol-pm}
\begin{array}{ll}
U^{wa,+}:=U^{wa} \,, 
&\qquad U^{kg,+}:=U^{kg} \,,\\[5pt]
U^{wa,-}:=\overline{U^{wa}}=[\partial_t + \ui\Lambda_{wa}]\, u(t) \,, 
&\qquad U^{kg,-}:=\overline{U^{kg}}=[\partial_t + \ui\Lambda_{kg}]\, v(t) \,,\\[5pt]
\end{array}
\end{equation}
so that the solutions $u,v$ can be recovered from $U^{wa}, U^{kg}$\, by the formulas
\begin{equation} \label{recover-sol}
\begin{array}{ll}
\partial_t u = \frac{1}{2}\left(U^{wa}+\overline{U^{wa}}\right) = \Re\left(U^{wa}\right) \,, 
&\quad \partial_t v = \frac{1}{2}\big(U^{kg}+\overline{U^{kg}}\big) = \Re\left(U^{kg}\right) \,,\\[5pt]
\Lambda_{wa} u = \frac{\ui}{2}\left(U^{wa}-\overline{U^{wa}}\right) = -\Im\left(U^{wa}\right) \,, 
&\quad \Lambda_{kg} v = \frac{\ui}{2}\big(U^{kg}-\overline{U^{kg}}\big) = -\Im\left(U^{kg}\right) \,,\\[5pt]
u = \frac{\ui}{2\Lambda_{wa}}\left(U^{wa}-\overline{U^{wa}}\right) = -\frac{1}{|\nabla|}\,\Im\left(U^{wa}\right) \,, 
&\quad v = \frac{\ui}{2\Lambda_{kg}}\big(U^{kg}-\overline{U^{kg}}\big) = -\frac{1}{\langle\nabla\rangle}\,\Im\left(U^{kg}\right) \,.\\[5pt]
\end{array}
\end{equation}

Conjugating with the linear propagator operator 
\,$\ue^{\ui t\Lambda} := \mathscr{F}^{-1}e^{\ui t\Lambda(\xi)}\mathscr{F}$\,, 
we define the associated {\it{profiles}} \,$V^{wa}, V^{kg}$\, by
\begin{equation} \label{profile}
V^{wa}(t):=\ue^{\ui t\Lambda_{wa}}U^{wa}(t) \,,\qquad\qquad\, V^{kg}(t):=\ue^{\ui t\Lambda_{kg}}U^{kg}(t) \,,
\end{equation}
and denote also
\begin{equation} \label{profile-pm}
\begin{array}{ll}
V^{wa,+}:=V^{wa} = \ue^{\ui t|\nabla|}\,U^{wa}(t) \,, 
&\qquad V^{kg,+}:=V^{kg} = \ue^{\ui t\langle\nabla\rangle}\,U^{kg}(t) \,,\\[5pt]
V^{wa,-}:=\overline{V^{wa}} = \ue^{-\ui t|\nabla|}\,\overline{U^{wa}}(t) \,, 
&\qquad V^{kg,-}:=\overline{V^{kg}} = \ue^{-\ui t\langle\nabla\rangle}\,\overline{U^{kg}}(t) \,,\\[5pt]
\end{array}
\end{equation}
which under the Fourier transform become
\begin{equation} \label{profile-pm-freq}
\begin{array}{ll}
\widehat{V^{wa,+}}(t,\xi) = \ue^{\ui t|\xi|}\,\widehat{U^{wa}}(t,\xi) \,, 
&\qquad \widehat{V^{kg,+}}(t,\xi) = \ue^{\ui t\langle\xi\rangle}\,\widehat{U^{kg}}(t,\xi) \,,\\[5pt]
\widehat{V^{wa,-}}(t,\xi) = \ue^{-\ui t|\xi|}\,\widehat{\overline{U^{wa}}}(t,\xi) \,, 
&\qquad \widehat{V^{kg,-}}(t,\xi) = \ue^{-\ui t\langle\xi\rangle}\,\widehat{\overline{U^{kg}}}(t,\xi) \,.\\[5pt]
\end{array}
\end{equation}

\begin{remark}
From the above definition we can see that 
\begin{align}
    \widehat{V^{wa,-}}(t,\xi) = \widehat{\overline{V^{wa,+}}}(t,\xi) = \overline{\widehat{V^{wa,+}}(t,-\xi)},\quad \widehat{V^{kg,-}}(t,\xi) =  \widehat{\overline{V^{kg,+}}}(t,\xi) = \overline{\widehat{V^{kg,+}}(t,-\xi)}.
\end{align}
\end{remark}

\smallskip
Now we reformulate the system \eqref{Sys:W-KG-2} to derive the Duhamel's formulas for \,$\widehat{V^{wa}}$\, and \,$\widehat{V^{kg}}$. 
Observing that both equations in \eqref{Sys:W-KG-2} are of hyperbolic and dispersive type with quadratic nonlinearities, we first write them in terms of the normalized variables $U^{wa}, U^{kg}$ as quadratic dispersive equations of the form
\begin{equation} \label{Sys:normalized}
\begin{split}
\left(\partial_t + \ui\Lambda_{wa}\right) U^{wa} &= \mathcal{N}^{wa} = \mathcal{Q}^{wa}(v,v) := |\nabla_{\!t,x}v|^2 + v^2 \,,\\
\left(\partial_t + \ui\Lambda_{kg}\right) U^{kg} &= \mathcal{N}^{kg} = \mathcal{Q}^{kg}(u,v) := u \Delta v \,,
\end{split}
\end{equation}
where the (quadratic) nonlinearities can be expressed as pseudo-product operators
\begin{equation} \label{nonlinearities}
\begin{split}
\mathcal{N}^{wa} &= \mathcal{T}_{a(\xi,\eta)}^{\,wa}(v,v) = \sum_{\pm,\pm}\mathcal{T}_{a_{\pm\pm}}^{\,wa}(U^{kg,\pm},U^{kg,\pm}) \,,\\
\mathcal{N}^{kg} &= \mathcal{T}_{b(\xi,\eta)}^{\,kg}(u,v) = \sum_{\pm,\pm}\mathcal{T}_{b_{\pm\pm}}^{\,kg}(U^{kg,\pm},U^{wa,\pm}) 
\end{split}
\end{equation}
of the general form
\begin{equation*}
\mathcal{T}_{m(\xi,\eta)}(f,g) := \mathscr{F}^{-1}\widehat{\mathcal{Q}(f,g)} = \mathscr{F}^{-1}\!\int\! m(\xi,\eta)\widehat{f}(\xi\!-\!\eta)\widehat{g}(\eta)\,\ud\eta
\end{equation*}
with the symbol-type multiplier $m(\xi,\eta)$ carrying information of the quadratic interaction (e.g. dispersion/wave-type, derivatives).

Since the normalized solutions $U^{wa}, U^{kg}$ display oscillations itself, and we want to isolate all the oscillations in a unique factor $\ue^{\ui t\Phi(\xi,\eta)}$, we then need to introduce the profiles $V^{wa}, V^{kg}$ as in (\ref{profile}) so that \eqref{Sys:normalized} can be rewritten in terms of the new unknowns as
\begin{align} 
\partial_t V^{wa} &= \ue^{\ui t\Lambda_{wa}} \mathcal{N}^{wa} := \ue^{\ui t\Lambda_{wa}} \sum_{\pm,\pm}\mathcal{T}_{a_{\pm\pm}}^{\,wa}\big(\ue^{\mp \ui t\Lambda_{kg}} V^{kg,\pm},\ue^{\mp \ui t\Lambda_{kg}} V^{kg,\pm}\big) \,,\label{Sys:profile-wa}\\
\partial_t V^{kg} &= \ue^{\ui t\Lambda_{kg}} \mathcal{N}^{kg} := \ue^{\ui t\Lambda_{kg}} \sum_{\pm,\pm}\mathcal{T}_{b_{\pm\pm}}^{\,kg}\big(\ue^{\mp \ui t\Lambda_{kg}} V^{kg,\pm},\ue^{\mp \ui t\Lambda_{wa}} V^{wa,\pm}\big) \,,\label{Sys:profile-kg}
\end{align}
from which we see that \,$V^{wa}, V^{kg}$ evolve purely nonlinearly (which is more stable).

Taking the Fourier transform (combined with its differentiation and product-convolution properties) and using the formulas (\ref{recover-sol}) and identities (\ref{profile-pm-freq}), we obtain
\begin{align} 
\partial_t \widehat{V^{wa}}(t,\xi) = \ue^{\ui t\Lambda_{wa}(\xi)} \widehat{\mathcal{N}^{wa}} 
\,&:= \sum_{\iota_1,\iota_2\in\{+,-\}} \text{\Large$\mathbf{I}$}_{\,wa}^{\iota_1\iota_2}\big[V^{kg,\iota_1},V^{kg,\iota_2}\big](t,\xi) \,,\label{Duhamel-diff-wa}\\
\partial_t \widehat{V^{kg}}(t,\xi) = \ue^{\ui t\Lambda_{kg}(\xi)} \widehat{\mathcal{N}^{kg}}
\,&:= \sum_{\iota_1,\iota_2\in\{+,-\}} \text{\Large$\mathbf{I}$}_{\,kg}^{\iota_1\iota_2}\big[V^{kg,\iota_1},V^{wa,\iota_2}\big](t,\xi) \,,\label{Duhamel-diff-kg}
\end{align}
where
\begin{equation} \label{Duhamel-wa}
\begin{split}
\text{\Large$\mathbf{I}$}_{\,wa}^{\iota_1\iota_2}\big[F,G\big](t,\xi) 
&:=\, \frac{1}{4}(2\pi)^{-\frac{3}{2}} \int_{\mathbb{R}^3} \text{\large$\ue$}^{\ui t\Phi_{\,wa}^{\iota_1\iota_2}(\xi,\eta)} \text{\large$a$}_{\iota_1\iota_2}(\xi,\eta) \widehat{F}(t,\xi\!-\!\eta)\widehat{G}(t,\eta)\,\dd\eta \,,\\[8pt]
\text{\large$\Phi$}_{\,wa}^{\iota_1\iota_2}(\xi,\eta) 
&:=\, \Lambda_{wa}(\xi)-\Lambda_{kg,\iota_1}(\xi\!-\!\eta)-\Lambda_{kg,\iota_2}(\eta) \;=\, |\xi| - \iota_1\langle\xi\!-\!\eta\rangle - \iota_2\langle\eta\rangle \,,\\[2pt]
\text{\large$a$}_{\iota_1\iota_2}(\xi,\eta) 
&:=\, 1+\frac{\iota_1\iota_2}{\Lambda_{kg}(\xi\!-\!\eta)\Lambda_{kg}(\eta)} \big[(\xi\!-\!\eta)\cdot\eta-1\big] 
\,=\, 1+\iota_1\iota_2\,\frac{(\xi\!-\!\eta)\cdot\eta-1}{\langle\xi\!-\!\eta\rangle\langle\eta\rangle} \,,
\end{split}
\end{equation}
and
\begin{equation} \label{Duhamel-kg}
\begin{split}
\text{\Large$\mathbf{I}$}_{\,kg}^{\iota_1\iota_2}\big[F,G\big](t,\xi) 
&:=\, \frac{1}{4}(2\pi)^{-\frac{3}{2}} \int_{\mathbb{R}^3} \text{\large$\ue$}^{\ui t\Phi_{\,kg}^{\iota_1\iota_2}(\xi,\eta)} \text{\large$b$}_{\iota_1\iota_2}(\xi,\eta) \widehat{F}(t,\xi\!-\!\eta)\widehat{G}(t,\eta)\,\dd\eta \,,\\[8pt]
\text{\large$\Phi$}_{\,kg}^{\iota_1\iota_2}(\xi,\eta) 
&:=\, \Lambda_{kg}(\xi)-\Lambda_{kg,\iota_1}(\xi\!-\!\eta)-\Lambda_{wa,\iota_2}(\eta) \;=\, \langle\xi\rangle - \iota_1\langle\xi\!-\!\eta\rangle - \iota_2|\eta| \,,\\[2pt]
\text{\large$b$}_{\iota_1\iota_2}(\xi,\eta) 
&:=\, \iota_1\iota_2\,\frac{|\xi\!-\!\eta|^2}{\Lambda_{kg}(\xi\!-\!\eta)\Lambda_{wa}(\eta)} 
\,=\, \iota_1\iota_2\,\frac{|\xi\!-\!\eta|^2}{\langle\xi\!-\!\eta\rangle|\eta|} \,. \\[5pt]
\end{split}
\end{equation}
In later proofs we will also use the notation ${\bbi}[\hat F,\hat G]$ for the same expression as $\bbi[F,G]$ when we would like to highlight the quantities in the phase space.

Our analysis is largely based on the nonlinear phases \,$\Phi^{\iota_1\iota_2}$\,, which measure the quadratic interactions between different wave-types.

\smallskip
\subsection{Vector Fields}

The major component of our analysis relies on the energy estimates for the system \eqref{Sys:W-KG-2}.
These energy estimates involve vector fields, which correspond to the symmetries of the linearized equations.

We denote by $\p_0:=\p_t$ and $\p_i:=\p_{x_i}$ for $i=1,2,3$. 
Define the Lorentz vector fields $\Gamma_j$ and the rotation vector fields $\Omega_{jk}$,
\begin{align}
    \Gamma_j:=x_j\dt+t\p_j,\qquad \Omega_{jk}:=x_j\p_k-x_k\p_j,
\end{align}
for $j,k=1,2,3$.
These vector fields commute with both the wave operator $-\square$ and the Klein-Gordon operator $-\square+1$. For any multi-index $a=(a_1,a_2,a_3)\in\mathbb{N}^3$, $b=(b_1,b_2,b_3)\in\mathbb{N}^3$,
with $|a|=\sum_{j=1}^3 a_j$, $|b|=\sum_{j=1}^3 b_j$, 
we define
\begin{align}
    \Gamma^{a}:=\Gamma_1^{a_1}\Gamma_2^{a_2}\Gamma_3^{a_3},\qquad
    \Omega^{b}:=\Omega_{23}^{b_1}\Omega_{31}^{b_2}\Omega_{12}^{b_3}.
\end{align}
For any $n\in\mathbb{N}$, define $\vvv_n$ as the set of differential operators of the form 
\begin{align}
    \vvv_n:=\big\{\lll=\Gamma^{a}\Omega^{b}
    :|a|+|b|
    \leq n\big\}\,.
\end{align}

Define $U^{wa}_{\lll}$ and $U^{kg}_{\lll}$ in a similar manner as $U^{wa}$ and $U^{kg}$, with $\lll[u]$ instead of $u$ and $\lll[v]$ instead of $v$. The same notation applies to all the other variables, like $V^{wa}_{\lll}$ and $V^{kg}_{\lll}$:
\begin{equation}
\begin{split}
&U^{wa}_{\mathcal{L}}(t):=(\partial_t-\ui\Lambda_{wa}) ({\mathcal{L}}[u])(t),\qquad\, U^{kg}_{\mathcal{L}}(t):=(\partial_t-\ui\Lambda_{kg}) ({\mathcal{L}}[v])(t),\\
&V^{wa}_{\mathcal{L}}(t):=\ue^{\ui t\Lambda_{wa}}U^{wa}_{\mathcal{L}}(t),\qquad\qquad\,\,\,\,\, V^{kg}_{\mathcal{L}}(t):=\ue^{\ui t\Lambda_{kg}}U^{kg}_{\mathcal{L}}(t).
\end{split}
\end{equation}
Similarly, the functions $\lll[u],\lll[v]$ can be recovered from the normalized variables $U^{wa}_{\lll}, U^{kg}_{\lll}$\, by the formulas
\begin{equation} 
\begin{array}{ll}
\partial_t (\lll[u]) = \frac{1}{2}\left(U^{wa}_{\lll}+\overline{U^{wa}_{\lll}}\,\right) = \Re\left(U^{wa}_{\lll}\right) \,, 
&\quad \partial_t (\lll[v]) = \frac{1}{2}\big(U^{kg}_{\lll}+\overline{U^{kg}_{\lll}}\,\big) = \Re\big(U^{kg}_{\lll}\big) \,,\\[5pt]
\Lambda_{wa} (\lll[u]) = \frac{\ui}{2}\left(U^{wa}_{\lll}-\overline{U^{wa}_{\lll}}\,\right) = -\Im\left(U^{wa}_{\lll}\right) \,, 
&\quad \Lambda_{kg} (\lll[v]) = \frac{\ui}{2}\big(U^{kg}_{\lll}-\overline{U^{kg}_{\lll}}\,\big) = -\Im\big(U^{kg}_{\lll}\big) \,,\\[5pt]
\lll[u] = \frac{\ui}{2\Lambda_{wa}}\left(U^{wa}_{\lll}-\overline{U^{wa}_{\lll}}\,\right) = -\frac{1}{|\nabla|}\,\Im\left(U^{wa}_{\lll}\right) \,, 
&\quad \lll[v] = \frac{\ui}{2\Lambda_{kg}}\big(U^{kg}_{\lll}-\overline{U^{kg}_{\lll}}\,\big) = -\frac{1}{\langle\nabla\rangle}\,\Im\big(U^{kg}_{\lll}\big) \,.\\[5pt]
\end{array}
\end{equation}
With the definitions above,
the system \eqref{Sys:normalized} yields, for any $\mathcal{L}\in\mathcal{V}_{n}$,
\begin{equation} \label{Sys:W-KG-3}
\begin{split}
(\partial_t+\ui\Lambda_{wa})U^{wa}_{\mathcal{L}}&=\mathcal{N}^{wa}_{\mathcal{L}}:=\mathcal{L}\big[|\nabla_{\!t,x}v|^2 + v^2\big],\\
(\partial_t+\ui\Lambda_{kg})U^{kg}_{\mathcal{L}}&=\mathcal{N}^{kg}_{\mathcal{L}}:=\mathcal{L}\big[u \Delta v\big].
\end{split}
\end{equation}

For any smooth function $f(x)$, we define $f_{\lll}(x)$ when the vector field $\lll$ is applied to $f$. Note that in our formulation, all the quantities can be categorized into three levels: $(u,v)$ level; $(\uwa,\ukg)$ level; $(\vwa,\vkg)$ level. $f$ is typically a function at $(\vwa,\vkg)$ level, but $\lll$ can only be directly applied at $(u,v)$ level, so $f_{\lll}\neq \lll[f]$. Instead, we have to go back and forth between the three levels. 
For example, 
\begin{align}
    \vwa_{\lll}&\,=\ue^{\ui t\lawa}\big(\dt-\ui\lawa\big)\lll\left[\frac{\ui}{2\lawa} \Big(\ue^{-\ui t\lawa}\vwa-\ue^{\ui t\lawa}\overline \vwa\Big) \right],\label{vector field 1} \\
    \vkg_{\lll}&\,=\ue^{\ui t\lakg}\big(\dt-\ui\lakg\big)\lll\left[\frac{\ui}{2\lakg} \Big(\ue^{-\ui t\lakg}\vkg-\ue^{\ui t\lakg}\overline \vkg\Big) \right].\label{vector field 2}
\end{align}

\smallskip
\subsection{Resonant System} \label{Sec:Thm1-Pf}

We now discuss the main intuition behind the derivation of the ``resonant system'' which drives the asymptotic behavior of our original system.
As explained in Section\;\ref{SubSec: Duhamel}, we will focus on the evolution of $\widehat{V^{wa}}$ and $\widehat{V^{kg}}$ by looking at the Duhamel's formulas \eqref{Duhamel-diff-wa} and \eqref{Duhamel-diff-kg}.

\smallskip
\subsubsection{Quadratic Phases}

In \eqref{Duhamel-diff-wa} and \eqref{Duhamel-diff-kg}, the behavior of the oscillatory integral is largely characterized by the phase functions $\Phi^{\iota_1\iota_2}_{wa}$ and $\Phi^{\iota_1\iota_2}_{kg}$,
which have the general form
\begin{align}
    &\Phi_{\sigma\mu\nu}:\r^3\times\r^3\rt\r,\\
    &\Phi_{\sigma\mu\nu}(\xi,\eta):=\la_{\sigma}(\xi)-\la_{\mu}(\xi-\eta)-\la_{\nu}(\eta)\no
\end{align}
for 
\begin{align}
    \sigma,\mu,\nu\in\mathcal{P}:=\big\{(wa,+),(wa,-),(kg,+),(kg,-)\big\}.
\end{align}
In fact, our resonant/stationary-phase analysis is based on essentially only one type of quadratic phase functions
\begin{align}
    \Phi(\xi_1,\xi_2)&=\la_{kg}(\xi_1)\pm\la_{kg}(\xi_2)\pm\la_{wa}(\xi_1+\xi_2)=\br{\xi_1}\pm\br{\xi_2}\pm\abs{\xi_1+\xi_2},\no
\end{align}
which has the lower bound
\begin{align}\label{est-phi}
    \big|\Phi(\xi_1,\xi_2)\big|\gs\frac{\abs{\xi_1+\xi_2}}{(1+\abs{\xi_1}+\abs{\xi_2})^2}.
\end{align}
Based on \eqref{Duhamel-wa} and \eqref{Duhamel-kg}, if $\abs{\xi},\abs{\xi-\eta},\abs{\eta}\in[0,b]$ for $b\geq1$, then
\begin{align}\label{phase-elliptic}
    \big|\text{\large$\Phi$}_{\,wa}^{\iota_1\iota_2}(\xi,\eta)\big|\geq \frac{\abs{\xi}}{4b^2},\qquad
    \big|\text{\large$\Phi$}_{\,kg}^{\iota_1\iota_2}(\xi,\eta)\big|\geq \frac{\abs{\eta}}{4b^2}.
\end{align}
Therefore, we expect that the interactions where the wave component has very small frequencies, in particular when $t\abs{\xi}\ls 1$ or $t\abs{\eta}\ls 1$, will play an important role in the analysis.

\smallskip
\subsubsection{Analysis of the Resonances} \label{SubSec:ST-Resonances}

Recall that \eqref{Duhamel-diff-wa} and \eqref{Duhamel-diff-kg} yield Duhamel's formulas of the form
\begin{align} \label{oscillatory-integral}
\widehat{V}(t,\xi) \,=\, \widehat{V_\infty}(\xi)
    -\!\int_t^\infty\!\! \int_{\mathbb{R}^3} \text{\large$\ue$}^{\ui s\Phi(\xi,\eta)} m(\xi,\eta) \widehat{V}(s,\xi\!-\!\eta)\widehat{V}(s,\eta)\,\dd\eta\dd s\,.
\end{align}
In order to prove long-range scattering, it is desirable to control the above oscillatory integral uniformly in time.
From the viewpoint of ``stationary phase'', we may want to take advantage of the oscillating factor $\ue^{\ui s\Phi(\xi,\eta)}$.
This can create rapid decay as time becomes large (and thus yields smallness),
except on the sets where the phase is stationary in either of the integration variables:

\begin{itemize}
    \item 
    Time resonant set \,$\mathcal{T}:=\big\{(\xi,\eta)\in\r^3\times\r^3: \Phi(\xi,\eta)=0\big\}$ \,--\, stationary over time $s$.
    \vspace{2pt}
    \item 
    Space resonant set \,$\mathcal{S}:=\big\{(\xi,\eta)\in\r^3\times\r^3: \nabla_{\!\eta}\Phi(\xi,\eta)=0\big\}$ \,--\, stationary in $\eta$.
    \vspace{2pt}
    \item 
    Space-time resonant set \,$\mathcal{R}:=\mathcal{T}\cap\mathcal{S}$ \,--\, stationary in both $s$ and $\eta$.
\end{itemize}
Therefore the main contribution of \eqref{oscillatory-integral} comes from those points where both $\Phi$ and $\nabla_{\!\eta}\Phi$ vanish.

By \eqref{phase-elliptic}, \eqref{Duhamel-wa} and \eqref{Duhamel-kg}, we can verify that 
for $\text{\large$\Phi$}_{\,wa}^{\iota_1\iota_2}$\,:
\begin{align}\label{wa-resonances}
    \left\{
    \begin{array}{ll}
    \mathcal{R}=\mathcal{T}=\varnothing\subset \mathcal{S}=\{\xi=2\eta\}
    \quad& \text{for }\; \iota_1\iota_2 = ++,-- \,, \\[2pt]
    \mathcal{R}=\mathcal{T}=\mathcal{S}=\{\xi=0\}
    \quad& \text{for }\; \iota_1\iota_2 = +-,-+ \,,
    \end{array}
    \right.
\end{align}
and for $\text{\large$\Phi$}_{\,kg}^{\iota_1\iota_2}$\,:
\begin{align}\label{kg-resonances}
    \left\{
    \begin{array}{ll}
    \mathcal{R}=\mathcal{T}=\varnothing\subset \mathcal{S}=\{\eta=0\}
    \quad& \text{for }\; \iota_1\iota_2 = -+,-- \,, \\[2pt]
    \mathcal{R}=\mathcal{T}=\mathcal{S}=\{\eta=0\}
    \quad& \text{for }\; \iota_1\iota_2 = ++,+- \,.
    \end{array}
    \right.
\end{align}
We call it the ``non-resonant'' case if \,$\mathcal{R}=\mathcal{T}=\varnothing$, for which we can handle using the normal form method (integration by parts in time). 
The rest ($\iota_1\iota_2 = +-,-+$ for Wave and $\iota_1\iota_2 = ++,+-$ for KG) are called the ``resonant'' case, which is the part that contributes to the ``resonant system''.

\smallskip
\subsubsection{Extraction of the ``Resonant System''}

Our general approach of obtaining the principal part of $\widehat{V}(t,\xi)$ will be to first restrict the integral in the Duhamel's formula near the resonant set and then analyzing the resulting contribution:
the ``resonant system'' emerges from the reduction of \eqref{Duhamel-diff-wa} and \eqref{Duhamel-diff-kg} to resonant interactions.
Below we give a schematic description of our main steps.

\noindent{\it \underline{Step\;1:} Low-Frequency Outputs for Wave Component.}

As discussed in the previous section, for the wave equation, we will focus on the resonant case $\iota_1\iota_2 = +-,-+$ (when the two inputs have the opposite signs)
and look at the contribution of low-frequency outputs:
\begin{align}
    \abs{\xi}\ls \br{t}^{-1+}.
\end{align}
To be precise, we may take a cutoff function  (see the rigorous definition of $\varphi_{\leq 0}$ in Section\;\ref{dyadic section})
\begin{align}
    \varphi_{\leq0}\big(\xi\br{t}^{\frac{7}{8}}\big).
\end{align}
The cutoff is made time-dependent (shrinking as $t\rightarrow\infty$) in order to trade size of support for time decay.
Here the exponent $\frac{7}{8}$ is a convenient number slightly smaller than $1$ (which is chosen from later proofs).

We start from \eqref{Duhamel-diff-wa} with $(\iota_1,\iota_2)=(+,-)$ or $(-,+)$ when $\abs{\xi}$ is small. By Taylor expansion with respect to $\xi$ up to the first order (here $\iota_2=-\iota_1$), 
\begin{align}
    \text{\large$\Phi$}_{\,wa}^{\iota_1\iota_2}(\xi,\eta)
&=\abs{\xi}-\iota_1\br{\xi\!-\!\eta}-\iota_2\br{\eta}=\abs{\xi}-\iota_1\big\{\br{\xi\!-\!\eta}-\br{\eta}\big\}\\
&\,\leadsto\, \abs{\xi}+\iota_1\frac{\xi\cdot\eta}{\br{\eta}}=: \text{\large$\Phi$}_{\,wa}^{\iota_1\iota_2,0}(\xi,\eta).\no
\end{align}
On the other hand, taking $\xi=0$, we get
\begin{align}
    \text{\large$a$}_{\iota_1\iota_2}(\xi,\eta)&\,\leadsto\, 2,\qquad
    \widehat{F}(t,\xi\!-\!\eta)\,\leadsto\, \widehat{F}(t,-\eta).
\end{align}
Substitute all the elements above into \eqref{Duhamel-wa} and let
\begin{align}
    \text{\Large$\mathbf{I}$}_{\,wa}^{+-,0}(t,\xi) 
&:=\, \frac{1}{2}(2\pi)^{-\frac{3}{2}} \int_{\mathbb{R}^3} \text{\large$\ue$}^{\ui t\Phi_{\,wa}^{+-,0}(\xi,\eta)} \widehat{V^{kg,+}}(t,-\eta)\widehat{V^{kg,-}}(t,\eta)\,\dd\eta,\\
    \text{\Large$\mathbf{I}$}_{\,wa}^{-+,0}(t,\xi) 
&:=\, \frac{1}{2}(2\pi)^{-\frac{3}{2}} \int_{\mathbb{R}^3} \text{\large$\ue$}^{\ui t\Phi_{\,wa}^{-+,0}(\xi,\eta)} \widehat{V^{kg,-}}(t,-\eta)\widehat{V^{kg,+}}(t,\eta)\,\dd\eta.
\end{align}
By change of variable \,$\eta\mapsto-\eta$\, in \,$\text{\Large$\mathbf{I}$}_{\,wa}^{+-,0}$\, and noticing that \,$\widehat{V^{kg,-}}(t,-\eta)=\overline{\widehat{V^{kg,+}}(t,\eta)}$\,, we see that
\begin{align}
    \text{\Large$\mathbf{I}$}_{\,wa}^{+-,0}(t,\xi) 
&=\, \frac{1}{2}(2\pi)^{-\frac{3}{2}} \int_{\mathbb{R}^3} \text{\large$\ue$}^{\ui t\left(\abs{\xi}-\frac{\xi\cdot\eta}{\br{\eta}}\right)} \widehat{V^{kg,+}}(t,\eta)\widehat{V^{kg,-}}(t,-\eta)\,\dd\eta\\
&=\, \frac{1}{2}(2\pi)^{-\frac{3}{2}} \int_{\mathbb{R}^3} \text{\large$\ue$}^{\ui t\left(\abs{\xi}-\frac{\xi\cdot\eta}{\br{\eta}}\right)} \babs{\widehat{V^{kg}}(t,\eta)}^2\,\dd\eta,\no\\
\text{\Large$\mathbf{I}$}_{\,wa}^{-+,0}(t,\xi) 
&=\, \frac{1}{2}(2\pi)^{-\frac{3}{2}} \int_{\mathbb{R}^3} \text{\large$\ue$}^{\ui t\left(\abs{\xi}-\frac{\xi\cdot\eta}{\br{\eta}}\right)} \babs{\widehat{V^{kg}}(t,\eta)}^2\,\dd\eta.
\end{align}
Finally, tagging the low-frequency cutoff, we arrive at
\begin{align}
    h(t,\xi)&:=\varphi_{\leq0}\big(\xi\br{t}^{\frac{7}{8}}\big)\cdot\left\{\text{\Large$\mathbf{I}$}_{\,wa}^{+-,0}(t,\xi) + \text{\Large$\mathbf{I}$}_{\,wa}^{-+,0}(t,\xi) \right\}=(2\pi)^{-\frac{3}{2}}\varphi_{\leq0}\big(\xi\br{t}^{\frac{7}{8}}\big)\cdot\int_{\mathbb{R}^3} \text{\large$\ue$}^{\ui t\left(\abs{\xi}-\frac{\xi\cdot\eta}{\br{\eta}}\right)} \babs{\widehat{V^{kg}}(t,\eta)}^2\,\dd\eta.
\end{align}
We further define
\begin{align}\label{low-frequency-H}
    H(t,\xi):=&\;\varphi_{\leq0}\big(\xi\br{t}^{\frac{7}{8}}\big)\cdot\bigg\{\widehat{V^{wa}}(0,\xi)+\int_0^t h(s,\xi)\,\ud s\bigg\}\\
    =&\;\varphi_{\leq0}\big(\xi\br{t}^{\frac{7}{8}}\big)\cdot\bigg\{\widehat{V^{wa}}(0,\xi)+(2\pi)^{-\frac{3}{2}}\int_0^t\int_{\mathbb{R}^3} \text{\large$\ue$}^{\ui s\left(\abs{\xi}-\frac{\xi\cdot\eta}{\br{\eta}}\right)} \babs{\widehat{V^{kg}}(s,\eta)}^2\,\dd\eta\ud s\bigg\}.\no
\end{align}

\begin{remark}
We expect that, roughly speaking,
$\widehat{V^{wa}}(t,\xi)\approx H(t,\xi) $ and 
\begin{align} \label{wa-app}
    \partial_t \widehat{V^{wa}}(t,\xi) \,&=  h(t,\xi)+\rwa(t,\xi).
\end{align}
Here $\rwa$ should be a remainder which has sufficiently fast time decay in some suitable norm (see \eqref{S_1'-norm} and \eqref{T_1'-norm}) to guarantee convergence to the resonant system.
In other words, the low-frequency outputs of wave component contribute a bulk term $h$ that decays at a critical rate $\br{t}^{-1}$ (see Lemma\;\ref{vtt lemma 9}).
\end{remark}

\noindent{\it \underline{Step\;2:} High-Low Interactions and ``Phase Shift'' for KG Component.}

The next step is to measure the feedback contribution of the low-frequency bulk $h(t,\xi)$ to the nonlinear interactions of the Klein-Gordon equation. 

For the Klein-Gordon equation, we focus on the resonant case $\iota_1\iota_2 = ++,+-$ (when one of the inputs is $V^{kg,+}$) and look at the high-low interactions: 
\begin{align}
    \abs{\eta}\ls \br{t}^{-1+} (\ll|\xi|).
\end{align}

We start from \eqref{Duhamel-diff-kg} with $(\iota_1,\iota_2)=(+,+)$ or $(+,-)$ when $\abs{\eta}$ is small. By Taylor expansion with respect to $\eta$ up to the first order (here $\iota_1=+$), 
\begin{align}
    \text{\large$\Phi$}_{\,kg}^{\iota_1\iota_2}(\xi,\eta) 
&=\br{\xi}-\br{\xi\!-\!\eta}-\iota_2\abs{\eta}=\big\{\br{\xi}-\br{\xi\!-\!\eta}\big\}-\iota_2\abs{\eta}\\
&\,\leadsto\, \frac{\xi\cdot\eta}{\br{\xi}}-\iota_2\abs{\eta}=: \text{\large$\Phi$}_{\,kg}^{+\iota_2,0}(\xi,\eta).\no
\end{align}
On the other hand, taking $\eta=0$, we get
\begin{align}
    \text{\large$b$}_{+\iota_2}(\xi,\eta)&\,\leadsto\, \iota_2\frac{\abs{\xi}^2}{\br{\xi}}\frac{1}{\abs{\eta}},\qquad
    \widehat{V^{kg,+}}(t,\xi\!-\!\eta)\,\leadsto\,\widehat{V^{kg}}(t,\xi).
\end{align}
Moreover, we formally replace the low-frequency part of $\widehat{V^{wa,\pm}}(t,\eta)$ by the contribution coming from $H^{\pm}(t,\eta)$, where $H^+(t,\eta)=H(t,\eta)=\overline{H^-(t,-\eta)}$ is given by \eqref{low-frequency-H}. Note that $H(t,\eta)$ contains a cutoff $\varphi_{\leq0}\big(\eta\br{t}^{\frac{7}{8}}\big)$.
Let
\begin{align}
    \text{\Large$\mathbf{I}$}_{\,kg,0}^{+\iota_2,1}(t,\xi) 
:=&\;\frac{1}{4}(2\pi)^{-\frac{3}{2}}\int_{\mathbb{R}^3} \text{\large$\ue$}^{\ui t\Phi_{\,kg}^{+\iota_2,0}(\xi,\eta)}\iota_2\frac{\abs{\xi}^2}{\br{\xi}} \frac{1}{\abs{\eta}} \widehat{V^{kg}}(t,\xi)H^{\iota_2}(t,\eta)\,\dd\eta,\\
=&\;\iota_2\frac{1}{4}(2\pi)^{-\frac{3}{2}}\frac{\abs{\xi}^2}{\br{\xi}}\widehat{V^{kg}}(t,\xi) \int_{\mathbb{R}^3} \text{\large$\ue$}^{\ui t\left(\frac{\xi\cdot\eta}{\br{\xi}}-\iota_2\abs{\eta}\right)}\frac{1}{\abs{\eta}} H^{\iota_2}(t,\eta)\,\dd\eta.\no
\end{align}
By change of variable \,$\eta\mapsto-\eta$\, in \,$\text{\Large$\mathbf{I}$}_{\,kg,0}^{+-,1}$\,, we have
\begin{align}
    \text{\Large$\mathbf{I}$}_{\,kg,0}^{+-,1}(t,\xi) 
&=\, -\frac{1}{4}(2\pi)^{-\frac{3}{2}}\frac{\abs{\xi}^2}{\br{\xi}}\widehat{V^{kg}}(t,\xi) \int_{\mathbb{R}^3} \text{\large$\ue$}^{\ui t\left(\frac{\xi\cdot\eta}{\br{\xi}}+\abs{\eta}\right)}\frac{1}{\abs{\eta}} H^{-}(t,\eta)\,\dd\eta\\
&=\, -\frac{1}{4}(2\pi)^{-\frac{3}{2}}\frac{\abs{\xi}^2}{\br{\xi}}\widehat{V^{kg}}(t,\xi) \int_{\mathbb{R}^3} \text{\large$\ue$}^{-\ui t\left(\frac{\xi\cdot\eta}{\br{\xi}}-\abs{\eta}\right)}\frac{1}{\abs{\eta}} \overline{H(t,\eta)}\,\dd\eta.\no
\end{align}
Hence,
\begin{align}
    \text{\Large$\mathbf{I}$}_{\,kg,0}^{++,1}(t,\xi) + \text{\Large$\mathbf{I}$}_{\,kg,0}^{+-,1}(t,\xi)
    &= \frac{\ui}{2}(2\pi)^{-\frac{3}{2}}\frac{\abs{\xi}^2}{\br{\xi}}\widehat{V^{kg}}(t,\xi)\cdot \Im\left\{\int_{\mathbb{R}^3} \text{\large$\ue$}^{\ui t\left(\frac{\xi\cdot\eta}{\br{\xi}}-\abs{\eta}\right)}\frac{1}{\abs{\eta}} H(t,\eta)\,\dd\eta\right\} \\
    &=: \ui\, \mathcal{C}(t,\xi)\cdot\widehat{V^{kg}}(t,\xi) , \no
\end{align}
where 
\begin{align}\label{c-phase}
\mathcal{C}(t,\xi):=\frac{1}{2}(2\pi)^{-\frac{3}{2}}\frac{\abs{\xi}^2}{\br{\xi}}\cdot \Im\left\{\int_{\mathbb{R}^3} \text{\large$\ue$}^{\ui t\left(\frac{\xi\cdot\eta}{\br{\xi}}-\abs{\eta}\right)}\frac{1}{\abs{\eta}} H(t,\eta)\,\dd\eta\right\}.
\end{align}

\begin{remark}
We will see that \eqref{Duhamel-diff-kg} can be written as 
\begin{align} \label{kg-app}
    \partial_t \widehat{V^{kg}}(t,\xi) \,&=\, \ui\, \mathcal{C}(t,\xi)\cdot\widehat{V^{kg}}(t,\xi)+\rkg(t,\xi),
\end{align}
where $\rkg$ is a remainder which decays sufficiently fast in time in some suitable norm (see \eqref{S_2'-norm} and \eqref{T_2'-norm}).
In particular, $\mathcal{C}(t,\xi)$ is not integrable in time (see Lemma\;\ref{vtt lemma 13}). 
This shows that 
the KG component is essentially transported by the low frequencies of wave bulk. Eventually, we will renormalize $V^{kg}$ to incorporate this modification via a phase correction/shift using the fact that $\mathcal{C}(t,\xi)\in\r$. 
This is a similar phenomenon to the one occurring in the scattering-critical equations, such as the critical nonlinear Schr\"{o}dinger equations and 2D water waves (see \cite{Hayashi.Naumkin1998, Kato.Pusateri2011, Hani.Pausader.Tzvetkov.Visciglia2015, Ionescu.Pusateri2015, Ionescu.Pusateri2018}).
\end{remark}

\noindent{\it \underline{Step\;3:} The Resonant System and Long-Time Asymptotics.}

Collecting all the contributions from previous steps (see \eqref{wa-app} and \eqref{kg-app}), 
we extract the ``resonant system'' as 
\begin{align}
    \partial_t \widehat{V^{wa}}(t,\xi) \,&\approx\,  h(t,\xi),\label{resonant system 1}\\
    \partial_t \widehat{V^{kg}}(t,\xi) \,&\approx\, \ui\, \mathcal{C}(t,\xi)\cdot\widehat{V^{kg}}(t,\xi),\label{resonant system 2}
\end{align}
and further deduce that
\begin{align}
    \widehat{V^{wa}}(t,\xi) \,&\approx\, \widehat{V^{wa}}(0,\xi) +\int_0^t h(s,\xi)\,\ud s,\label{resonant system 3}\\
    \widehat{V^{kg}}(t,\xi) \,&\approx\, \ue^{\ui \ddd(t,\xi)}\widehat{V^{kg}}(0,\xi),\label{resonant system 4}
\end{align}
where the phase correction 
\begin{align}
    \ddd(t,\xi):=\int_0^t\ccc(s,\xi)\,\ud s
\end{align}
is a real-valued nonlinear function of $\widehat{V^{kg}}$ which takes into account the low-frequency part of wave component.
In particular, it grows mildly to infinity in time (see Lemma\;\ref{vtt lemma 13'}), and therefore will lead to genuine nonlinear scattering for the KG component.

The asymptotic dynamics and modified scattering for the W-KG system can be explained heuristically as follows: 
\begin{enumerate}
    \item[(1)] The wave and Klein-Gordon components are generated linearly by initial data; 
    \item[(2)] The KG component interacts with itself in the wave equation to produce a large low-frequency wave component
    $|u(t,x)| \approx t^{-1} \varepsilon_0^{\,2}$\, (if $t\gg 1$ and $|x|\leq ct$) supported at low frequency \,$|\xi|\lesssim 1/t$\,;
    \item[(3)] In the Klein-Gordon equation, this large low-frequency wave component interacts again with the high frequencies of the Klein-Gordon component to produce a ``phase shift''. 
\end{enumerate}

\smallskip
\subsubsection{The ``Asymptotic Profile''}

Now that we have identified the asymptotic system (the so-called ``resonant system'') that governs long-time behavior of the W-KG system,
our next step is to construct an ``asymptotic profile'' (which still depends on $t$) as the limit of a nonlinear profile $V$ as $t\rightarrow\infty$.

Motivated by \eqref{resonant system 3}\eqref{resonant system 4}, we define the ``asymptotic profile'' $(\vv^{wa},\vv^{kg})$ (which can be understood as an asymptotic approximation of $(V^{wa},V^{kg})$) as follows:
\begin{align}
    \widehat{\vv^{wa}}(t,\xi) \,&:= \widehat{V^{wa}_{\infty}}(\xi) +\hhfi(t,\xi),\label{resonant 1}\\
    \widehat{\vv^{kg}}(t,\xi) \,&:=\, \ue^{\ui \ddi(t,\xi)}\widehat{V^{kg}_{\infty}}(\xi),\label{resonant 2}
\end{align}
where $(V^{wa}_{\infty},V^{kg}_{\infty})$ is the initial data
\begin{align}
    \widehat{\vv^{wa}}(0,\xi)=\widehat{V^{wa}_{\infty}}(\xi),\qquad \widehat{\vv^{kg}}(0,\xi)=\widehat{V^{kg}_{\infty}}(\xi).
\end{align}
with quantities
\begin{align}
    \hhi(t,\xi) &:= (2\pi)^{-\frac{3}{2}}\varphi_{\leq0}\big(\xi\br{t}^{\frac{7}{8}}\big)\cdot\!\int_{\mathbb{R}^3} \text{\large$\ue$}^{\ui t\left(\abs{\xi}-\frac{\xi\cdot\eta}{\br{\eta}}\right)} \big|\widehat{V^{kg}_\infty}(\eta)\big|^2\!\dd\eta \,,\label{h_infty}\\
    \hhfi(t,\xi) &:= \int_0^t\!\hhi(s,\xi)\,\ud s \,,\label{H_infty}\\
    H_\infty(t,\xi) &:= \varphi_{\leq0}\big(\xi\br{t}^{\frac{7}{8}}\big)\!\cdot\!
    \Big\{\widehat{V^{wa}_{\infty}}(\xi) + \hhfi(t,\xi)\Big\} \,,
\end{align}
and  
\begin{align}
    \cci(t,\xi) &:= \frac{1}{2}(2\pi)^{-\frac{3}{2}}\frac{\abs{\xi}^2}{\br{\xi}}\cdot \Im\left\{\int_{\mathbb{R}^3} \text{\large$\ue$}^{\ui t\left(\frac{\xi\cdot\eta}{\br{\xi}}-\abs{\eta}\right)}\frac{1}{\abs{\eta}} H_\infty(t,\eta)\,\dd\eta\right\} \,,\label{C_infty}\\
    \ddi(t,\xi) &:= \int_0^t\!\cci(s,\xi)\,\ud s\,. \label{D_infty}
\end{align}

\begin{remark}
For all functions above, we also define its positive and negative counterpart in the same way as $\widehat{V^\pm}$: for a complex-valued function $\widehat S(t,\xi)$, let $\widehat{S^+}(t,\xi):=\widehat S(t,\xi)$ and $\widehat{S^-}(t,\xi):=\overline{\widehat S(t,-\xi)}$.
\end{remark}

\smallskip
\subsection{Equations of Perturbation}\label{definition section}

In this subsection, we will take the difference between the profile $(V^{wa},V^{kg})$ of solutions and the asymptotic profile $(\vv^{wa},\vv^{kg})$ as our perturbation unknown, and derive the equations for this new unknown. 
We represent the perturbation $\gwa(t,x)$, $\gkg(t,x)$ by
\begin{align}
    \hgwa(t,\xi):=\;&\hvwa(t,\xi)-\left(\lvwa(\xi)+\hhfi(t,\xi)\right) \,, \label{ott 01}\\
    \hgkg(t,\xi):=\;&\hvkg(t,\xi)-\Big(\ue^{\ui \ddi(t,\xi)}\lvkg(\xi)+\bbfi(t,\xi)\Big) \,. \label{ott 02}
\end{align}
Here $\bbfi(t,\xi)$ comes from contributions of the non-resonant part 
and is given by
\begin{align}
    \bbfi(t,\xi)&:=-\int_t^{\infty}\ue^{\ui \big(\ddi(t,\xi)-\ddi(s,\xi)\big)}\bii(s,\xi)\,\ud s,\label{B_infty}\\
    \bii(t,\xi)&:=\sum_{\iota_2\in\{+,-\}}{\bbi^{-\iota_2}_{kg}}\Big[\ue^{-\ui \ddi^-(t,\xi)}\widehat{V_{\infty}^{kg,-}},H_{\infty}^{\iota_2}\Big]. \label{b_infty}
\end{align}

\begin{remark}
From \eqref{ott 01} and \eqref{ott 02} we see that $\hgwa=\hvwa-\widehat{\vv^{wa}}$ and $\hgkg\approx\hvkg-\widehat{\vv^{kg}}$; the latter is slightly different compared with the expression \eqref{resonant 2} of $\vv^{kg}$. In particular, it contains a non-resonant contribution $\bbfi$. This is due to technical difficulties in nonlinear estimates (see Section \ref{sec:s2t2}). Roughly speaking, though being non-resonant, $\bii$ decays like $\br{t}^{-1}$ and will be captured by the asymptotic analysis. However, it will not affect the final result since the time integral $\bbfi$ decays like $\br{t}^{-\frac{7}{8}}$ (see Lemma\;\ref{vtt lemma 18}). In other words, the the estimate of non-resonant part will improve upon taking time integration.
\end{remark}

Our goal is to show that given any resonant system starting from $(V^{wa}_{\infty},V^{kg}_{\infty})$, there exists a unique solution $(V^{wa},V^{kg})$ to the W-KG system such that their difference goes to zero under some proper topology (which we will specify in Section\;\ref{Sec:norms}) as $t\rt\infty$, i.e.
\begin{align}
    \lim_{t\rt\infty}\hgwa(t,\xi)=0,\qquad \lim_{t\rt\infty}\hgkg(t,\xi)=0.
\end{align}

Inserting \eqref{ott 01}\eqref{ott 02} into \eqref{Duhamel-diff-wa}\eqref{Duhamel-diff-kg}, we obtain the system for $(\gwa,\gkg)$
\begin{align}\label{wtt 05-diff}
    \dt\hgwa
    =&\; -\hhi + \sum_{\iota_1,\iota_2\in\{+,-\}} \bbi_{\,wa}^{\iota_1\iota_2}\Big[\hgkgg{\iota_1}+\ue^{\iota_1\ui \ddi^{\iota_1}}\lvkgg{\iota_1}+\bbfi^{\iota_1},\ \hgkgg{\iota_2}+\ue^{\iota_2\ui \ddi^{\iota_2}}\lvkgg{\iota_2}+\bbfi^{\iota_2}\Big]\,,
\end{align}
and
\begin{align}\label{wtt 06-diff}
    \dt\hgkg
    =&\;-\ui\,\cci\big(\ue^{\ui\ddi}\lvkg\big)-\ui\,\cci\bbfi-\bii\;+ \sum_{\iota_1,\iota_2\in\{+,-\}} \bbi_{\,kg}^{\iota_1\iota_2}
    \Big[\hgkgg{\iota_1}+\ue^{\iota_1\ui \ddi^{\iota_1}}\lvkgg{\iota_1}+\bbfi^{\iota_1},\ \hgwaa{\iota_2}+\lvwaa{\iota_2}+\hhf^{\iota_2}_\infty\Big]\,,
\end{align}
with final data
\begin{align} \label{G-final-condition}
    \hgwa(\infty,\xi)=0\,,\qquad \hgkg(\infty,\xi)=0\,.
\end{align}
It suffices to consider the well-posedness of $(\gwa,\gkg)$ as solutions to \eqref{wtt 05-diff}\eqref{wtt 06-diff}\eqref{G-final-condition}.

\smallskip
\subsection{Definition and Notation}

Throughout the paper, $C$ will generally denote a universal constant that may vary from line to line.
The notation $A \lesssim B$\, means that $A\leq CB$\, for some universal constant $C>0$\,; we will use \,$\gtrsim$\, and \,$\simeq$\, in a similar standard way.

\smallskip
\subsubsection{Littlewood-Paley Projections}\label{dyadic section}

In order to make the following analysis more precise, we need a good way of localization in Fourier space and physical space, which then allows us to quantify the variables in terms of time.

We first use a standard (inhomogeneous) dyadic decomposition of the indicator function $\mathbf{1}_{[0,+\infty)}^t$ to localize in $t$.
Fix a smooth cutoff function $\tau:\mathbb{R}_+\rightarrow[0,1]$\, supported in $[0,2]$ and equal to $1$ in $[0,1]$, and define for any $m\in\mathbb{N}$,
\begin{equation} \label{t-decomp}
\begin{split}
\tau_m(t) &:= \tau(t/2^m)-\tau(t/2^{m-1}) \,, \quad m\geq 1 \,; \\
\tau_0(t) &:= 1 - \sum_{m=1}^{+\infty}\tau_m(t) = \tau(t) \,,
\end{split}
\end{equation}
so that the sequence of functions $\big\{\tau_m(t)\big\}_{m\in\mathbb{N}}$ \,has the properties
\begin{equation} \label{t-decomp-prop}
\mathrm{supp}\,\tau_0\subseteq [0,2] \,, \qquad \mathrm{supp}\,\tau_m\subseteq [2^{m-1},2^{m+1}]\;\; (m\geq 1) \,, \qquad
\sum_{m=0}^{+\infty}\tau_m(t) \equiv 1 \,.
\end{equation}
In other words, in the support of \,$\tau_m$\, we have \,$t\approx 2^m$. 
Also, we choose the function $\tau$ in such a way that \,$\left|\tau_m^\prime(t)\right|\lesssim 2^{-m}$.

Also, we introduce the (homogeneous) dyadic decomposition\footnote{ The homogeneous dyadic decomposition differs from the inhomogeneous one in that it decomposes also near the origin.} in three dimensions.
Fix a smooth radial cutoff function \,$\varphi:\mathbb{R}^3\rightarrow[0,1]$\, that equals \,$1$\, for \,$|z|\leq 1$\, and vanishes for \,$|z|\geq 2$\,. 
For any \,$k\in\mathbb{Z}$\,, denote by \,$k^+:=\max(k,0)$\, and \,$k^-:=\min(k,0)$\,. Let 
\begin{equation} \label{z-decomp}
\varphi_k(z) :=\, \varphi(z/2^k)-\varphi(z/2^{k-1})\,,\quad k\in\mathbb{Z}
\end{equation}
be a sequence of functions with the properties 
\begin{equation} \label{z-decomp-prop}
\mathrm{supp}\,\varphi_k\subseteq\left\{z\in\mathbb{R}^3:|z|\in \big[2^{k-1},2^{k+1}\big]\right\} \,, \qquad
\sum_{k\in\mathbb{Z}}\varphi_k(z) \equiv 1\;\;\; (z\in\mathbb{R}^3\backslash\{0\}) \,,
\end{equation}
so that in the support of \,$\varphi_k$\, we have \,$|z|\approx 2^k$.
We define also
\begin{align*} \label{z-decomp-def}
\begin{split}
& \varphi_I(z) := \sum_{k \in I\cap \Z}\varphi_k(z) \quad \,\textrm{for any }I\subset\r\,,
\\
& \varphi_{\leq A} := \varphi_{(-\infty, A]} \,,
\quad \varphi_{\geq A} := \varphi_{[A,+\infty)} \,,
\quad \varphi_{<A} := \varphi_{(-\infty, A)} \,,
\quad \varphi_{>A} := \varphi_{(A,+\infty)} 
\quad \,\textrm{for any }A\in\r\,. \\[5pt]
\end{split}
\end{align*}
Let 
\begin{equation*}
\mathcal{J} := \big\{(k,j)\in\mathbb{Z}\times\mathbb{Z}_+:\,k+j\geq 0 \;\textrm{ i.e. }j\geq -k^- \big\}.
\end{equation*}
Given any $k\in\mathbb{Z}$\,, for any $(k,j)\in\mathcal{J}$\,, we further define 
\begin{equation}
\varphi_j^{(k)} :=
\left\{
\begin{array}{ll}
\varphi_j \,,&j>-k^- \,; \\[3pt]
\varphi_{\leq j} \,,&j=-k^- \,,
\end{array}
\right.
\end{equation}
and notice that, for any fixed $k\in\mathbb{Z}$\,, 
\begin{equation*}
\sum_{j\geq -k^-}\varphi^{(k)}_j(z)\equiv 1\;\;\; (z\in\mathbb{R}^3\backslash\{0\}) \,.
\end{equation*}

\begin{definition}[Littlewood-Paley Projections] \label{Def:L-P}
(1) Let \,$P_k \;(k\in\Z)$\, denote the standard Littlewood-Paley projection on $\r^3$ defined by the Fourier multiplier \,$\xi\mapsto\varphi_k(\xi)$\,:
\begin{equation}
P_k := \mathscr{F}^{-1}\varphi_k(\xi)\,\mathscr{F} \,,\qquad 
\widehat{P_kf}(\xi) := \varphi_k(\xi)\widehat{f}(\xi) \,.
\end{equation}
Similarly, for any \,$A\in\r$\,, define the operator \,$P_{\leq A}$\, (resp. $P_{>A}$) on $\r^3$ by the Fourier multiplier \,$\xi\mapsto\varphi_{\leq A}(\xi)$\, (resp. $\xi\mapsto\varphi_{>A}(\xi)$\,).

\smallskip
(2) For $(k,j)\in\mathcal{J}$, let the operator \,$Q_{jk}$ be defined as
\begin{equation}
Q_{jk}f(x) := \varphi_j^{(k)}\!\cdot P_kf(x) \,.
\end{equation}

\smallskip
(3) For $(k,j)\in\mathcal{J}$, let the operator \,$\qq_{jk}$ be defined as
\begin{equation}
\qq_{jk}f(x) := P_{[k-2,\,k+2]}\big\{\varphi_j^{(k)}\!\cdot P_kf\big\}(x) \,.
\end{equation}

\smallskip(4)
Furthermore, let
\begin{align}
    Q_{\leq jk}f&:=\sum_{-k^-\leq j'\leq j}Q_{j'k}f,\qquad
    \qq_{\leq jk}f:=\sum_{-k^-\leq j'\leq j}\qq_{j'k}f.
\end{align}
\end{definition}

\begin{remark} \label{Rmk:L-P}
(i) In our case we have \,$\Lambda(\xi)=\omega(|\xi|)$\,, so we will focus on quantifying the ``energy'' $|\xi|$. The direction is also important, but can usually be recovered by commuting with the rotation vector fields.

(ii) In view of the Heisenberg's uncertainty principle\footnote{\,It states that the more precisely the position (in physical space) of some particle is determined, the less precisely its momentum (frequency) can be known, and vice versa.}, the operators \,$Q_{jk}$\, are relevant only when \,$2^j2^k\gtrsim 1$\, i.e. \,$j+k\gtrsim 0$\,, which explains the definitions above.

(iii) The Littlewood--Paley operators give a decomposition of the identity (at least formally):
\begin{equation} \label{Id-decomp}
\begin{split}
\sum_{k\in\Z}P_k = \id \,, \qquad 
&\quad \sum_{k\in\Z}P_k f = f \,; \\
\sum_{j\geq -k^-} Q_{jk} = P_k  \,, \qquad 
&\sum_{(k,j)\in\mathcal{J}}\! Q_{jk}f = \sum_{k\in\Z}P_k f = f \,; \\
\sum_{j\geq -k^-}\qq_{jk} = P_k  \,, \qquad 
&\sum_{(k,j)\in\mathcal{J}}\! \qq_{jk}f = \sum_{k\in\Z}P_k f = f \,.
\end{split}
\end{equation}
\end{remark}

\smallskip
\subsubsection{Norms and Function Spaces} \label{Sec:norms}

Set the parameters 
\begin{align*}
    &N_0:=40,\quad N_1:=3,\quad d:=10,\quad \d=10^{-10} \\
    &N(0):=N_0,\qquad N(n):=N_0-dn,\\
    &H(0):=800,\qquad H(n):=800-200n,
\end{align*}
and 
\begin{align*}
    &H''_{wa}(n) = H''_{kg}(n) = H(n+1),\\
    &N''_{wa}(n) = N''_{kg}(n) = N(n)-5.
\end{align*}

\noindent{\it \underline{Time-Independent Norms.}}
We define the norms characterizing time-independent data (designed for $(V_{\infty}^{wa},V_{\infty}^{kg})$).
Assume $f(x)$ is a function of certain regularity defined for $x\in\r^3$. Define $f_{\lll}$ as in \eqref{vector field 1} and \eqref{vector field 2}.

Define the ``data spaces'' 
\begin{align}
    Y_i:=\big\{f: \nm{f}_{Y_i}<\infty\big\}
\end{align}
for $i=1,2$, with the norms
\begin{align}
    \nm{f}_{Y_1}:=&\sup_{n\leq N_1}\sup_{\lll\in\vvv_n}\bnm{\abs{\na}^{-\frac{1}{2}}f_{\lll}}_{H^{N(n-3)}}
    +\sup_{n\leq N_1-1}\sup_{\lll\in\vvv_n, \ell\in\{1,2,3\}}\sup_{k\in\mathbb{Z}}\Big(2^{N(n-2)k^+}2^{\frac{k}{2}}\bnm{\varphi_k \p_{\xi_{\ell}}\widehat{f_{\lll}}}_{L^2_{\xi}}\Big),\\[5pt]
\nm{f}_{Y_2}:=&\sup_{n\leq N_1}\sup_{\lll\in\vvv_n}\nm{f_{\lll}}_{H^{N(n-3)}}+\sup_{n\leq N_1-1}\sup_{\lll\in\vvv_n, \ell\in\{1,2,3\}}\sup_{k\in\mathbb{Z}}\Big(2^{N(n-2)k^+}2^{k^+}\bnm{\varphi_k\p_{\xi_{\ell}}\widehat{f_{\lll}}}_{L^2_{\xi}}\Big).
\end{align}

\noindent{\it \underline{Time-Dependent Norms.}}
Next we define the time-dependent norms. Assume $f(t,x)$ is a function of certain regularity defined for $t\in[1,\infty)$ and $x\in\r^3$. 

Define the ``working spaces'' (to run the fixed-point argument on $\gwa$ and $\gkg$) 
\begin{align}
    X_i:=\big\{f: \nm{f}_{X_i}<\infty\big\}
\end{align}
for $i=1,2$, with the norms
\begin{align}
    \nm{f}_{X_1}:=&\nm{f}_{S_1}+\nm{f}_{T_1} \quad\text{(for $\gwa$)},\\
    \nm{f}_{X_2}:=&\nm{f}_{S_2}+\nm{f}_{T_2} \quad\text{(for $\gkg$)},
\end{align}
where
\begin{align}
    \nm{f}_{S_1}:=&\sup_{t\in[1,\infty)}\sup_{n\leq N_1}\sup_{\lll\in\vvv_n}\Big(\br{t}^{H(n)\d}\bnm{\abs{\na}^{-\frac{1}{2}}f_{\lll}}_{H^{N(n)}}\Big),\\
    \nm{f}_{T_1}:=&\sup_{t\in[1,\infty)}\sup_{n\leq N_1-1}\sup_{\lll\in\vvv_n, \ell\in\{1,2,3\}}\sup_{k\in\mathbb{Z}}\Big(\br{t}^{H(n+1)\d}2^{N(n+1)k^+}2^{\frac{k}{2}}\bnm{\varphi_k\p_{\xi_{\ell}}\widehat{f_{\lll}}}_{L^2_{\xi}}\Big),
\end{align}
and
\begin{align}
    \nm{f}_{S_2}:=&\sup_{t\in[1,\infty)}\sup_{n\leq N_1}\sup_{\lll\in\vvv_n}\Big(\br{t}^{H(n)\d}\nm{f_{\lll}}_{H^{N(n)}}\Big),\\
    \nm{f}_{T_2}:=&\sup_{t\in[1,\infty)}\sup_{n\leq N_1-1}\sup_{\lll\in\vvv_n, \ell\in\{1,2,3\}}\sup_{k\in\mathbb{Z}}\Big(\br{t}^{H(n+1)\d}2^{N(n+1)k^+}2^{k^+}\bnm{\varphi_k\p_{\xi_{\ell}}\widehat{f_{\lll}}}_{L^2_{\xi}}\Big).
\end{align}
We also need to introduce 
the auxiliary norms measuring corresponding variables with time derivatives (for $\dt\gwa$ and $\dt\gkg$):
\begin{align}
    \nm{f}_{S_1'}:=&\sup_{t\in[1,\infty)}\sup_{n\leq N_1}\sup_{\lll\in\vvv_n}\Big(\br{t}^{1+H''_{wa}(n)\d}2^{N''_{wa}(n)k^+}2^{-\frac{1}{2}k^-}\bnm{\varphi_k\widehat{f_{\lll}}}_{L^2_{\xi}}\Big),\label{S_1'-norm}\\
    \nm{f}_{T_1'}:=&\sup_{t\in[1,\infty)}\sup_{n\leq N_1-1}\sup_{\lll\in\vvv_n, \ell\in\{1,2,3\}}\sup_{k\in\mathbb{Z}}\bigg(\br{t}^{H''_{wa}(n)\d}2^{N''_{wa}(n)k^+}2^{-\frac{1}{2}k^-}\nm{\varphi_k \p_{\xi_{\ell}}\big(\ue^{-\ui t\lawa(\xi)}\widehat{f_{\lll}}\big)}_{L^2_{\xi}}\bigg),\label{T_1'-norm}
\end{align}
and
\begin{align}
    \nm{f}_{S_2'}:=&\sup_{t\in[1,\infty)}\sup_{n\leq N_1}\sup_{\lll\in\vvv_n}\Big(\br{t}^{1+H''_{kg}(n)\d}2^{N''_{kg}(n)k^+}\bnm{\varphi_k\widehat{f_{\lll}}}_{L^2_{\xi}}\Big),\label{S_2'-norm}\\
    \nm{f}_{T_2'}:=&\sup_{t\in[1,\infty)}\sup_{n\leq N_1-1}\sup_{\lll\in\vvv_n, \ell\in\{1,2,3\}}\sup_{k\in\mathbb{Z}}\bigg(\br{t}^{H''_{kg}(n)\d}2^{N''_{kg}(n)k^+}\nm{\varphi_k\p_{\xi_{\ell}}\big(\ue^{-\ui t\lakg(\xi)}\widehat{f_{\lll}}\big)}_{L^2_{\xi}}\bigg).\label{T_2'-norm}
\end{align}

\begin{remark}
Roughly speaking, the designed $X_i$ norms indicate that the perturbation $(\gwa,\gkg)$ decays mildly in time in certain Sobolev-type and weighted $L^2$ norms.
Notice also that there is an energy hierarchy expressed in terms of the parameters $H(n)$ (for decay rate) and $N(n)$ (for regularity/weight), in the sense that the variables with more vector fields are allowed to decay slightly slower than those with fewer vector fields, in weaker Sobolev spaces.
\end{remark}

\smallskip
\subsection{Main Theorem and Proposition} \label{Sec:Main-Thm}

As a counterpart of the forward problem, in this work we establish the existence of wave operators, namely that
any (modified) asymptotic profile leads to a unique global nonlinear solution that converges to it.

Provided that the (initial) data $(V_{\infty}^{wa},V_{\infty}^{kg})\in Y_1\times Y_2$ is given. Our goal is to justify that \eqref{wtt 05-diff}\eqref{wtt 06-diff}\eqref{G-final-condition} can uniquely determine a solution $(\gwa,\gkg)$. Below is our main theorem.

\begin{theorem}[Main Theorem - Modified Wave Operator, precise version]\label{main theorem}
There exists $\ep_0>0$ such that if the scattering data  $(V^{wa}_{\infty},V^{kg}_{\infty})\in Y_1\times Y_2$ satisfies
\begin{align}
    \big\|V_{\infty}^{wa}\big\|_{Y_1}+ \big\|V_{\infty}^{kg}\big\|_{Y_2}\ls \ep_0,
\end{align}
then there exists a unique solution $(u,v)$ of the system \eqref{Sys:W-KG-2} on $t\in[1,\infty)$ with the associated profile $\big({V^{wa}},{V^{kg}}\big)\in X_1\times X_2 $ satisfying
\begin{align}
    \big\|V^{wa}-\vv^{wa}\big\|_{X_1}+\big\|V^{kg} -\vv^{kg}\big\|_{X_2}\ls \ep_0,
\end{align}
where the asymptotic profile $(\vv^{wa},\vv^{kg})$ is explicitly given by $(V_{\infty}^{wa},V_{\infty}^{kg})$ through \eqref{resonant 1}\,--\,\eqref{D_infty}.
\end{theorem}

\begin{remark}
Since $X_i$ norms contain time growth, the above theorem actually implies that $V$ and $\vv$ demonstrate the same asymptotic behavior in certain norm as $t\rt\infty$.
\end{remark}
 
Theorem\;\ref{main theorem} follows from a standard fixed-point argument.
It suffices to prove the following bootstrap proposition, which naturally yields the boundedness and contraction in order to apply the fixed-point theorem.

\begin{proposition}\label{main proposition}
There exists $\ep_0>0$ (sufficiently small) such that for any $0\leq\ep\leq\ep_0$, if the data satisfies
\begin{align}\label{wtt-assumption}
    &\big\|V_{\infty}^{wa}\big\|_{Y_1}+\big\|V_{\infty}^{kg}\big\|_{Y_2}\ls \ep,
\end{align}
and the solution satisfies
\begin{align}\label{wtt-assumption'}
    &\big\|\gwa\big\|_{X_1}+\big\|\gkg\big\|_{X_2}\ls \ep,
\end{align}
then we have the improved bound
\begin{align}\label{bootstrap-est}
    \big\|\gwa\big\|_{X_1}+  \big\|\gkg\big\|_{X_2}\ls \ep^{\frac{3}{2}}.
\end{align}
\end{proposition}

The rest of this paper is concerned with the proof of Proposition\;\ref{main proposition}, which consists of 
Proposition\;\ref{energy 1}, Proposition\;\ref{energy 3}, Proposition\;\ref{energy 1'}, Proposition\;\ref{energy 3'}.

\begin{remark}
Based on Lemma \ref{vtt lemma 18}, we know $\tnm{\bbf_{\infty}}\ls \ep^2\br{t}^{-\frac{7}{8}+}$. Hence, the presence of $\bbfi$ in \eqref{ott 02} will not affect the main result.
\end{remark}

\smallskip
\subsection{Strategy of the Proof} 

In order to close the fixed-point argument, we plan to attack it in two steps:
\begin{itemize}
    \item 
    Step 1: Nonlinearity Estimates.\\[2pt]
    We first estimate the right-hand side nonlinear terms in \eqref{wtt 05-diff} and \eqref{wtt 06-diff}. The quadratic terms $\bbi[G,G]$ and linear terms $\bbi[G,\vv]$ may be handled with standard dispersive estimates. However, the forcing terms $\bbi[\vv,\vv]$ is much more delicate. This relies on a frequency decomposition, where we have to utilize the resonant terms $\hhi$ and $\cci$ to cancel out the slow-decaying low-frequency contributions, integration by parts in $\eta$ for the high-frequency terms, and stationary-phase argument to handle the non-resonant terms. 
    A direct consequence of this step is the estimate of time derivatives of the unknowns, i.e. $S_i'$ and $T_i'$ norms bound.
    \item
    Step 2: Solution Estimates.\\[2pt]
    Then we turn back to the bound of the original solution itself in $S_i$ and $T_i$ norm. Here actually we treat them in completely different fashion. 
    \begin{itemize}
        \item Energy estimates: $S_i$ norm is the classical energy norm, so we resort to the energy structure of the perturbation equation, which heavily relies on the trilinear estimates and $S_i'$ norm bounds.
        \item Vector fields: $T_i$ norm is much trickier. We utilize the Lorentz vector fields to generate a relation in the mild formulation (note that we cannot directly take $\p_{\xi_{\ell}}$ derivative in the mild formulation since it will hit the phase function and generate $\br{t}$ growth), which is closely related to $T_i'$ bounds.
    \end{itemize}
\end{itemize}
In this procedure, we need to utilize both the normal forms and vector fields. Normal form method helps to improve estimates of time integrals; vector fields help improve dispersive estimates. We will delicately analyze all kinds of terms showing up in the resonant system and the perturbation equations, and these estimates actually constitute most part of this paper.

\begin{remark}
Essentially, our proof mainly relies on the energy estimates in Sobolev space. Since we use the normal form transformation (integration by parts in time), in turn we must estimate the time derivatives. Since we need $\qq_{jk}$ estimate and integration by parts in $\eta$, we also require $\xi_{\ell}$ derivative estimates (Lemma \ref{wtt lemma 2}) and vector fields (Lemma \ref{wtt prelim 2}).
\end{remark}

\smallskip
\subsection{Organization of the Paper}

This paper is organized as follows: in Section\;3, we present some preliminary lemmas and a priori estimates on the solution and data; in Section\;4 and Section\;5, we discuss the nonlinear estimate for $S_i'$ and $T_i'$ norms using the Duhamel's formula without time integration; finally, in Section\;6 and Section\;7, we study the energy estimates for $S_i$ norms and the regularity estimates for $T_i$ norms.

\bigskip
\section{A Priori Estimates} 

Let
\begin{align}
    X(t;n)&:=\br{t}^{-H(n)\d},\\
    Y(k,t;n)&:=\br{t}^{-H(n+1)\d}2^{-N(n+1)k^+},\\
    Z(k;n)&:=2^{-N(n-2)k^+}.
\end{align}

\smallskip
\subsection{\texorpdfstring{$L^2$}{} Estimates of \texorpdfstring{$\gwa$}{} and \texorpdfstring{$\gkg$}{}} 

\begin{lemma}\label{wtt lemma 1}
Assume \eqref{wtt-assumption} and \eqref{wtt-assumption'} hold. Then for any $\lll\in\vvv_n$ with $0\leq n\leq N_1$,
\begin{align}
    \sup_{\lll\in\vvv_n}\bigg(\Bnm{\abs{\na}^{-\frac{1}{2}}\llgwa(t)}_{H^{N(n)}}+\bnm{\llgkg(t)}_{H^{N(n)}}\bigg)&\ls \e X(t;n),
\end{align}
and for any $0\leq n\leq N_1-1$,
\begin{align}
    \sup_{\lll\in\vvv_n, \ell\in\{1,2,3\}}\sup_{k\in\mathbb{Z}}\bigg(2^{\frac{1}{2}k}\nm{\varphi_k\p_{\xi_{\ell}}\hlgwa(t)}_{L^2_{\xi}}+2^{k^+}\Big\|\varphi_k\p_{\xi_{\ell}}\hlgkg(t)\Big\|_{L^2_{\xi}}\bigg)\ls \e Y(k,t;n).
\end{align}
\end{lemma}

\begin{proof}
This can be directly obtained from \eqref{wtt-assumption} based on the definition of $X_1$ and $X_2$ norms.
\end{proof}

\begin{remark}\label{remark 1}
Using the definition of Sobolev spaces with Fourier transform, this lemma actually yields the $L^2$ bound for $\gwa$ and $\gkg$. For any $k\in\mathbb{Z}$ and for any $0\leq n\leq N_1$,
\begin{align}
    \big\|P_k\llgwa(t)\big\|_{L^2}&\ls \ep X(t;n)2^{\frac{1}{2}k}2^{-N(n)k^+},\\
    \big\|P_k\llgkg(t)\big\|_{L^2}&\ls \ep X(t;n)2^{-N(n)k^+},
\end{align}
and for any $0\leq n\leq N_1-1$,
\begin{align}
    \bnm{\varphi_k\p_{\xi_{\ell}}\widehat{\llgwa}(t)}_{L^2_{\xi}}&\ls \ep Y(k,t;n)2^{-\frac{1}{2}k},\\
    \big\|\varphi_k\p_{\xi_{\ell}}\widehat{\llgkg}(t)\big\|_{L^2_{\xi}}&\ls \ep Y(k,t;n)2^{-k^+}.
\end{align}
\end{remark}

\begin{lemma}\label{wtt lemma 3}
Assume \eqref{wtt-assumption} and \eqref{wtt-assumption'} hold. Then for any $\lll\in\vvv_n$ with $0\leq n\leq N_1-1$ and $(k,j)\in\jj$, we have
\begin{align}
    2^{j}\big\|Q_{jk}\llgwa(t)\big\|_{L^2}&\ls \ep Y(k,t;n)2^{-\frac{1}{2}k},
    \label{wtt 08-1}\\
    2^{j}\big\|Q_{jk}\llgkg(t)\big\|_{L^2}&\ls\ep Y(k,t;n)2^{-k^+},\label{wtt 08-2}
\end{align}
and 
\begin{align}
    \big\|P_k\llgwa(t)\big\|_{L^2}&\ls\ep Y(k,t;n)2^{-\frac{1}{2}k}2^{k^-},
    \label{wtt 09-1}\\
    \big\|{P_k}\llgkg(t)\big\|_{L^2}&\ls\ep Y(k,t;n)2^{-k^+}2^{k^-}.\label{wtt 09-2}
\end{align}
\end{lemma}

\begin{proof}
\eqref{wtt 08-1}\eqref{wtt 08-2} are direct consequences of Remark \ref{remark 1} and \eqref{wtt 07}. 
Summing over $j\geq -k^-$, we get \eqref{wtt 09-1}\eqref{wtt 09-2}.
\end{proof}

\begin{lemma}\label{wtt lemma 3'}
Assume \eqref{wtt-assumption} and \eqref{wtt-assumption'} hold. Then for any $\lll\in\vvv_n$ with $0\leq n\leq N_1-1$ and $(k,j)\in\jj$, we have
\begin{align}
    \big\|P_k\llgwa(t)\big\|_{L^2}&\ls \ep\br{t}^{-H(n+1)\d} 2^{-\left(N(n)-\frac{1}{2}\right)k^+}2^{\frac{1}{2}k^-},\label{wtt 10-1} \\
    \big\|P_k\llgkg(t)\big\|_{L^2}&\ls \ep\br{t}^{-H(n+1)\d} 2^{-N(n)k^+}2^{k^-}. \label{wtt 10-2}
\end{align}
\end{lemma}

\begin{proof}
These follow from a combination of Remark \ref{remark 1} (if $k\geq 0$) and Lemma \ref{wtt lemma 3} (if $k\leq 0$). 
\end{proof}

\begin{remark}
Lemma \ref{wtt lemma 3'} is an improved version for Remark \ref{remark 1}.
In particular, ``$2^{k^-}$'' in (\ref{wtt 10-2}) is favorable when $k\leq 0$.
We will use these bounds repeatedly in the energy estimates.
\end{remark}

\smallskip
\subsection{\texorpdfstring{$L^\infty$}{} Estimates of \texorpdfstring{$\gwa$}{} and \texorpdfstring{$\gkg$}{}} 

\begin{lemma}\label{wtt lemma 6}
Assume \eqref{wtt-assumption} and \eqref{wtt-assumption'} hold. Then for any $\lll\in\vvv_n$ with $0\leq n\leq N_1-1$ and $(k,j)\in\jj$, we have
\begin{align}
    \big\|\ue^{-\ui t\lawa}\qq_{jk}\llgwa(t)\big\|_{L^{\infty}}&\ls \ep Y(k,t;n)  \min\big\{2^{-j},\,\br{t}^{-1}\big\}2^{k}, \label{wwt6-1}\\
    \big\|\ue^{-\ui t\lakg}\qq_{jk}\llgkg(t)\big\|_{L^{\infty}}&\ls\ep Y(k,t;n) \min\Big\{2^{\frac{3}{2}k},\,2^{3k^+}2^{\frac{3}{2}j}\br{t}^{-\frac{3}{2}}\Big\}2^{-j}2^{-k^+}. \label{wwt6-2}
\end{align}
\end{lemma}

\begin{proof}
Combining Lemma \ref{wtt prelim 5} and Lemma \ref{wtt lemma 3} yields these bounds.
\end{proof}

\begin{lemma}\label{wtt lemma 7}
Assume \eqref{wtt-assumption} and \eqref{wtt-assumption'} hold. Then for any $\lll\in\vvv_n$ with $n\leq N_1-1$ and $(k,j)\in\jj$, we have
\begin{align} 
    \sum_{j\geq-k^-}\Bnm{\ue^{-\ui t\lawa}\qq_{jk}\llgwa(t)}_{L^{\infty}}&\ls \ep Y(k,t;n)  \min\big\{2^{k^-},\,\ln\!\br{t}\br{t}^{-1}\big\}2^{k},
    \label{wtt7-1}\\
    \sum_{j\geq-k^-}\Big\|\ue^{-\ui t\lakg}\qq_{jk}\llgkg(t)\Big\|_{L^{\infty}}&\ls\ep Y(k,t;n) \min\Big\{2^{\frac{1}{2}k^+}2^{\frac{5}{2}k^-},\,2^{\frac{3}{2}k^+}2^{\frac{1}{2}k^-}\br{t}^{-1}\Big\}. \label{wtt7-2}
\end{align}
\end{lemma}

\begin{proof}
Summing up (\ref{wwt6-1}) in Lemma \ref{wtt lemma 6}, we have
\begin{align} \label{wtt7-pf-1}
    \sum_{j\geq-k^-}\Bnm{\ue^{-\ui t\lawa}\qq_{jk}\llgwa}_{L^{\infty}}&\ls\ep Y(k,t;n) \sum_{j\geq-k^-}\min\big\{2^{-j},\br{t}^{-1}\big\}2^{k}.
\end{align}
For the first term in the minimum of (\ref{wtt7-1}), we naturally have
\begin{align}
    \sum_{j\geq-k^-}\Bnm{\ue^{-\ui t\lawa}\qq_{jk}\llgwa}_{L^{\infty}}&\ls\ep Y(k,t;n) \sum_{j\geq-k^-}2^{-j}2^{k}\ls \ep Y(k,t;n) 2^{k^-}2^{k}.
\end{align}
For the second term in the minimum, we divide it into two cases:
\begin{itemize}
    \item 
    If $2^{-j}\geq\br{t}^{-1}$, then since $-k^-\geq0$, there are at most $\ln\br{t}$ such $j$. Hence, by (\ref{wtt7-pf-1}) we have the summation
    \begin{align}
        \sum_{2^{-j}\geq\br{t}^{-1}}\Bnm{\ue^{-\ui t\lawa}\qq_{jk}\llgwa}_{L^{\infty}}&\ls\ep Y(k,t;n) \ln\!\br{t}\br{t}^{-1}2^{k}.
    \end{align}
    \item
    If $2^{-j}\leq\br{t}^{-1}$, then we naturally have
    \begin{align}
        \sum_{2^{-j}\leq\br{t}^{-1}}\Bnm{\ue^{-\ui t\lawa}\qq_{jk}\llgwa}_{L^{\infty}}&\ls\ep Y(k,t;n)\sum_{2^{-j}\leq\br{t}^{-1}}2^{-j}2^k\ls\ep Y(k,t;n)\br{t}^{-1}2^k.
    \end{align}
\end{itemize}
Summarizing the above two cases, we arrive at the desired result (\ref{wtt7-1}).
 
On the other hand, for (\ref{wtt7-2}), we have from (\ref{wwt6-2}) that
\begin{align} \label{wtt7-pf-2}
    \sum_{j\geq-k^-}\nm{\ue^{-\ui t\lakg}\qq_{jk}\llgkg}_{L^{\infty}}&\ls\ep Y(k,t;n) \sum_{j\geq-k^-}\min\Big\{2^{\frac{3k}{2}}2^{-j}2^{-k^+},\,2^{2k^+}2^{\frac{j}{2}}\br{t}^{-\frac{3}{2}}\Big\}.
\end{align}
For the first term in the minimum of \eqref{wtt7-pf-2}, we have
\begin{align}
    \sum_{j\geq-k^-}\nm{\ue^{-\ui t\lakg}\qq_{jk}\llgkg}_{L^{\infty}}&\ls\ep Y(k,t;n) \sum_{j\geq-k^-}2^{\frac{3k}{2}}2^{-j}2^{-k^+}\ls\ep Y(k,t;n) 2^{\frac{3k}{2}}2^{k^-}2^{-k^+}
    \ls \ep Y(k,t;n) 2^{\frac{k^+}{2}}2^{\frac{5k^-}{2}}.
\end{align}
For the second term in the minimum of \eqref{wtt7-pf-2}, $j$ power is positive, so we cannot directly sum up over $j$. 
The key is to combine both bounds in (\ref{wtt7-pf-2}).
We further divide it into two cases:
\begin{itemize}
    \item 
    If $2^{j}\leq\br{t}2^k2^{-2k^+}$, then $2^{\frac{3k}{2}}2^{-j}2^{-k^+}\geq2^{2k^+}2^{\frac{j}{2}}\br{t}^{-\frac{3}{2}}$. Hence, we choose the second term in the minimum of \eqref{wtt7-pf-2} to sum up over $j$ and get
    \begin{align}
        \sum_{j\geq-k^-}2^{2k^+}2^{\frac{j}{2}}\br{t}^{-\frac{3}{2}}\leq 2^{2k^+}\br{t}^{-\frac{3}{2}}\Big(\br{t}^{\frac{1}{2}}2^{\frac{1}{2}k}2^{-k^+}\Big)\leq\br{t}^{-1} 2^{\frac{3}{2}k^+}2^{\frac{1}{2}k^-}.
    \end{align}
    \item
    If $2^{j}\geq\br{t}2^k2^{-2k^+}$, then $2^{\frac{3k}{2}}2^{-j}2^{-k^+}\leq2^{2k^+}2^{\frac{j}{2}}\br{t}^{-\frac{3}{2}}$. Hence, we choose the first term in the minimum of \eqref{wtt7-pf-2} to sum up over $j$ and get
    \begin{align}
        \sum_{j\geq-k^-}2^{\frac{3k}{2}}2^{-j}2^{-k^+}\leq 2^{\frac{3k}{2}}2^{-k^+}\Big(\br{t}^{-1}2^{-k}2^{2k^+}\Big)\leq\br{t}^{-1} 2^{\frac{3}{2}k^+}2^{\frac{1}{2}k^-}.
    \end{align}
\end{itemize}
Summarizing the above two cases, we have that 
\begin{align}
    \sum_{j\geq-k^-}\nm{\ue^{-\ui t\lakg}\qq_{jk}\llgkg}_{L^{\infty}}&\ls\br{t}^{-1}2^{\frac{3}{2}k^+}2^{\frac{1}{2}k^-}.
\end{align}
\end{proof}

\begin{lemma}\label{wtt lemma 8}
Assume \eqref{wtt-assumption} and \eqref{wtt-assumption'} hold. Then for any $\lll\in\vvv_n$ with $0\leq n\leq N_1-2$ and $(k,j)\in\jj$ with $1\leq 2^{2k^-}\!\br{t}$, we have for some sufficiently small $\sigma>0$,
\begin{align}
    \sum_{j\geq -k^-}\nm{\ue^{-\ui t\lakg}\qq_{jk}\llgkg(t)}_{L^{\infty}}&\ls\ep Y(k,t;n+1)\br{t}^{-\frac{3}{2}+\frac{\sigma}{8}}2^{4k^+}2^{-k^-(\frac{1}{2}-\frac{\sigma}{4})}.
\end{align}
\end{lemma}

\begin{proof}
Based on Lemma \ref{wtt prelim 7}, we know
\begin{align}
    \nm{\ue^{-\ui t\lakg}\qq_{jk}\llgkg}_{L^{\infty}}&\ls \br{t}^{-\frac{3}{2}}2^{5k^+}2^{-k^-}2^{\frac{j}{2}}\big(2^{2k^-}\br{t}\big)^{\frac{\sigma}{8}}\bnm{Q_{jk}\llgkg}_{H^{0,1}_{\Omega}}.
\end{align}
From Lemma \ref{wtt lemma 3}, we have
\begin{align}
    \bnm{Q_{jk}\llgkg}_{H^{0,1}_{\Omega}}\ls \ep Y(k,t;n+1)2^{-j}2^{-k^+}.
\end{align}
Combining the above two, we have
\begin{align}
    \nm{\ue^{-\ui t\lakg}\qq_{jk}\llgkg}_{L^{\infty}}&\ls \ep Y(k,t;n+1) \br{t}^{-\frac{3}{2}}2^{4k^+}2^{-k^-}2^{-\frac{j}{2}}\big(2^{2k^-}\br{t}\big)^{\frac{\sigma}{8}}.
\end{align}
Summing up over $j\geq -k^-$, we have
\begin{align}
    \sum_{j\geq -k^-}\nm{\ue^{-\ui t\lakg}\qq_{jk}\llgkg}_{L^{\infty}}&\ls \ep Y(k,t;n+1) \br{t}^{-\frac{3}{2}}2^{4k^+}2^{-\frac{k^-}{2}}\big(2^{2k^-}\br{t}\big)^{\frac{\sigma}{8}}.
\end{align}
Hence, our result naturally follows.
\end{proof}

\begin{remark}
Roughly speaking, Lemma \ref{wtt lemma 7} and Lemma \ref{wtt lemma 8} provide the dispersive estimates 
\begin{align}
    \nm{\ue^{-\ui t\lawa}\llgwa(t)}_{L^{\infty}}&\sim \br{t}^{-1},\\
    \big\|\ue^{-\ui t\lakg}\llgkg(t)\big\|_{L^{\infty}}&\sim \br{t}^{-\frac{3}{2}}.
\end{align}
This is consistent with the standard dispersive estimates for the wave equation and Klein-Gordon equations. However, there is a delicate trade-off between time decay and the weights of $k,k^+,k^-$.
\end{remark}

\begin{lemma}\label{wtt lemma 16}
Assume \eqref{wtt-assumption} and \eqref{wtt-assumption'} hold. Then for any $\lll\in\vvv_n$ with $0\leq n\leq N_1-1$, we have that for some sufficiently small $\sigma>0$,
\begin{align}
 \lnmx{\widehat{\qq_{jk}\llgwa}}&\ls\ep Y(k,t;n)2^{\frac{\sigma(j+k)}{8}}2^{-\frac{j}{2}} 2^{-\frac{3}{2}k},\\
    \lnmx{\widehat{P_k\llgwa}}&\ls \ep Y(k,t;n)2^{-k^-}2^{-k^+},\label{wtt lemma 16-1}
\end{align}
and
\begin{align}
    \Blnmx{\widehat{\qq_{jk}\llgkg}}&\ls\ep Y(k,t;n)2^{\frac{\sigma(j+k)}{8}}2^{-\frac{j}{2}}2^{-k} 2^{-k^+},\\
    \Blnmx{\widehat{P_k\llgkg}}&\ls\ep Y(k,t;n)2^{-\frac{1}{2}k^-}2^{-k^+}.\label{wtt lemma 16-2}
\end{align}
\end{lemma}

\begin{proof}
Based on \eqref{wtt 16} in Lemma \ref{wtt prelim 2}, we know
\begin{align}
    \lnmx{\widehat{\qq_{jk}\llgwa}}\ls 2^{\frac{j}{2}}2^{-k}2^{\frac{\sigma(j+k)}{8}}\bnm{Q_{jk}\llgwa}_{H_{\Omega}^{0,1}}.
\end{align}
From Lemma \ref{wtt lemma 3}, we know 
\begin{align}
    \bnm{Q_{jk}\llgwa}_{H_{\Omega}^{0,1}}&\ls \ep Y(k,t;n) 2^{-j}2^{-\frac{k}{2}}.
\end{align}
Combining the above two bounds, we have
\begin{align}
    \lnmx{\widehat{\qq_{jk}\llgwa}}\ls\ep Y(k,t;n)2^{\frac{\sigma(j+k)}{8}}2^{-\frac{j}{2}}2^{-\frac{3}{2}k}.
\end{align}
Summing up over $j\geq -k^-$, we have 
\begin{align}
    \lnmx{\widehat{P_k\llgwa}}\ls\ep Y(k,t;n)2^{(\frac{1}{2}-\frac{\sigma}{8})k^-}2^{-(\frac{3}{2}-\frac{\sigma}{8})k}.
\end{align}
For $k\geq0$, $k=k^+$ and for $k\leq 0$, $k=k^-$. Hence, we arrive at (\ref{wtt lemma 16-1}).

Similarly, based on \eqref{wtt 16} in Lemma \ref{wtt prelim 2}, we know
\begin{align}
    \Blnmx{\widehat{\qq_{jk}\llgkg}}\ls 2^{\frac{j}{2}}2^{-k}2^{\frac{\sigma(j+k)}{8}}\bnm{Q_{jk}\llgkg}_{H_{\Omega}^{0,1}}.
\end{align}
From Lemma \ref{wtt lemma 3}, we know 
\begin{align}
    \bnm{Q_{jk}\llgkg}_{H_{\Omega}^{0,1}}&\ls \ep Y(k,t;n) 2^{-j}2^{-k^+}.
\end{align}
Combining the above two bounds, we have
\begin{align}
    \Blnmx{\widehat{\qq_{jk}\llgkg}}\ls\ep Y(k,t;n)2^{\frac{\sigma(j+k)}{8}}2^{-\frac{j}{2}}2^{-k}2^{-k^+}.
\end{align}
Summing up over $j\geq -k^-$, we have 
\begin{align}
    \Blnmx{\widehat{P_k\llgkg}}\ls\ep Y(k,t;n)2^{(\frac{1}{2}-\frac{\sigma}{8})k^-}2^{-(1-\frac{\sigma}{8})k}2^{-k^+}.
\end{align}
For $k\geq0$, $k=k^+$ and for $k\leq 0$, $k=k^-$. Hence, we arrive at (\ref{wtt lemma 16-2}).
\end{proof}

\smallskip
\subsection{Estimates of \texorpdfstring{$V_{\infty}^{wa}$}{} and \texorpdfstring{$V_{\infty}^{kg}$}{}} 

\begin{lemma}\label{vtt lemma 1}
Assume \eqref{wtt-assumption} and \eqref{wtt-assumption'} hold. Then for any $\lll\in\vvv_n$ with $0\leq n\leq N_1$,
\begin{align}
    \tnm{P_kV_{\lll,\infty}^{wa}}&\ls\ep Z(k;n-1) 2^{\frac{1}{2}k},\\
    \big\|P_kV_{\lll,\infty}^{kg}\big\|_{L^2}&\ls \ep Z(k;n-1),
\end{align}
and for any $0\leq n\leq N_1-1$,
\begin{align}
    \nm{\varphi_k\p_{\xi_{\ell}}\widehat{V_{\lll,\infty}^{wa}}}_{L^2_{\xi}}&\ls \ep Z(k;n)2^{-\frac{1}{2}k},\\
    \Big\|\varphi_k\p_{\xi_{\ell}}\widehat{V_{\lll,\infty}^{kg}}\Big\|_{L^2_{\xi}}&\ls \ep Z(k;n)2^{-k^+}.
\end{align}
\end{lemma}
\begin{proof}
Based on the definition of $Y_1$ and $Y_2$ norms, this is similar to the proof of Lemma \ref{wtt lemma 1} and Remark \ref{remark 1}.
\end{proof}

\begin{lemma}\label{vtt lemma 3}
Assume \eqref{wtt-assumption} and \eqref{wtt-assumption'} hold. Then for any $\lll\in\vvv_n$ with $0\leq n\leq N_1-1$ and $(k,j)\in\jj$, we have
\begin{align}
    2^{j}\nm{Q_{jk}V_{\lll,\infty}^{wa}}_{L^2}&\ls \ep Z(k;n)2^{-\frac{1}{2}k},\\
    2^{j}\big\|Q_{jk}V_{\lll,\infty}^{kg}\big\|_{L^2}&\ls\ep Z(k;n)2^{-k^+},
\end{align}
and 
\begin{align}
    \nm{P_kV_{\lll,\infty}^{wa}}_{L^2}&\ls\ep Z(k;n)2^{-\frac{1}{2}k}2^{k^-},\\
    \big\|{P_k}V_{\lll,\infty}^{kg}\big\|_{L^2}&\ls\ep Z(k;n)2^{-k^+}2^{k^-}.
\end{align}
\end{lemma}
\begin{proof}
See the proof of Lemma \ref{wtt lemma 3}.
\end{proof}

\begin{lemma}\label{vtt lemma 3'}
Assume \eqref{wtt-assumption} and \eqref{wtt-assumption'} hold. Then for any $\lll\in\vvv_n$ with $0\leq n\leq N_1-1$ and $(k,j)\in\jj$, we have
\begin{align}
    \tnm{P_kV_{\lll,\infty}^{wa}}&\ls \ep Z(k;n)2^{\frac{1}{2}k^-},\label{vtt 10-1} \\
    \big\|P_kV_{\lll,\infty}^{kg}\big\|_{L^2}&\ls \ep Z(k;n)2^{k^-}. \label{vtt 10-2}
\end{align}
\end{lemma}
\begin{proof}
See the proof of Lemma \ref{wtt lemma 3'}.
\end{proof}

\begin{lemma}\label{vtt lemma 6}
Assume \eqref{wtt-assumption} and \eqref{wtt-assumption'} hold. Then for any $\lll\in\vvv_n$ with $0\leq n\leq N_1-1$ and $(k,j)\in\jj$, we have
\begin{align}
    \nm{\ue^{-\ui t\lawa}\qq_{jk}V_{\lll,\infty}^{wa}}_{L^{\infty}}&\ls \ep Z(k;n) \min\Big\{2^{-j},\,\br{t}^{-1}\Big\}2^{k},\\
    \big\|\ue^{-\ui t\lakg}\qq_{jk}V_{\lll,\infty}^{kg}\big\|_{L^{\infty}}&\ls\ep Z(k;n) \min\Big\{2^{\frac{3}{2}k},\,2^{3k^+}2^{\frac{3}{2}j}\br{t}^{-\frac{3}{2}}\Big\}2^{-j}2^{-k^+}.
\end{align}
\end{lemma}
\begin{proof}
See the proof of Lemma \ref{wtt lemma 6}.
\end{proof}

\begin{lemma}\label{vtt lemma 7}
Assume \eqref{wtt-assumption} and \eqref{wtt-assumption'} hold. Then for any $\lll\in\vvv_n$ with $0\leq n\leq N_1-1$ and $(k,j)\in\jj$, we have
\begin{align}
    \sum_{j\geq-k^-}\Bnm{\ue^{-\ui t\lawa}\qq_{jk}V_{\lll,\infty}^{wa}}_{L^{\infty}}&\ls \ep Z(k;n) \min\Big\{2^{k^-},\ln\!\br{t}\br{t}^{-1}\Big\}2^{k},
    \label{wtt 7-1}\\
    \sum_{j\geq-k^-}\Big\|\ue^{-\ui t\lakg}\qq_{jk}V_{\lll,\infty}^{kg}\Big\|_{L^{\infty}}&\ls\ep Z(k;n) \min\Big\{2^{\frac{1}{2}k^+}2^{\frac{5}{2}k^-},2^{\frac{3}{2}k^+}2^{\frac{1}{2}k^-}\br{t}^{-1}\Big\}. \label{wtt 7-2}
\end{align}
\end{lemma}
\begin{proof}
See the proof of Lemma \ref{wtt lemma 7}.
\end{proof}

\begin{remark}
Note that we have $\ln\!\br{t}$ in the estimate. Unlike the $\gwa$ estimate in Lemma \ref{wtt lemma 7} which has $\br{t}^{-H(n+1)}$ decay to kill the logarithmic term, this is a major troublemaker for later proofs since it does not give exact one-order decay. We need an improved version (see Lemma \ref{vtt lemma 8'}). 
\end{remark}

\begin{lemma}\label{vtt lemma 8}
Assume \eqref{wtt-assumption} and \eqref{wtt-assumption'} hold. Then for any $\lll\in\vvv_n$ with $0\leq n\leq N_1-2$ and $(k,j)\in\jj$ with $1\leq 2^{2k^-}\!\br{t}$, we have for some sufficiently small $\sigma>0$,
\begin{align}
    \sum_{j\geq -k^-}\nm{\ue^{-\ui t\lakg}\qq_{jk}V_{\lll,\infty}^{kg}}_{L^{\infty}}&\ls\ep Z(k;n+1)\br{t}^{-\frac{3}{2}+\frac{\sigma}{8}}2^{4k^+}2^{-k^-(\frac{1}{2}-\frac{\sigma}{4})}.
\end{align}
\end{lemma}
\begin{proof}
See the proof of Lemma \ref{wtt lemma 8}.
\end{proof}

\begin{lemma}\label{vtt lemma 16}
Assume \eqref{wtt-assumption} and \eqref{wtt-assumption'} hold. Then for any $\lll\in\vvv_n$ with $0\leq n\leq N_1$, we have that for some sufficiently small $\sigma>0$,
\begin{align}
 \lnmx{\widehat{\qq_{jk}V_{\lll,\infty}^{wa}}}&\ls\ep Z(k;n)2^{\frac{\sigma(j+k)}{8}}2^{-\frac{j}{2}} 2^{-\frac{3}{2}k},\\
    \lnmx{\widehat{P_kV_{\lll,\infty}^{wa}}}&\ls \ep Z(k;n)2^{-k^-}2^{-k^+},\label{vtt lemma 16-1}
\end{align}
and
\begin{align}
    \Blnmx{\widehat{\qq_{jk}V_{\lll,\infty}^{kg}}}&\ls\ep Z(k;n)2^{\frac{\sigma(j+k)}{8}}2^{-\frac{j}{2}}2^{-k} 2^{-k^+},\\
    \Blnmx{\widehat{P_kV_{\lll,\infty}^{kg}}}&\ls\ep Z(k;n)2^{-\frac{1}{2}k^-}2^{-k^+}.\label{vtt lemma 16-2}
\end{align}
\end{lemma}

\begin{proof}
See the proof of Lemma \ref{wtt lemma 16}.
\end{proof}

\begin{lemma}\label{vtt lemma 8'}
Assume \eqref{wtt-assumption} and \eqref{wtt-assumption'} hold. Then for any $\lll\in\vvv_n$ with $0\leq n\leq N_1-2$ and $(k,j)\in\jj$, we have
\begin{align}
    \sum_{j\geq -k^-}\Bnm{\ue^{-\ui t\lawa}\qq_{jk}V_{\lll,\infty}^{wa}}_{L^{\infty}}&\ls\ep Z(k;n+1)\br{t}^{-1}2^{k}.
\end{align}
\end{lemma}

\begin{proof}
Let $2^{J_0}=\br{t}^{\frac{1}{2}}2^{-\frac{k}{2}}$. We divide it into the following cases:
\begin{itemize}
    \item
    If $2^{k^-}\leq \br{t}^{-1}$, then by Lemma \ref{vtt lemma 7}, we have
    \begin{align}
       \sum_{j\geq-k^-} \Bnm{\ue^{-\ui t\lawa}\qq_{jk}V_{\lll,\infty}^{wa}}_{L^{\infty}}\ls \ep Z(k;n) 2^{k^-}2^{k}\ls \ep Z(k;n) \br{t}^{-1}2^{k}.
    \end{align}
    \item 
    If $2^{k^-}\geq \br{t}^{-1}$ and $j\leq J_0$, by (\ref{wtt prelim 6-2}) in Lemma \ref{wtt prelim 6}, we have
    \begin{align}
        \Bnm{\ue^{-\ui t\lawa}\qq_{\leq J_0k}V_{\lll,\infty}^{wa}}_{L^{\infty}}&\ls \br{t}^{-1}2^{2k}\nm{\widehat{Q_{\leq J_0k}V_{\lll,\infty}^{wa}}}_{L^{\infty}_{\xi}}.
    \end{align}
    Based on Lemma \ref{wtt prelim 14} and Lemma \ref{vtt lemma 3'}, we have
    \begin{align}
        \nm{\widehat{Q_{\leq J_0k}V_{\lll,\infty}^{wa}}-\widehat{\qq_{\leq J_0k}V_{\lll,\infty}^{wa}}}_{L^{\infty}_{\xi}}\ls& \sum_{j\geq -k^-}2^{-\frac{5}{2}j}2^{-4k}\tnm{P_kV_{\lll,\infty}^{wa}}
        \ls\; 2^{-\frac{3}{2}k^-}\tnm{P_kV_{\lll,\infty}^{wa}}\ls \ep Z(k;n)2^{-k^-}.
    \end{align}
    Also, based on the proof of Lemma \ref{vtt lemma 16}, we know
    \begin{align}
        \nm{\widehat{\qq_{\leq J_0k}V_{\lll,\infty}^{wa}}}_{L^{\infty}_{\xi}}\ls\lnmx{\widehat{P_kV_{\lll,\infty}^{wa}}}&\ls \ep Z(k;n)2^{-k^-}.
    \end{align}
    In total, we have
    \begin{align}
        \nm{\widehat{Q_{\leq J_0k}V_{\lll,\infty}^{wa}}}_{L^{\infty}_{\xi}}\ls \ep Z(k;n)2^{-k^-}.
    \end{align}
    Hence, we have
    \begin{align}
        \nm{\ue^{-\ui t\lawa}\qq_{\leq J_0k}V_{\lll,\infty}^{wa}}_{L^{\infty}}&\ls
        \ep Z(k;n)\br{t}^{-1}2^{2k}2^{-k^-}.
    \end{align}
    For $k\geq0$, $k^-=0$, so the desired result naturally holds. For $k\leq 0$, $k=k^-$, so we know $2^{2k}2^{-k^-}=2^{k^-}$. Hence, our desired result is true.
    \item
    If $2^{k^-}\geq \br{t}^{-1}$ and $j\geq J_0$ and $2^k\ls \br{t}$ and $2^j\leq \br{t}$, by (\ref{wtt prelim 6-1}) in Lemma \ref{wtt prelim 6}, we have
    \begin{align}
        \nm{\ue^{-\ui t\lawa}\qq_{\geq J_0k}V_{\lll,\infty}^{wa}}_{L^{\infty}}&\ls\sum_{j\geq J_0} \br{t}^{-1}2^{\frac{k}{2}}\big(1+2^{k}\br{t}\big)^{\frac{\sigma}{8}}\nm{Q_{jk}V_{\lll,\infty}^{wa}}_{H^{0,1}_{\Omega}}.
    \end{align}
    Lemma \ref{vtt lemma 3} implies
    \begin{align}
        \nm{Q_{jk}V_{\lll,\infty}^{wa}}_{H^{0,1}_{\Omega}}\ls \ep Z(k;n+1) 2^{-j}2^{-\frac{k}{2}}.
    \end{align}
    Combining them together, we have
    \begin{align}
        \nm{\ue^{-\ui t\lawa}\qq_{\geq J_0k}V_{\lll,\infty}^{wa}}_{L^{\infty}}&\ls\sum_{j\geq J_0}\ep Z(k;n+1) \br{t}^{-1}2^{-j}\big(1+2^{k}\br{t}\big)^{\frac{\sigma}{8}}\\
        &\ls\sum_{j\geq J_0}\ep Z(k;n+1) \br{t}^{-1-\frac{\sigma}{8}}2^{\frac{\sigma}{8}k}2^{-j}\ls\ep Z(k;n+1) \br{t}^{-\frac{3}{2}-\frac{\sigma}{8}}2^{(\frac{1}{2}+\frac{\sigma}{8})k}.\no
    \end{align}
    Then for $\sigma>0$ small, our result naturally holds.
    \item
    If $2^{k^-}\geq \br{t}^{-1}$ and $j\geq J_0$ and $2^k\ls \br{t}$ and $2^j\geq \br{t}$, then $2^{-j}\leq \br{t}^{-1}$.
    Using Lemma \ref{vtt lemma 6}, we have
    \begin{align}
        \nm{\ue^{-\ui t\lawa}\qq_{\geq J_0k}V_{\lll,\infty}^{wa}}_{L^{\infty}}&\ls \e Z(k;n)\sum_{2^{-j}\leq \br{t}^{-1}}2^{-j}2^k\ls \ep Z(k;n)\br{t}^{-1}2^k.
    \end{align}
    \item
    If $2^{k^-}\geq \br{t}^{-1}$ and $j\geq J_0$ and $2^k\gs \br{t}$, we know $k\geq0$ and $k=k^+$. Hence, by Sobolev embedding theorem, we have
    \begin{align}
        \nm{\ue^{-\ui t\lawa}\qq_{\geq J_0k}V_{\lll,\infty}^{wa}}_{L^{\infty}}&\ls\nm{P_kV_{\lll,\infty}^{wa}}_{L^{\infty}} \ls\nm{P_kV_{\lll,\infty}^{wa}}_{H^2}\ls \ep 2^{2k^+}\bnm{\abs{\nabla}^{-\frac{1}{2}}V_{\lll,\infty}^{wa}}_{H^{N(0)}}\ls \ep Z(k;n) \br{t}^{-N_0}.
    \end{align}
    This suffices to close the proof.
\end{itemize}
\end{proof}
\begin{remark}
Compared with Lemma \ref{vtt lemma 7}, the key improvement here is that we get rid of $\ln\!\br{t}$. This is crucial for later estimates.
\end{remark}

\smallskip
\subsection{Estimates of \texorpdfstring{$\hhh_{\infty}$}{} and \texorpdfstring{$\hhf_{\infty}$}{}} 

The low-frequency part $\hhh_{\infty}(t,\xi)$ and $\hhf_{\infty}(t,\xi)$ of the wave component are defined in \eqref{h_infty} and \eqref{H_infty}.
Note that the low-frequency cutoff function in $\hhi$ and $\hhfi$ implies that $\abs{\xi}\ls\br{t}^{-\frac{7}{8}}\leq 1$. Hence, all estimates in this subsection do not have $k^+$ part.

\begin{lemma}\label{vtt lemma 9}
Assume \eqref{wtt-assumption} and \eqref{wtt-assumption'} hold. Then we have 
\begin{align}
\bnm{\varphi_k\hhh_{\infty}}_{L^2_\xi}\ls \ep^2\br{t}^{-1}2^{\frac{1}{2}k^-},
\end{align}
and 
\begin{align}
    \blnmx{\varphi_k\hhh_{\infty}}&\ls \ep^2 \br{t}^{-1}2^{-k}.
\end{align}
Also, we have
\begin{align}
\bnm{\varphi_k\p_{\xi_{\ell}}\hhh_{\infty}}_{L^2_\xi}\ls \ep^22^{\frac{1}{2}k^-},
\end{align}
and 
\begin{align}
    \blnmx{\varphi_k\p_{\xi_{\ell}}\hhh_{\infty}}&\ls \ep^2 2^{-k}.
\end{align}
\end{lemma}

\begin{proof}
Based on \eqref{h_infty}, we know
\begin{align}
    \hhi(t,\xi)\simeq \varphi_{\leq0}\big(\xi\br{t}^{\frac{7}{8}}\big)\!\cdot\!\int_{\mathbb{R}^3} \ue^{\ui t\left(\abs{\xi}-\frac{\xi\cdot\eta}{\br{\eta}}\right)} \big|\widehat{V^{kg}_{\infty}}(\eta)\big|^2\dd\eta.
\end{align}
Let $\psi(\xi,\eta):=\abs{\xi}-\frac{\xi\cdot\eta}{\br{\eta}}$. 
Then $\nabla_{\eta}\psi=\frac{\eta(\xi\cdot\eta)-\xi\br{\eta}^2}{\br{\eta}^3}$ and $\nabla^2_{\eta}\psi=\frac{6(\xi\cdot\eta)\br{\eta}^2-3(\xi\cdot\eta)\abs{\eta}^2}{\br{\eta}^5}$, which yields $\abs{\nabla_{\eta}\psi}\simeq \abs{\xi}\br{\eta}^{-1}$ and $\abs{\nabla^2_{\eta}\psi}\simeq \abs{\xi}\br{\eta}^{-2}$.
Using integration by parts in $\eta$ (see \eqref{resonance-space}) and estimates in Lemma \ref{vtt lemma 1}, we obtain
\begin{align}
    \blnmx{\varphi_k\hhh_{\infty}}
    &\ls\br{t}^{-1}2^{-k}\big\|\widehat{V^{kg}_{\infty}}\big\|_{L^2_{\eta}}\big\|\widehat{V^{kg}_{\infty}}\big\|_{H^1_{\eta}}
    \ls\ep^22^{-k}\br{t}^{-1}.
\end{align}
Also, by H\"{o}lder's inequality and taking into account the measure of the support of $\varphi_k$, we have
\begin{align}
    \bnm{\varphi_k\hhh_{\infty}}_{L^2_\xi}&\ls2^{\frac{3}{2}k}\blnmx{\varphi_k\hhh_{\infty}}
    \ls\ep^22^{\frac{1}{2}k}\br{t}^{-1}.
\end{align}

Noting that 
\begin{align}
    \p_{\xi_{\ell}}\hhi(t,\xi)\simeq\br{t}\varphi_{\leq0}\big(\xi\br{t}^{\frac{7}{8}}\big)\!\cdot\!\int_{\mathbb{R}^3} \ue^{\ui t\left(\abs{\xi}-\frac{\xi\cdot\eta}{\br{\eta}}\right)} \big|\widehat{V^{kg}_{\infty}}(\eta)\big|^2\dd\eta,
\end{align}
the $\p_{\xi_{\ell}}\hhi$ estimates follow. 
\end{proof}

\begin{lemma}\label{vtt lemma 10}
Assume \eqref{wtt-assumption} and \eqref{wtt-assumption'} hold. Then we have
\begin{align}
    \bnm{\varphi_k\hhf_{\infty}}_{L^2_\xi}\ls \ep^22^{\frac{1}{2}k^-},
\end{align}
and 
\begin{align}
    \blnmx{\varphi_k\hhf_{\infty}}&\ls \ep^2 2^{-k}.
\end{align}
Also, we have
\begin{align}
    \bnm{\varphi_k\p_{\xi_{\ell}}\hhf_{\infty}}_{L^2_\xi}\ls \ep^22^{-\frac{1}{2}k^-},
\end{align}
and 
\begin{align}
    \blnmx{\varphi_k\p_{\xi_{\ell}}\hhf_{\infty}}&\ls \ep^2 2^{-2k}.
\end{align}
\end{lemma}

\begin{proof}
We first use integration in time (see \eqref{resonance-time}) to obtain
\begin{align}
    \hhf_{\infty}(t,\xi)&\simeq \int_0^t\varphi_{\leq0}\big(\xi\br{s}^{\frac{7}{8}}\big) \int_{\mathbb{R}^3}\ue^{\ui s\left(\abs{\xi}-\frac{\xi\cdot\eta}{\br{\eta}}\right)} \big|\widehat{V^{kg}_{\infty}}(\eta)\big|^2 \dd\eta\ud s\simeq \int_{\mathbb{R}^3}\frac{1}{\ui\big(\abs{\xi}-\frac{\xi\cdot\eta}{\br{\eta}}\big)}\Big(\ue^{\ui t\left(\abs{\xi}-\frac{\xi\cdot\eta}{\br{\eta}}\right)}-1\Big)\big|\widehat{V^{kg}_{\infty}}(\eta)\big|^2\dd\eta.
\end{align}
Here the cutoff $\varphi_{\leq0}\big(\xi\br{s}^{\frac{7}{8}}\big)$ only sets an upper bound $\abs{\xi}^{-\frac{8}{7}}$ for the time integration limit. It will not intervene the integral estimate.  
Notice that
\begin{align}
    \abs{\frac{1}{\ui\big(\abs{\xi}-\frac{\xi\cdot\eta}{\br{\eta}}\big)}\Big(\ue^{\ui t\left(\abs{\xi}-\frac{\xi\cdot\eta}{\br{\eta}}\right)}-1\Big)}\ls \abs{\xi}^{-1}\br{\eta}.
\end{align}
Hence, using Lemma \ref{vtt lemma 1}, we obtain
\begin{align}
    \blnmx{\varphi_k\hhf_{\infty}}\ls2^{-k}\int_{\r^3}\br{\eta}\big|\widehat{V^{kg}_{\infty}}(\eta)\big|^2\ud\eta
    \ls\ep^22^{-k},
\end{align}
and
\begin{align}
    \bnm{\varphi_k\hhf_{\infty}}_{L^2_\xi}&\ls2^{\frac{3}{2}k}\blnmx{\varphi_k\hhf_{\infty}}
    \ls\ep^22^{\frac{1}{2}k^-}.
\end{align}

On the other hand, note that
\begin{align}
    \p_{\xi_{\ell}}\hhf_{\infty}&\simeq \int_0^t\varphi_{\leq0}\big(\xi\br{s}^{\frac{7}{8}}\big) \int_{\mathbb{R}^3}\br{s}\ue^{\ui s\left(\abs{\xi}-\frac{\xi\cdot\eta}{\br{\eta}}\right)} \big|\widehat{V^{kg}_{\infty}}(\eta)\big|^2 \dd\eta\ud s.
\end{align}
The key is to handle the extra $\br{s}$. Integration by parts in $\eta$ as in the proof of Lemma \ref{vtt lemma 9} creates $\br{s}^{-1}\abs{\xi}^{-1}$ (thus the net effect is to lose $\abs{\xi}^{-1}$). 
We then apply the above argument to obtain
\begin{align}
    \blnmx{\varphi_k\p_{\xi_{\ell}}\hhf_{\infty}}\ls2^{-2k}\big\|\widehat{V^{kg}_{\infty}}\big\|_{H^1_\eta}^2
    \ls\ep^22^{-2k},
\end{align}
and 
\begin{align}
    \bnm{\varphi_k\p_{\xi_{\ell}}\hhf_{\infty}}_{L^2_\xi}&\ls2^{\frac{3}{2}k}\blnmx{\varphi_k\p_{\xi_{\ell}}\hhf_{\infty}}
    \ls\ep^22^{-\frac{1}{2}k^-}.
\end{align}
\end{proof}

\begin{lemma}\label{vtt lemma 11}
Assume \eqref{wtt-assumption} and \eqref{wtt-assumption'} hold. Then for $(k,j)\in\jj$, we have
\begin{align}
    \Blnm{\ue^{-\ui t\lawa}\qq_{jk} \big(\mathscr{F}^{-1}\hhf_{\infty}\big)}&\ls \ep^2 \min\Big\{2^{-j},\br{t}^{-1}\Big\}2^{k},
\end{align}
and 
\begin{align}
    \sum_{j\geq-k^-}\Blnm{\ue^{-\ui t\lawa}\qq_{jk}\big(\mathscr{F}^{-1}\hhf_{\infty}\big)}&\ls \ep^2 \min\Big\{2^{k^-},\,\ln\!\br{t}\br{t}^{-1}\Big\}2^{k}
\end{align}
\end{lemma}
\begin{proof}
It is similar to the proof of Lemma \ref{vtt lemma 6} and Lemma \ref{vtt lemma 7} using Lemma \ref{vtt lemma 10}. 
\end{proof}

\begin{lemma}\label{vtt lemma 12}
Assume \eqref{wtt-assumption} and \eqref{wtt-assumption'} hold. Then we have
\begin{align}
    \sum_{j\geq-k^-}\Blnm{\ue^{-\ui t\lawa}\qq_{jk}\big(\mathscr{F}^{-1}\hhf_{\infty}\big)}&\ls \ep^2 \br{t}^{-1}2^{k}.
\end{align}
\end{lemma}
\begin{proof}
It is similar to the proof of Lemma \ref{vtt lemma 8'} using Lemma \ref{vtt lemma 10}.
\end{proof}

\smallskip
\subsection{Estimates of \texorpdfstring{$\cci$}{} and \texorpdfstring{$\ddi$}{}} 

The definitions of $\cci(t,\xi)$ and $\ddi(t,\xi)$ are given in \eqref{C_infty}  and \eqref{D_infty}.

\begin{lemma}\label{vtt lemma 13}
Assume \eqref{wtt-assumption} and \eqref{wtt-assumption'} hold. Then we have
\begin{align}
    \bnm{\varphi_k\ccc_{\infty}}_{L^{\infty}_{\xi}}\ls \ep2^{3k}\br{t}^{-1}\ln\!\br{t},
\end{align}
and
\begin{align}
    \bnm{\varphi_k\p_{\xi_{\ell}}\ccc_{\infty}}_{L^{\infty}_{\xi}}\ls \ep2^{5k}\br{t}^{-1}\ln\!\br{t}. \label{panghu}
\end{align}
\end{lemma}

\begin{proof}
Recall that
\begin{align}
    \cci(t,\xi)\simeq \frac{\abs{\xi}^2}{\br{\xi}}\Im\left\{\int_{\mathbb{R}^3} \text{\large$\ue$}^{\ui t\left(\frac{\xi\cdot\eta}{\br{\xi}}-\abs{\eta}\right)}\frac{1}{\abs{\eta}}H_{\infty}(t,\eta)\,\dd\eta\right\},
\end{align}
where
\begin{align}
    H_\infty(t,\xi) &= \varphi_{\leq0}\big(\xi\br{t}^{\frac{7}{8}}\big)\!\cdot\!
    \Big\{\widehat{V^{wa}_{\infty}}(\xi) + \hhfi(t,\xi)\Big\}.
\end{align}
We will focus on $\hhfi$ part, and the similar techniques can apply to $\widehat{V^{wa}_{\infty}}$ part. 
We look at
\begin{align}
    \int_{\mathbb{R}^3} \text{\large$\ue$}^{\ui t\left(\frac{\xi\cdot\eta}{\br{\xi}}-\abs{\eta}\right)}\frac{1}{\abs{\eta}}\varphi_{\leq0}\big(\eta\br{t}^{\frac{7}{8}}\big)\hhfi(t,\eta)\,\dd\eta
    \,\simeq \int_{\abs{\eta}\ls \br{t}^{-\frac{7}{8}}} \text{\large$\ue$}^{\ui t\left(\frac{\xi\cdot\eta}{\br{\xi}}-\abs{\eta}\right)}\frac{1}{\abs{\eta}} \hhf_\infty(t,\eta)\,\dd\eta\,.
\end{align}
From Lemma \ref{vtt lemma 10}, we know 
\begin{align}
    \babs{\hhfi(t,\xi)}&\ls \ep^2\abs{\xi}^{-1}, \\
    \babs{\p_{\xi_{\ell}}\hhf_{\infty}(t,\xi)}&\ls \ep^2 \abs{\xi}^{-2}.
\end{align}
For $\abs{\eta}\ls\br{t}^{-1}$, we directly integrate to obtain
\begin{align}
    \abs{\int_{\abs{\eta}\ls\br{t}^{-1}} \text{\large$\ue$}^{\ui t\left(\frac{\xi\cdot\eta}{\br{\xi}}-\abs{\eta}\right)}\frac{1}{\abs{\eta}} \hhf_\infty(t,\eta)\,\dd\eta}&\ls\,\ep^2 \int_{\abs{\eta}\ls\br{t}^{-1}} \abs{\eta}^{-2} \,\dd\eta \,\ls\,\ep^2\br{t}^{-1}. 
\end{align}
For $\br{t}^{-1}\ls \abs{\eta}\ls \br{t}^{-\frac{7}{8}}$, 
we use integration by parts in $\eta$ following \eqref{resonance-space} with $\psi(\xi,\eta):=\frac{\xi\cdot\eta}{\br{\xi}}-\abs{\eta}$.
Since $\abs{\nabla_{\eta}\psi}\simeq \abs{\xi}^{-1}$ and $\abs{\nabla^2_{\eta}\psi}\simeq \abs{\eta}^{-1}$, we get
\begin{align}
    &\abs{\int_{\br{t}^{-1}\ls \abs{\eta}\ls \br{t}^{-\frac{7}{8}}} \text{\large$\ue$}^{\ui t\left(\frac{\xi\cdot\eta}{\br{\xi}}-\abs{\eta}\right)}\frac{1}{\abs{\eta}} \hhf_\infty(t,\eta)\,\dd\eta}\ls\;\ep^2\abs{\xi}^2\br{t}^{-1}\!\int_{\br{t}^{-1}\ls \abs{\eta}\ls \br{t}^{-\frac{7}{8}}}\abs{\eta}^{-3}\ud\eta\,\ls\, \ep^2\abs{\xi}^2\br{t}^{-1}\ln\!\br{t}.
\end{align}
Therefore, from the definition of $\cci$, we have
\begin{align}
\babs{\cci(t,\xi)}\ls \ep^2\abs{\xi}^3\br{t}^{-1}\ln\!\br{t}.
\end{align}

Next we turn to the bound of $\nabla_{\xi}\cci$. If $\nabla_{\xi}$ hits $\frac{\xi}{\br{\xi}^2}$, then following the proof of Lemma\;\ref{vtt lemma 13}, we get the desired result. Hence, we focus on the case where $\nabla_{\xi}$ hits the integral:
\begin{align}
    \nabla_{\xi}\int_{\mathbb{R}^3} \text{\large$\ue$}^{\ui t\left(\frac{\xi\cdot\eta}{\br{\xi}}-\abs{\eta}\right)}\frac{1}{\abs{\eta}} \hhf_\infty(t,\eta)\,\dd\eta&\simeq\int_{\mathbb{R}^3}\abs{\br{t}\eta} \text{\large$\ue$}^{\ui t\left(\frac{\xi\cdot\eta}{\br{\xi}}-\abs{\eta}\right)}\frac{1}{\abs{\eta}} \hhf_\infty(t,\eta)\,\dd\eta\simeq \int_{\mathbb{R}^3}\br{t} \text{\large$\ue$}^{\ui t\left(\frac{\xi\cdot\eta}{\br{\xi}}-\abs{\eta}\right)}\hhf_\infty(t,\eta)\,\dd\eta.
\end{align}
Note that now we do not have the singularity of $\abs{\eta}^{-1}$ anymore. 
Using again integration by parts in $\eta$, we obtain
\begin{align}
    &\int_{\mathbb{R}^3}\br{t} \text{\large$\ue$}^{\ui t\left(\frac{\xi\cdot\eta}{\br{\xi}}-\abs{\eta}\right)}\hhf_\infty(t,\eta)\,\dd\eta
    \simeq \, \abs{\xi}^2\int_{\abs{\eta}\ls \br{t}^{-\frac{7}{8}}} \text{\large$\ue$}^{\ui t\left(\frac{\xi\cdot\eta}{\br{\xi}}-\abs{\eta}\right)}\frac{1}{\abs{\eta}} \hhf_\infty(t,\eta)\,\dd\eta \,+\, \abs{\xi}^2\int_{\abs{\eta}\ls \br{t}^{-\frac{7}{8}}} \text{\large$\ue$}^{\ui t\left(\frac{\xi\cdot\eta}{\br{\xi}}-\abs{\eta}\right)}\nabla_{\eta}\hhf_\infty(t,\eta)\,\dd\eta.
\end{align}
Then for both terms, we may follow the above discussion for $\cci$ estimate to get
\begin{align}
    \abs{\int_{\abs{\eta}\ls\br{t}^{-\frac{7}{8}}}\br{t} \text{\large$\ue$}^{\ui t\left(\frac{\xi\cdot\eta}{\br{\xi}}-\abs{\eta}\right)} \hhf_\infty(t,\eta)\,\dd\eta}& \ls\ep^2\abs{\xi}^5\br{t}^{-1}\ln\!\br{t}, 
\end{align}
which concludes the proof of (\ref{panghu}).
\end{proof}

\begin{lemma}\label{vtt lemma 13'}
Assume \eqref{wtt-assumption} and \eqref{wtt-assumption'} hold. Then we have
\begin{align}
    \bnm{\varphi_k\ddi}_{L^{\infty}_{\xi}}\ls \ep2^{3k}\big(\ln\!\br{t}\big)^2,
\end{align}
and
\begin{align}
    \bnm{\varphi_k\p_{\xi_{\ell}}\ddi}_{L^{\infty}_{\xi}}\ls \ep2^{5k}\big(\ln\!\br{t}\big)^2.
\end{align}
\end{lemma}

\begin{proof}
It follows from taking time integral of Lemma \ref{vtt lemma 13}.
\end{proof}

\begin{lemma}\label{vtt lemma 14}
Assume \eqref{wtt-assumption} and \eqref{wtt-assumption'} hold. Then for $1\leq n\leq N_1$, we have
\begin{align} \label{vtt lemma 14-1}
    \bigg\|\varphi_k \Big[\ue^{\ui \ddi(t,\xi)}\lvkg(\xi)\Big]_{\lll}\bigg\|_{L^2_\xi}\ls \ep\big(\ln\!\br{t}\big)^2Z(k;n-1)2^{k^-}.
\end{align}
Also, for $1\leq n\leq N_1-1$, we have
\begin{align} \label{vtt lemma 14-2}
    \bigg\|\varphi_k \Big[\p_{\xi_{\ell}}\Big(\ue^{\ui \ddi(t,\xi)}\lvkg(\xi)\Big)\Big]_{\lll}\bigg\|_{L^2_\xi}\ls \ep\big(\ln\!\br{t}\big)^2Z(k;n-1)2^{k^-}.
\end{align}
\end{lemma}

\begin{proof}
Let $\widehat{W^{kg}}(t,\xi):=\ue^{\ui \ddi(t,\xi)}\lvkg(\xi)$. 
Based on \eqref{vector field 2}, we have
\begin{align}
    \widehat{W^{kg}_{\lll}}(t,\xi)\sim \ue^{\ui t\lakg(\xi)}\widehat{\lll}\Big[\ue^{-\ui t\lakg(\xi)}\widehat{W^{kg}}(t,\xi)\Big].
\end{align}
Here we use $\widehat{\lll}$ to denote the corresponding vector field under Fourier transform.

We first consider $n=1$ case. 
For rotation vector fields, we know
\begin{align}
    \widehat{\Omega}_{jk}\Big(\ue^{-\ui t\lakg(\xi)}\widehat{W^{kg}}(t,\xi)\Big)&=\ue^{-\ui t\lakg(\xi)}\bigg[\Big(\widehat{\Omega}_{jk}\,\ue^{\ui \ddi(t,\xi)}\Big)\lvkg(\xi)+\ue^{\ui \ddi(t,\xi)}\Big(\widehat{\Omega}_{jk}\lvkg(\xi)\Big)\bigg].
\end{align}
Hence, 
\begin{align}
    \Big\|\varphi_k\widehat{W^{kg}_{\lll}}(t,\xi)\Big\|_{L^2_\xi}
    \ls\Big\|\widehat{\Omega}_{jk}\cci(t,\xi)\Big\|_{L^\infty_\xi} \Big\|\lvkg(\xi)\Big\|_{L^2_\xi}+\Big\|\widehat{\Omega}_{jk}\lvkg(\xi)\Big\|_{L^2_\xi}\ls \ep\ln\!\br{t} Z(k;n-1)2^{k^-}.
\end{align}
For Lorentz vector fields, we know
\begin{align}
    \widehat{\Gamma}_{j}\Big(\ue^{-\ui t\lakg(\xi)}\widehat{W^{kg}}(t,\xi)\Big)
    =&\;\ue^{-\ui t\lakg(\xi)}\bigg\{\ui\br{\xi}\Big[\p_{\xi_j}\ue^{\ui \ddi(t,\xi)}\lvkg(\xi)+\ue^{\ui \ddi(t,\xi)}\p_{\xi_j}\lvkg(\xi)\Big]\\
    &+\ui\frac{\xi_j}{\br{\xi}}\ue^{\ui \ddi(t,\xi)}\lvkg(\xi)+\ui t\frac{\xi_j}{\br{\xi}}\cci(t,\xi)\ue^{\ui \ddi(t,\xi)}\lvkg(\xi)\no\\
    &+\Big[\p_{\xi_j}\cci(t,\xi)\ue^{\ui \ddi(t,\xi)}\lvkg(\xi)+\cci(t,\xi)\p_{\xi_j}\ue^{\ui \ddi(t,\xi)}\lvkg(\xi)+\cci(t,\xi)\ue^{\ui \ddi(t,\xi)}\p_{\xi_j}\lvkg(\xi)\Big]\bigg\}.\no
\end{align}
Hence,
\begin{align}
    \Big\|\varphi_k\widehat{W^{kg}_{\lll}}(t,\xi)\Big\|_{L^2_\xi}
    \ls&\,\Big(1+\big\|\nabla_{\xi}\ddi(t,\xi)\big\|_{L^\infty_\xi}+\big\|\nabla_{\xi}\cci(t,\xi)\big\|_{L^\infty_\xi}+\big\|\br{t}\cci(t,\xi)\big\|_{L^\infty_\xi}\Big) \Big(\big\|\lvkg(\xi)\big\|_{L^2_\xi}+\big\|\nabla_{\xi}\lvkg(\xi)\big\|_{L^2_\xi}\Big)\no\\
    \ls&\;\ep\big(\ln\!\br{t}\big)^2 Z(k;n-1)2^{k^-}.
\end{align}
In summary, we have verified that for $n=1$,
\begin{align}
    \Big\|\varphi_k\widehat{W^{kg}_{\lll}}(t,\xi)\Big\|_{L^2_\xi}\ls\ep\big(\ln\!\br{t}\big)^2 Z(k;n-1)2^{k^-}.
\end{align}
For $n>1$, we may continue the above process. However, the newly created terms are always combination of the above. With the techniques in Lemma \ref{vtt lemma 13} and Lemma \ref{vtt lemma 13'} in hand, the result will follow. The similar argument also applies to the $\p_{\xi_{\ell}}$ estimate.
\end{proof}

\begin{lemma}\label{vtt lemma 14'}
Assume \eqref{wtt-assumption} and \eqref{wtt-assumption'} hold. Then for $1\leq n\leq N_1$ and $(k,j)\in\jj$, we have
\begin{align}
    \sum_{j\geq-k^-}\lnm{\ue^{-\ui t\lakg}\qq_{jk}\bigg\{\mathscr{F}^{-1} \Big[\ue^{\ui \ddi(t,\xi)}\lvkg(\xi)\Big] \bigg\}_{\lll}}&\ls \ep\big(\ln\!\br{t}\big)^2 Z(k;n-1) \min\Big\{2^{\frac{1}{2}k^+}2^{\frac{5}{2}k^-},2^{\frac{3}{2}k^+}2^{\frac{1}{2}k^-}\br{t}^{-1}\Big\}.
\end{align}
\end{lemma}

\begin{proof}
It is similar to Lemma \ref{vtt lemma 7} using Lemma \ref{vtt lemma 14}.
\end{proof}

\begin{lemma}\label{vtt lemma 14''}
Assume \eqref{wtt-assumption} and \eqref{wtt-assumption'} hold. Then for $0<n\leq N_1$ and $(k,j)\in\jj$ with $1\leq 2^{2k^-}\!\br{t}$, we have for some $\sigma$ sufficiently small,
\begin{align}
    \sum_{j\geq -k^-}\nm{\ue^{-\ui t\lakg}\qq_{jk}\bigg\{\mathscr{F}^{-1} \Big[\ue^{\ui \ddi(t,\xi)}\lvkg(\xi)\Big] \bigg\}_{\lll}}_{L^\infty}&\ls\ep\br{t}^{-\frac{3}{2}+\frac{\sigma}{8}}\big(\ln\!\br{t}\big)^2 Z(k;n)2^{4k^+}2^{-k^-(\frac{1}{2}-\frac{\sigma}{4})}.
\end{align}
\end{lemma}

\begin{proof}
It is similar to Lemma \ref{vtt lemma 8} using Lemma \ref{vtt lemma 14}.
\end{proof}

\begin{remark}
For $n=0$ case, the result also holds, but it does not have $\ln\!\br{t}$ since there is no vector fields applied.
\end{remark}

\smallskip
\subsection{Estimates of \texorpdfstring{$\bii$}{} and \texorpdfstring{$\bbfi$}{}} 

Recall that $\bii$ and $\bbfi$ are non-resonant contributions defined in \eqref{b_infty} and \eqref{B_infty}.

\begin{lemma}\label{vtt lemma 17}
Assume \eqref{wtt-assumption} and \eqref{wtt-assumption'} hold. Then we have 
\begin{align}
    \btnm{\varphi_k\bbb_{\infty}}\ls \ep^2\br{t}^{-1}2^{-N(n)k},
\end{align}
and 
\begin{align}
    \btnm{\varphi_k\p_{\xi_{\ell}}\bbb_{\infty}}\ls \ep^22^{-N(n)k}.
\end{align}
\end{lemma}

\begin{proof}
Note that
\begin{align}
    \bii\simeq \sum_{\iota_2\in\{+,-\}}\sum_{k_1,k_2}\int_{\mathbb{R}^3} \varphi_{\leq0}\big(\eta\br{t}^{\frac{7}{8}}\big)\ue^{\ui t(\br{\xi}+\br{\xi-\eta}-\iota_2\abs{\eta})} \frac{\abs{\xi\!-\!\eta}^2}{\br{\xi\!-\!\eta}\abs{\eta}} \Big(\varphi_{k_1}\ue^{\ui\ddi}\widehat{V_{\infty}^{kg,-}}\Big)(t,\xi\!-\!\eta)\varphi_{k_2}H_{\infty}^{t,\iota_2}(\eta)\,\ud\eta.
\end{align}
Since $H_{\infty}$ has a cutoff function $\varphi_{\leq0}\big(\xi\br{t}^{\frac{7}{8}}\big)$, using Lemma \ref{vtt lemma 1} and Lemma \ref{vtt lemma 12}, we have
\begin{align}
    \btnm{\varphi_k\bii}\no
    \ls&\sum_{k_1,k_2}2^{k_1}2^{-k_2}\Btnm{\varphi_{k_1}\ue^{\ui\ddi}\widehat{V_{\infty}^{kg,-}}}\Blnm{\ue^{\ui t\lawa}P_{k_2}\big(\mathscr{F}^{-1}H_{\infty}^{\iota_2}\big)}\\
    \ls&\sum_{k_1,k_2}2^{k_1}2^{-k_2}\Big(\ep2^{-N(N_0)k_1^+}\Big)\Big(\ep\br{t}^{-1}2^{-N(N_0)k_2^+}2^{k_2}\Big)
    \ls\sum_{k_1,k_2}\ep^2\br{t}^{-1}2^{k_1}2^{-N_0k_1^+-N_0k_2^+}.\no
\end{align}
Then the result naturally follows from Lemma \ref{wtt prelim 1}. The $\xi$ derivative estimates follow with an extra $\br{t}$ created.
\end{proof}

\begin{lemma}\label{vtt lemma 18}
Assume \eqref{wtt-assumption} and \eqref{wtt-assumption'} hold. Then we have 
\begin{align}
    \btnm{\varphi_k\bbf_{\infty}}\ls \ep^2\br{t}^{-\frac{7}{8}}\ln\!\br{t} 2^{-N(n)k},
\end{align}
and
\begin{align}
    \btnm{\varphi_k\p_{\xi_{\ell}}\bbf_{\infty}}\ls \ep^2\br{t}^{\frac{1}{8}}\ln\!\br{t} 2^{-N(n)k}.
\end{align}
\end{lemma}

\begin{proof}
Note that
\begin{align}
    \bbfi\simeq \ue^{\ui \ddi(t,\xi)}\int_t^{\infty}\ue^{-\ui \ddi(s,\xi)}\bii(s,\xi)\,\ud s.
\end{align}
For the non-resonant case,
$\psi(\xi,\eta):=\br{\xi}+\br{\xi\!-\!\eta}-\iota_2\abs{\eta}$ satisfies $\abs{\psi}\gs1$. 
Noticing that $\bii(\infty,\xi)=0$ due to the cutoff function, we then integrate by parts in time using \eqref{resonance-time} to obtain
\begin{align}\label{temp 02}
    &\int_t^{\infty}\ue^{-\ui \ddi(s,\xi)}\bii(s,\xi)\,\ud s\\
    =&\int_t^{\infty}\int_{\mathbb{R}^3} \varphi_{\leq0}\big(\eta\br{s}^{\frac{7}{8}}\big)\ue^{-\ui \ddi(s,\xi)}\ue^{\ui s\psi} \frac{\abs{\xi\!-\!\eta}^2}{\br{\xi\!-\!\eta}\abs{\eta}} \Big(\ue^{\ui\ddi}\widehat{V_{\infty}^{kg,-}}\Big)(s,\xi\!-\!\eta)\big(V_{\infty}^{wa,\iota_2}+\hhfi^{\iota_2}\big)(s,\eta)\,\ud\eta\ud s\no\\
    =&\int_{\mathbb{R}^3} \varphi_{\leq0}\big(\eta\br{t}^{\frac{7}{8}}\big)\frac{1}{\ui\psi}\ue^{-\ui \ddi(t,\xi)}\ue^{\ui t\psi} \frac{\abs{\xi\!-\!\eta}^2}{\br{\xi\!-\!\eta}\abs{\eta}} \Big(\ue^{\ui\ddi}\widehat{V_{\infty}^{kg,-}}\Big)(t,\xi\!-\!\eta)\big(V_{\infty}^{wa,\iota_2}+\hhfi^{\iota_2}\big)(t,\eta)\,\ud\eta\no\\
    &-\int_t^{\infty}\int_{\mathbb{R}^3} \frac{1}{\ui\psi}\frac{\abs{\xi\!-\!\eta}^2}{\br{\xi\!-\!\eta}\abs{\eta}}\ue^{\ui s\psi}\p_s\bigg[\varphi_{\leq0}\big(\eta\br{s}^{\frac{7}{8}}\big)\ue^{-\ui \ddi(s,\xi)}  \Big(\ue^{\ui\ddi}\widehat{V_{\infty}^{kg,-}}\Big)(s,\xi\!-\!\eta)\Big(V_{\infty}^{wa,\iota_2}+\hhfi^{\iota_2}\Big)(s,\eta)\bigg]\,\ud\eta\ud s.\no
\end{align}
In the following estimates, we will focus on the $\hhfi^{\iota_2}$ part and the bounds for $V_{\infty}^{wa,\iota_2}$ is similar and even simpler.
on the right-hand side of \eqref{temp 02}, we call the first integral $M_1$, and split the second integral into $M_2, M_3, M_4$ when $\p_s$ hits different parts.
The first term $M_1$ in \eqref{temp 02} does not have time integral, so we may directly bound its $L^2$ norm by Lemma \ref{vtt lemma 1} and Lemma \ref{vtt lemma 10} as
\begin{align}
    \btnm{\varphi_kM_1}\ls&\sum_{k_1,k_2}2^{\frac{3}{2}\min(k,k_1,k_2)}2^{k_1}2^{-k_2}\Big\|\varphi_{k_1}\widehat{V_{\infty}^{kg,-}}\Big\|_{L^2}\btnm{\varphi_{k_2}\hhfi^{\iota_2}}\\
    \ls&\sum_{k_1,k_2}2^{\frac{1}{2}k_2}2^{k_1}\Big(\ep Z(k_1;n)2^{-k_1^+}\Big)\Big(\ep^22^{\frac{1}{2}k_2^-}\Big)\,\ls\, \ep^3\br{t}^{-\frac{7}{8}}Z(k;n).\no
\end{align}
For the second term in \eqref{temp 02}, we further discuss the effects of $\p_s$. Note that $\p_s\varphi_{\leq0}\big(\eta\br{s}^{\frac{7}{8}}\big)\simeq |\eta|\br{s}^{-\frac{1}{8}}\varphi_{\leq0}\big(\eta\br{s}^{\frac{7}{8}}\big)$. Then the resulting quantity $M_2$ is bounded by Lemma \ref{vtt lemma 1} and Lemma \ref{vtt lemma 10} as
\begin{align}
    \btnm{\varphi_kM_2}\ls&\sum_{k_1,k_2}\int_t^{\infty}\br{s}^{-\frac{1}{8}}2^{\frac{3}{2}\min(k,k_1,k_2)}2^{k_1}\Big\|\varphi_{k_1}\widehat{V_{\infty}^{kg,-}}\Big\|_{L^2}\btnm{\varphi_{k_2}\hhfi^{\iota_2}}\\
    \ls&\sum_{k_1,k_2}\int_t^{\infty}\br{s}^{-\frac{1}{8}}2^{\frac{3}{2}k_2}2^{k_1}\Big(\ep Z(k_1;n)2^{-k_1^+}\Big)\Big(\ep^22^{\frac{1}{2}k_2^-}\Big)
    \ls \int_t^{\infty}\ep^3\br{s}^{-\frac{15}{8}}Z(k;n)\,\ls\, \ep^3\br{t}^{-\frac{7}{8}}Z(k;n).\no
\end{align}
Next, since $\p_s\ue^{\ui \ddi(s,\xi)}=\ui\,\cci(s,\xi)\ue^{\ui \ddi(s,\xi)}$, the resulting quantity $M_3$ is bounded using Lemma\;\ref{vtt lemma 1}, Lemma \ref{vtt lemma 13} and Lemma \ref{vtt lemma 10} as
\begin{align}
    \btnm{\varphi_kM_3}\ls&\sum_{k_1,k_2}\int_t^{\infty}2^{\frac{3}{2}\min(k,k_1,k_2)}2^{k_1}2^{-k_2}\Big(\ep2^{3k_1}\br{s}^{-1}\ln\!\br{s}\Big)\Big\|\varphi_{k_1}\widehat{V_{\infty}^{kg,-}}\Big\|_{L^2}\btnm{\varphi_{k_2}\hhfi^{\iota_2}}\\
    \ls&\sum_{k_1,k_2}\int_t^{\infty}\ep\br{s}^{-1}\ln\!\br{s}2^{\frac{1}{2}k_2}2^{4k_1}\Big(\ep Z(k_1;n)2^{-k_1^+}\Big)\Big(\ep^22^{\frac{1}{2}k_2^-}\Big)\no\\
    \ls& \int_t^{\infty}\ep^4\br{s}^{-\frac{15}{8}}\ln\!\br{s}Z(k;n)\,\ls\, \ep^4\br{t}^{-\frac{7}{8}}\ln\!\br{t}Z(k;n).\no
\end{align}
Finally, since $\p_s\hhfi^{\iota_2}=h_{\infty}^{\iota_2}$, the resulting quantity $M_4$ is bounded using Lemma \ref{vtt lemma 1} and Lemma\;\ref{vtt lemma 9} as
\begin{align}
    \btnm{\varphi_kM_4}\ls&\sum_{k_1,k_2}\int_t^{\infty}2^{\frac{3}{2}\min(k,k_1,k_2)}2^{k_1}2^{-k_2}\Big\|\varphi_{k_1}\widehat{V_{\infty}^{kg,-}}\Big\|_{L^2}\btnm{\varphi_{k_2}h_{\infty}^{\iota_2}}\\
    \ls&\sum_{k_1,k_2}\int_t^{\infty}2^{\frac{1}{2}k_2}2^{k_1}\Big(\ep Z(k_1;n)2^{-k_1^+}\Big)\Big(\ep^2\br{s}^{-1}2^{\frac{1}{2}k_2^-}\Big)
    \ls \int_t^{\infty}\ep^3\br{s}^{-\frac{15}{8}}Z(k;n)\,\ls\, \ep^3\br{t}^{-\frac{7}{8}}Z(k;n).\no
\end{align}
Summarizing all above, we have that
\begin{align}
    \btnm{\varphi_k\bbfi}\ls \ep^3\br{t}^{-\frac{7}{8}}\ln\!\br{t}Z(k;n).
\end{align}
The $\xi$ derivative estimates follow with an extra $\br{t}$ created.
\end{proof}

\begin{remark}
Comparing Lemma \ref{vtt lemma 17} and Lemma \ref{vtt lemma 18}, we can clearly see that the estimate of $\bbf_{\infty}$ is much better than direct integrating over time in the estimate of $\bbb_{\infty}$. The main reason is that $\bbf_{\infty}$ has time integration, which allows us to make use of integration by parts in time (normal forms) for the non-resonant case.
However, without time integration, $\bbb_{\infty}$ estimate is similar to that of the resonant terms, which decays rather slowly.
\end{remark}

\begin{lemma}\label{vtt lemma 18'}
Assume \eqref{wtt-assumption} and \eqref{wtt-assumption'} hold. Then for $(k,j)\in\jj$, we have
\begin{align}
    \sum_{j\geq-k^-}\lnm{\ue^{-\ui t\lakg}\qq_{jk} \bbf_{\infty}}\ls \ep\br{t}^{-\frac{7}{8}}\ln\!\br{t} Z(k;n-1) \min\Big\{2^{\frac{1}{2}k^+}2^{\frac{5}{2}k^-},\,2^{\frac{3}{2}k^+}2^{\frac{1}{2}k^-}\!\br{t}^{-1}\Big\}.
\end{align}
\end{lemma}

\begin{proof}
It is similar to the proof of Lemma \ref{vtt lemma 7} using Lemma \ref{vtt lemma 18}.
\end{proof}

\begin{lemma}\label{vtt lemma 18''}
Assume \eqref{wtt-assumption} and \eqref{wtt-assumption'} hold. Then for $(k,j)\in\jj$ with $1\leq 2^{2k^-}\!\br{t}$, we have for some $\sigma>0$,
\begin{align}
    \sum_{j\geq -k^-}\lnm{\ue^{-\ui t\lakg}\qq_{jk}\bbf_{\infty}}&\ls\ep\br{t}^{-\frac{7}{8}-\frac{3}{2}+\frac{\sigma}{8}}\ln\!\br{t} Z(k;n)2^{4k^+}2^{-k^-(\frac{1}{2}-\frac{\sigma}{4})}.
\end{align}
\end{lemma}

\begin{proof}
It is similar to the proof of Lemma \ref{vtt lemma 8} using Lemma \ref{vtt lemma 18}.
\end{proof}

\bigskip
\section{Bounds on the Nonlinearities of Wave Equation} 

Based on \eqref{wtt 05-diff}, in order to control $\dt\gwa_{\lll}$, it suffices to bound for any $\lll\in\vvv_n$ with $0\leq n\leq N_1$,
\begin{align}\label{decomposition 1}
    \ue^{\ui t\lawa(\xi)}\widehat{\mathcal{N}^{wa}_{\lll}}-\hhh_{\lll,\infty}
    =&\sum_{\lll_1,\lll_2}\bigg\{\sum_{\iota_1,\iota_2\in\{+,-\}}\bbi_{wa}^{\iota_1\iota_2}\big[G^{kg,\iota_1}_{\lll_1},G^{kg,\iota_2}_{\lll_2}\big]\\
    &+\sum_{\iota_1,\iota_2\in\{+,-\}}{\bbi_{wa}^{\iota_1\iota_2}}\Big[\widehat{G^{kg,\iota_1}_{\lll_1}},\big(\ue^{\iota_2\ui\ddi^{\iota_2}}\widehat{V^{kg,\iota_2}_{\infty}}\big)_{\lll_2}+\bbf^{\iota_2}_{\lll_2,\infty}\Big]\no\\
    &+\sum_{\iota_1,\iota_2\in\{+,-\}}{\bbi_{wa}^{\iota_1\iota_2}}\Big[\big(\ue^{\iota_1\ui\ddi^{\iota_1}}\widehat{V^{kg,\iota_1}_{\infty}}\big)_{\lll_1}+\bbf^{\iota_1}_{\lll_1,\infty}\,,\widehat{G^{kg,\iota_2}_{\lll_2}}\Big]\bigg\}\no\\
    &+\bigg\{\sum_{\lll_1,\lll_2}\sum_{\iota_1,\iota_2\in\{+,-\}}{\bbi_{wa}^{\iota_1\iota_2}}\Big[\big(\ue^{\iota_1\ui\ddi^{\iota_1}}\widehat{V^{kg,\iota_1}_{\infty}}\big)_{\lll_1}+\bbf^{\iota_1}_{\lll_1,\infty}\,,\big(\ue^{\iota_2\ui\ddi^{\iota_2}}\widehat{V^{kg,\iota_2}_{\infty}}\big)_{\lll_2}+\bbf^{\iota_2}_{\lll_2,\infty}\Big]-\hhh_{\lll,\infty}\bigg\}.\no
\end{align}
$\sum_{\lll_1,\lll_2}$ is a summation over all combinations of vector fields $\lll_1\in\vvv_{n_1}$ and $\lll_2\in\vvv_{n_2}$ with $n_1+n_2= n$ and $\lll=\lll_1\circ\lll_2$ (due to Leibniz rule). Also, all variables here with $\lll_1$ or $\lll_2$ should be understood in the sense of \eqref{vector field 1} and \eqref{vector field 2}.

\smallskip
\subsection{\texorpdfstring{$S_1'$}{} Estimates} 

\begin{lemma}\label{s lemma 1}
Assume \eqref{wtt-assumption} and \eqref{wtt-assumption'} hold. 
For any $\lll\in\vvv_n$ with $0\leq n\leq N_1$, we have
\begin{align}
    \bigg\|\sum_{\iota_1,\iota_2\in\{+,-\}}\varphi_k\,\bbi_{wa}^{\iota_1\iota_2}\big[G^{kg,\iota_1}_{\lll_1},G^{kg,\iota_2}_{\lll_2}\big]\bigg\|_{L^2_\xi}\ls\ep^2\br{t}^{-(1+H_{wa}''(n)\d)}2^{-N''_{wa}(n)k^+}2^{\frac{1}{2}k^-}.
\end{align}
\end{lemma}

\begin{proof}
Without loss of generality, we assume $n_1\leq n_2$.  
For simplicity, we temporarily ignore $\iota_1$ and $\iota_2$ superscripts. 
Our proof mainly relies on two types of bounds \eqref{ttt 01} and \eqref{ttt 02}.

In principle, based on Remark \ref{remark 1} (the worst scenario is that all $\lll$ hit the same $\gkg$, then we cannot use Lemma \ref{wtt lemma 3'}), we always have 
\begin{align}
    \bnm{P_{k_2}G^{kg}_{\lll_2}}_{L^2}\ls \ep\br{t}^{-H(n_2)\d}2^{-N(n_2)k_2^+}.
\end{align}
Then we focus on the bounds of $P_{k_1}G^{kg}_{\lll_1}$. 
Note that we always have $n_1\leq N_1-2$ since $n_1\leq n_2$ and $n_1+n_2\leq n\leq N_1$, and that $k^+\ls \max\{k_1^+,k_2^+\}$ due to Lemma \ref{wtt prelim 1}.

We first discuss the case when $n=N_1=3$. We divide it into several cases:
\begin{itemize}
    \item 
    Case I: For $2^{k^-}\ls\br{t}^{-1}$, \eqref{ttt 01} and Lemma \ref{wtt lemma 3'} justify that
    \begin{align}
        &\sum_{k_1,k_2}\nm{\varphi_k\bbi_{wa}\big[P_{k_1}G^{kg}_{\lll_1},P_{k_2}G^{kg}_{\lll_2}\big]}_{L^2}
        \ls\sum_{k_1,k_2} 2^{\frac{3}{2}\min\{k_1,k_2,k\}}\bnm{P_{k_1}G^{kg}_{\lll_1}}_{L^2}\bnm{P_{k_2}G^{kg}_{\lll_2}}_{L^2}\\
        \ls&\sum_{k_1,k_2}2^{\frac{1}{2}k^-}\br{t}^{-1}\Big(\ep\br{t}^{-H(n_1)\d}2^{-N(n_1)k_1^+}2^{k_1^-}\Big)\Big(\ep\br{t}^{-H(n_2)\d}2^{-N(n_2)k_2^+}\Big)\no\\
        \ls&\sum_{k_1,k_2}\ep^22^{\frac{1}{2}k^-}\br{t}^{-1}\br{t}^{-(H(n_1)+H(n_2))\d}2^{-N(n_1)k_1^+-N(n_2)k_2^+}2^{k_1^-}\no\\
        \ls&\;\ep^22^{\frac{1}{2}k^-}\br{t}^{-(1+H(n_1)\d+H(n_2)\d)}2^{-\min\{N(n_1),N(n_2)\}k^+}.\no
    \end{align}
    Note that
    \begin{align}
        H(n_1)+H(n_2)\geq H_{wa}''(n),\qquad \min\big\{N(n_1),N(n_2)\big\}\geq N_{wa}''(n).
    \end{align}
    \item
    Case II: For $2^{k^-}\gs\br{t}^{-1}$ and $2^{k_1^-}\ls 2^{k^-}$, \eqref{ttt 02} together with Lemma \ref{wtt lemma 7} implies that
    \begin{align}
        &\sum_{k_1,k_2}\nm{\varphi_k\bbi_{wa}\big[P_{k_1}G^{kg}_{\lll_1},P_{k_2}G^{kg}_{\lll_2}\big]}_{L^2}
        \ls \sum_{k_1,k_2}\bnm{\ue^{\ui t\lakg}P_{k_1}G^{kg}_{\lll_1}}_{L^{\infty}}\bnm{P_{k_2}G^{kg}_{\lll_2}}_{L^2}\\
        \ls& \sum_{k_1,k_2}\Big(\ep\br{t}^{-H(n_1+1)\d}2^{-(N(n_1+1)-\frac{5}{2})k_1^+}2^{\frac{1}{2}k_1^-}\br{t}^{-1}\Big)\Big(\ep\br{t}^{-H(n_2)\d}2^{-N(n_2)k_2^+}\Big)\no\\
        \ls&\sum_{k_1,k_2}\ep^2\br{t}^{-1}\br{t}^{-(H(n_1+1)+H(n_2))\d}2^{-(N(n_1+1)-\frac{5}{2})k_1^+-N(n_2)k_2^+}2^{\frac{1}{2}k_1^-}\no\\
        \ls&\;\ep^2\br{t}^{-(1+H(n_1+1)\d+H(n_2)\d)}2^{-\min\{N(n_1+1)-\frac{5}{2},N(n_2)\}k^+}2^{\frac{1}{2}k^-}.\no
    \end{align}
    Note that summation over $k_1$ will result in $2^{\frac{1}{2}k^-}$.
    Also, we have
    \begin{align}
        H(n_1+1)+H(n_2)\geq H_{wa}''(n),\qquad \min\Big\{N(n_1+1)-\frac{5}{2},N(n_2)\Big\}\geq N_{wa}''(n).
    \end{align}
    \item
    Case III: For $2^{k^-}\gs\br{t}^{-1}$ and $2^{k_1^-}\ls 2^{-k^-}\br{t}^{-1}$, \eqref{ttt 01} and Lemma \ref{wtt lemma 3'} implies that
    \begin{align}
        &\sum_{k_1,k_2}\nm{\varphi_k\bbi_{wa}\big[P_{k_1}G^{kg}_{\lll_1},P_{k_2}G^{kg}_{\lll_2}\big]}_{L^2}
        \ls\sum_{k_1,k_2} 2^{\frac{3}{2}\min\{k_1,k_2,k\}}\bnm{P_{k_1}G^{kg}_{\lll_1}}_{L^2}\bnm{P_{k_2}G^{kg}_{\lll_2}}_{L^2}\\
        \ls&\sum_{k_1,k_2}2^{\frac{3}{2}k^-}\Big(\ep\br{t}^{-H(n_1)\d}2^{-N(n_1)k_1^+}2^{k_1^-}\Big)\Big(\ep\br{t}^{-H(n_2)\d}2^{-N(n_2)k_2^+}\Big)\no\\
        \ls&\sum_{k_1,k_2}\ep^22^{\frac{3}{2}k^-}\br{t}^{-(H(n_1)+H(n_2))\d}2^{-N(n_1)k_1^+-N(n_2)k_2^+}2^{k_1^-}\no\\
        \ls&\;\ep^2\br{t}^{-(1+H(n_1)\d+H(n_2)\d)}2^{-\min\{N(n_1),N(n_2)\}k^+}2^{\frac{1}{2}k^-}.\no
    \end{align}
    Note that summation over $k_1$ will result in $2^{-k^-}\br{t}^{-1}$, and that
    \begin{align}
        H(n_1)+H(n_2)\geq H_{wa}''(n),\qquad \min\big\{N(n_1),N(n_2)\big\}\geq N_{wa}''(n).
    \end{align}
    \item
    Case IV: For $2^{k^-}\gs\br{t}^{-1}$ and $2^{k_1^-}\gs 2^{k^-}$ and $2^{k_1^-}\gs 2^{-k^-}\br{t}^{-1}$, \eqref{ttt 02} and Lemma \ref{wtt lemma 8} implies 
    (in this case, we naturally have $2^{2k_1^-}\br{t}\gs 2^{k_1^-}2^{k^-}\br{t}\gs 1$) 
    \begin{align}
        &\sum_{k_1,k_2}\nm{\varphi_k\bbi_{wa}\big[P_{k_1}G^{kg}_{\lll_1},P_{k_2}G^{kg}_{\lll_2}\big]}_{L^2}
        \ls \sum_{k_1,k_2}\nm{\ue^{-\ui t\lakg}P_{k_1}G^{kg}_{\lll_1}}_{L^{\infty}}\bnm{P_{k_2}G^{kg}_{\lll_2}}_{L^2}\\
        \ls& \sum_{k_1,k_2}\Big(\ep\br{t}^{-H(n_1+2)\d}2^{-(N(n_1+2)-4)k_1^+}2^{-k_1^-(\frac{1}{2}-\frac{\sigma}{4})}\br{t}^{-\frac{3}{2}+\frac{\sigma}{8}}\Big)\Big(\ep\br{t}^{-H(n_2)\d}2^{-N(n_2)k_2^+}\Big)\no\\
        \ls&\sum_{k_1,k_2}\ep^22^{\frac{1}{2}k^-}\!\br{t}^{-\frac{3}{2}}\br{t}^{-(H(n_1+2)+H(n_2))\d+\frac{\sigma}{8}}2^{-(N(n_1+2)-4)k_1^+-N(n_2)k_2^+}2^{\frac{\sigma}{4}k_1^-}\Big(2^{-\frac{1}{2}k_1^-}2^{-\frac{1}{2}k^-}\Big).\no
    \end{align}
    Note that $2^{-\frac{1}{2}k_1^-}2^{-\frac{1}{2}k^-}\ls \br{t}^{\frac{1}{2}}$. Hence,
    \begin{align}
        &\sum_{k_1,k_2}\nm{\varphi_k\bbi_{wa}\big[P_{k_1}G^{kg}_{\lll_1},P_{k_2}G^{kg}_{\lll_2}\big]}_{L^2}\\
        \ls&\sum_{k_1,k_2}\ep^22^{\frac{1}{2}k^-}\br{t}^{-1}\br{t}^{-(H(n_1+2)+H(n_2))\d+\frac{\sigma}{8}}2^{-(N(n_1+2)-4)k_1^+-N(n_2)k_2^+}2^{\frac{\sigma}{4}k_1^-}\no\\
        \ls&\;\ep^2\br{t}^{-(1+H(n_1)\d+H(n_2)\d)-\frac{\delta}{8}}2^{-\min\{N(n_1+2)-4,N(n_2)\}k^+}2^{\frac{1}{2}k^-}.\no
    \end{align}
    Here we have for $\sigma\simeq\d$,
    \begin{align}
        H(n_1+2)+H(n_2)-\frac{1}{8}\geq H_{wa}''(n),\qquad \min\Big\{N(n_1+2)-4,N(n_2)\Big\}\geq N_{wa}''(n).
    \end{align}
\end{itemize}

The above discussion works perfectly well for $n=N_1=3$ (i.e. $(n_1,n_2)=(0,3)$ or $(n_1,n_2)=(1,2)$). For the other $n$, there are subtle issues (a naive application of the above recipe does not give sufficient control to $N_{wa}''$), so we need a detailed discussion:
\begin{itemize}
    \item 
    $n=2$ (i.e. $(n_1,n_2)=(0,2)$ or $(n_1,n_2)=(1,1)$): If $(n_1,n_2)=(0,2)$, then the above recipe works well. However, if $(n_1,n_2)=(1,1)$, we need to discuss based on $k,k_1,k_2$ values. Based on Lemma \ref{wtt prelim 1}, if $k\simeq k_2\geq k_1$ or $k_1\simeq k_2\geq k$, we can still use the above recipe; if $k\simeq k_1\geq k_2$, then in Case I through Case IV, we interchange the status of $G^{kg}_{\lll_1}$ and $G^{kg}_{\lll_2}$, which can improve $N_{wa}''$. 
    \item
    $n=1$ (i.e. $(n_1,n_2)=(0,1)$) or $n=0$ (i.e. $(n_1,n_2)=(0,0)$): Based on Lemma \ref{wtt prelim 1}, if $k\simeq k_2\geq k_1$ or $k_1\simeq k_2\geq k$, we can still use the above recipe; if $k\simeq k_1\geq k_2$, then in Case I through Case IV, we interchange the status of $G^{kg}_{\lll_1}$ and $G^{kg}_{\lll_2}$.
\end{itemize}
\end{proof}

\begin{lemma}\label{s lemma 2}
Assume \eqref{wtt-assumption} and \eqref{wtt-assumption'} hold. 
For any $\lll\in\vvv_n$ with $0\leq n\leq N_1$, we have
\begin{align}
    \bigg\|\sum_{\iota_1,\iota_2\in\{+,-\}}\varphi_k\,{\bbi_{wa}^{\iota_1\iota_2}}\Big[\widehat{G^{kg,\iota_1}_{\lll_1}},\big(\ue^{\iota_2\ui\ddi^{\iota_2}}\widehat{V^{kg,\iota_2}_{\infty}}\big)_{\lll_2}+\bbf^{\iota_2}_{\lll_2,\infty}\Big]\bigg\|_{L^2_\xi}\ls\ep^{2}\br{t}^{-(1+H(n)\d)}2^{-N(n)k^+}2^{\frac{1}{2}k^-},\\
    \bigg\|\sum_{\iota_1,\iota_2\in\{+,-\}}\varphi_k\,{\bbi_{wa}^{\iota_1\iota_2}}\Big[\big(\ue^{\iota_1\ui\ddi^{\iota_1}}\widehat{V^{kg,\iota_1}_{\infty}}\big)_{\lll_1}+\bbf^{\iota_1}_{\lll_1,\infty}\,,\widehat{G^{kg,\iota_2}_{\lll_2}}\Big]\bigg\|_{L^2_\xi}\ls\ep^{2}\br{t}^{-(1+H(n)\d)}2^{-N(n)k^+}2^{\frac{1}{2}k^-}.
\end{align}
\end{lemma}

\begin{proof}
Basically, the proof is similar to that of Lemma \ref{s lemma 1}. The only difference is that we do not have $H(n_2)$ time decay, so we should always assign $L^2$ to $G^{kg,\iota_1}_{\lll_1}$ or $G^{kg,\iota_2}_{\lll_2}$. Then $\big(\ue^{\iota_1\ui\ddi^{\iota_1}}\widehat{V^{kg,\iota_1}_{\infty}}\big)_{\lll_1}+\bbf^{\iota_1}_{\lll_1,\infty}$ or $\big(\ue^{\iota_2\ui\ddi^{\iota_2}}\widehat{V^{kg,\iota_2}_{\infty}}\big)_{\lll_2}+\bbf_{\lll_2,\infty}$ may take $L^2$ or $L^{\infty}$ with Lemma \ref{vtt lemma 1}, Lemma \ref{vtt lemma 7}, Lemma \ref{vtt lemma 18} and Lemma \ref{vtt lemma 18''}  following the argument in the proof of Lemma \ref{s lemma 1}. 
\begin{itemize}
    \item 
    If only partial $\lll$ hits $\gkg$, then we have sufficient time decay and $k^+$ decay since less $\lll$ yields faster decay rate.
    \item
    If the full $\lll$ hits $\gkg$, then this term already achieve the critical time decay and $k^+$ decay in the desired result. 
\end{itemize}
In both cases, we can get the desired result.
\end{proof}

\begin{lemma}\label{s lemma 3}
Assume \eqref{wtt-assumption} and \eqref{wtt-assumption'} hold. 
For any $\lll\in\vvv_n$ with $0\leq n\leq N_1$, we have
\begin{align}
    &\Bigg\|\varphi_k\bigg\{\sum_{\lll_1,\lll_2}\sum_{\iota_1,\iota_2\in\{+,-\}}{\bbi_{wa}^{\iota_1\iota_2}}\Big[\big(\ue^{\iota_1\ui\ddi^{\iota_1}}\widehat{V^{kg,\iota_1}_{\infty}}\big)_{\lll_1}+\bbf^{\iota_1}_{\lll_1,\infty}\,,\big(\ue^{\iota_2\ui\ddi^{\iota_2}}\widehat{V^{kg,\iota_2}_{\infty}}\big)_{\lll_2}+\bbf^{\iota_2}_{\lll_2,\infty}\Big]-\hhh_{\lll,\infty}\bigg\}\Bigg\|_{L^2_\xi}\\
    \ls&\;\ep^2\br{t}^{-(1+H(n)\d)}2^{-N(n)k^+}2^{\frac{1}{2}k^-}.\no
\end{align}
\end{lemma}
\begin{proof}

Based on the proof of \cite[Lemma 4.2]{Deng.Pusateri2020}, we have (up to fast-decaying terms)
\begin{align}
    \hhh_{\lll,\infty}\sim&\sum_{\lll_1,\lll_2}\hhh_{\lll_1,\lll_2,\infty}
    :=\sum_{\lll_1,\lll_2}\ (2\pi)^{-\frac{3}{2}}\varphi_{\leq0}\big(\xi\br{t}^{\frac{7}{8}}\big)\cdot\!\int_{\mathbb{R}^3} \text{\large$\ue$}^{\ui t\left(\abs{\xi}-\frac{\xi\cdot\eta}{\br{\eta}}\right)} \big(\ue^{\iota_1\ui\ddi^{\iota_1}}\widehat{V^{kg,\iota_1}_{\infty}}\big)_{\lll_1}(-\eta)\big(\ue^{\iota_2\ui\ddi^{\iota_2}}\widehat{V^{kg,\iota_2}_{\infty}}\big)_{\lll_2}(\eta)\,\dd\eta.
\end{align}
In the following estimates, let $\widehat{g_1}=\big(\ue^{\iota_1\ui\ddi^{\iota_1}}\widehat{V^{kg,\iota_1}_{\infty}}\big)_{\lll_1}$ and $\widehat{g_2}=\big(\ue^{\iota_2\ui\ddi^{\iota_2}}\widehat{V^{kg,\iota_2}_{\infty}}\big)_{\lll_2}$.

\noindent
{\it \underline{Step\;1}. Resonant Case: $(\iota_1,\iota_2)=(+,-)$ or $(\iota_1,\iota_2)=(-,+)$} 

We further decompose ${\bbi_{wa}^{\iota_1\iota_2}}$ into low-frequency and high-frequency parts:
\begin{align}
    {\bbi_{wa}^{\iota_1\iota_2}}\big[\widehat{g_1},\widehat{g_2}\big]
    =\;&\varphi_{\leq0}\big(\xi\br{t}^{\frac{7}{8}}\big)\!\cdot{\bbi_{wa}^{\iota_1\iota_2}}\big[\widehat{g_1},\widehat{g_2}\big]+\varphi_{\geq0}\big(\xi\br{t}^{\frac{7}{8}}\big)\!\cdot{\bbi_{wa}^{\iota_1\iota_2}}\big[\widehat{g_1},\widehat{g_2}\big].
\end{align} 

\noindent{\it \underline{Step\;1-1}. Resonant Case - Low Frequency}

We may decompose
\begin{align}
    &\;\varphi_{\leq0}\big(\xi\br{t}^{\frac{7}{8}}\big)\!\cdot{\bbi_{wa}^{-+}}\big[\widehat{g_1},\widehat{g_2}\big]-\frac{1}{2}{\hhh}_{\lll_1,\lll_2,\infty}\\
    \simeq&\;\varphi_{\leq0}\big(\xi\br{t}^{\frac{7}{8}}\big)\bigg\{\int_{\mathbb{R}^3} \ue^{\ui t\Phi_{\,wa}^{-+}(\xi,\eta)} a_{-+}(\xi,\eta) \widehat{g_1}(t,\xi\!-\!\eta)\widehat{g_2}(t,\eta)\,\dd\eta-2\int_{\mathbb{R}^3} \ue^{\ui t\left(\abs{\xi}-\frac{\xi\cdot\eta}{\br{\eta}}\right)} \widehat{g_1}(t,-\eta) \widehat{g_2}(t,\eta)\,\dd\eta\bigg\}\no\\
    =&\;\ss_1+\ss_2+\ss_3,\no
\end{align}
where
\begin{align}
    \ss_1&:=\varphi_{\leq0}\big(\xi\br{t}^{\frac{7}{8}}\big)\int_{\mathbb{R}^3} \ue^{\ui t\Phi_{\,wa}^{-+}(\xi,\eta)} \big[a_{-+}(\xi,\eta)-2\big] \widehat{g_1}(t,\xi\!-\!\eta)\widehat{g_2}(t,\eta)\,\dd\eta,\\
    \ss_2&:=2\varphi_{\leq0}\big(\xi\br{t}^{\frac{7}{8}}\big)\int_{\mathbb{R}^3} \Big[\ue^{\ui t\Phi_{\,wa}^{-+}(\xi,\eta)}-\ue^{\ui t\left(\abs{\xi}-\frac{\xi\cdot\eta}{\br{\eta}}\right)} \Big] \widehat{g_1}(t,\xi\!-\!\eta)\widehat{g_2}(t,\eta)\,\dd\eta,\\
    \ss_3&:=2\varphi_{\leq0}\big(\xi\br{t}^{\frac{7}{8}}\big)\int_{\mathbb{R}^3}\ue^{\ui t\left(\abs{\xi}-\frac{\xi\cdot\eta}{\br{\eta}}\right)}\Big[\widehat{g_1}(t,\xi\!-\!\eta)-\widehat{g_1}(t,-\eta)\Big]\widehat{g_2}(t,\eta)\,\dd\eta.
\end{align}
For $\ss_1$, noticing that
\begin{align}
    \big|a_{-+}(\xi,\eta)-2\big|=\abs{1-\frac{(\xi\!-\!\eta)\cdot\eta-1}{\br{\xi\!-\!\eta}\br{\eta}}-2}\ls \abs{\xi},
\end{align}
we may use \eqref{ttt 01} and Lemma \ref{vtt lemma 14} to obtain
\begin{align}
    \btnm{\varphi_k\ss_1}\ls 2^{\frac{3}{2}k}2^k\varphi_{\leq0}\big(2^k\!\br{t}^{\frac{7}{8}}\big)\tnm{\widehat{g_1}}\tnm{\widehat{g_2}}\ls \ep^2\br{t}^{-\frac{7}{4}}\,2^{\frac{1}{2}k^-}.
\end{align}
For $\ss_2$, noticing that
\begin{align}
    \abs{\ue^{\ui t\Phi_{\,wa}^{-+}(\xi,\eta)}-\ue^{\ui t\big(\abs{\xi}-\frac{\xi\cdot\eta}{\br{\eta}}\big)}}=\Babs{\ue^{\ui t\big(\abs{\xi}+\br{\xi-\eta}-\br{\eta}\big)}-\ue^{\ui t\left(\abs{\xi}-\frac{\xi\cdot\eta}{\br{\eta}}\right)}}\ls \abs{\xi}^2\br{t},
\end{align}
we may use \eqref{ttt 01} and Lemma \ref{vtt lemma 14} to obtain
\begin{align}
    \btnm{\varphi_k\ss_2}\ls 2^{\frac{3}{2}k}2^{2k}\br{t}\varphi_{\leq0}\big(2^k\!\br{t}^{\frac{7}{8}}\big)\tnm{\widehat{g_1}}\tnm{\widehat{g_2}}\ls \ep^2\br{t}^{-\frac{13}{8}}\,2^{\frac{1}{2}k^-}.
\end{align}
$\ss_3$ is a bit more delicate. Lemma \ref{vtt lemma 1} together with Minkowski's integral inequality implies
\begin{align}
    \,\big\|\widehat{g_1}(t,-\eta)-\widehat{g_1}(t,\xi\!-\!\eta)\big\|_{L^2}
    \ls&\,\bigg\|\int_0^{\abs{\xi}}\p_{\hat{\xi}\,}\widehat{g_1}(t,-\eta+c\,\hat{\xi})\,\ud c\bigg\|_{L^2_{\eta}}
    \ls\int_0^{\abs{\xi}}\big\|\p_{\hat{\xi}\,}\widehat{g_1}(t,-\eta+c\,\hat{\xi})\big\|_{L^2_{\eta}}\ud c\\
    \ls&\int_0^{\abs{\xi}}\big\|\p_{\eta}\widehat{g_1}(t,-\eta)\big\|_{L^2_{\eta}}\ud c
    \,\ls\,\abs{\xi}\big\|\p_{\eta}\widehat{g_1}(t,\eta)\big\|_{L^2_{\eta}}.\no
\end{align}
Then we may use \eqref{ttt 01} and Lemma \ref{vtt lemma 14} to obtain
\begin{align}
    \btnm{\varphi_k\ss_3}\ls 2^{\frac{3}{2}k}\varphi_{\leq0}\big(2^k\!\br{t}^{\frac{7}{8}}\big)\nm{\widehat{g_1}}_{H^1_{\eta}}\tnm{\widehat{g_2}}\ls \ep^2\br{t}^{-\frac{13}{8}}\,2^{\frac{1}{2}k^-}.
\end{align}
In total, we have for $(\iota_1,\iota_2)=(-,+)$ case
\begin{align}
    \Btnm{\varphi_{\leq0}\big(\xi\br{t}^{\frac{7}{8}}\big)\!\cdot{\bbi_{wa}^{-+}}\big[\widehat{g_1},\widehat{g_2}\big]-\frac{1}{2}\hhh_{\lll_1,\lll_2,\infty}}\ls \ep^2\br{t}^{-\frac{7}{4}}\,2^{\frac{1}{2}k^-}.
\end{align}
By a similar argument, we may justify $(\iota_1,\iota_2)=(+,-)$ case.\\
\ \\
\noindent{\it \underline{Step\;1-2}. Resonant Case - High Frequency}

For the high-frequency part, we bound
\begin{align}
    &\tnm{\varphi_k\Big\{\varphi_{\geq0}\big(\xi\br{t}^{\frac{7}{8}}\big)\!\cdot{\bbi_{wa}^{\iota_1\iota_2}}\big[\widehat{g_1},\widehat{g_2}\big]\Big\}}
    \ls\sum_{k_1,k_2} \tnm{\varphi_k\Big\{\varphi_{\geq0}\big(\xi\br{t}^{\frac{7}{8}}\big)\!\cdot{\bbi_{wa}^{\iota_1\iota_2}}\big[\varphi_{k_1}\widehat{g_1},\varphi_{k_2}\widehat{g_2}\big]\Big\}}.
\end{align}
If $2^{\min(k_1,k_2)}\leq \br{t}^{-\frac{1}{2}}$, we integrate by parts in $\eta$ using \eqref{resonance-space}, and apply Lemma \ref{vtt lemma 14} and \eqref{ttt 01} to obtain
\begin{align} 
    &\sum_{k_1,k_2}\tnm{\varphi_k\Big\{\varphi_{\geq0}\big(\xi\br{t}^{\frac{7}{8}}\big)\!\cdot{\bbi_{wa}^{\iota_1\iota_2}}\big[\varphi_{k_1}\widehat{g_1},\varphi_{k_2}\widehat{g_2}\big]\Big\}}\\
    \ls&\sum_{k_1,k_2}2^{\frac{3}{2}k}\lnm{\varphi_k\varphi_{\geq0}\big(\xi\br{t}^{\frac{7}{8}}\big)\!\cdot{\bbi_{wa}^{\iota_1\iota_2}}\big[\varphi_{k_1}\widehat{g_1},\varphi_{k_2}\widehat{g_2}\big]}
    \ls\sum_{k_1,k_2}
    2^{\frac{3}{2}k}\br{t}^{-1}2^{-k}\bnm{\varphi_{k_1}\widehat{g_1}}_{H^1_{\eta}}\bnm{\varphi_{k_2}\widehat{g_2}}_{H^1_{\eta}}\no\\
    \ls&\sum_{k_1,k_2}\ep^22^{\frac{1}{2}k}\br{t}^{-1}
    2^{-N(n_1)k_1^+-N(n_2)k_2^+}2^{k_1^-+k_2^-}
    \ls\;\ep^22^{\frac{1}{2}k^-}\br{t}^{-\frac{3}{2}}2^{-\min\{N(n_1),N(n_2)\}k^+}.\no
\end{align}
If $2^{\min(k_1,k_2)}\geq \br{t}^{-\frac{1}{2}}$ and $\max(k_1,k_2)\geq0$ (without loss of generality, we assume $k_1\gs k_2$), we apply \eqref{ttt 02} and Lemma \ref{vtt lemma 14''}, Lemma \ref{vtt lemma 8} to obtain
\begin{align}
    &\sum_{k_1,k_2}\tnm{\varphi_k\Big\{\varphi_{\geq0}\big(\xi\br{t}^{\frac{7}{8}}\big)\!\cdot{\bbi_{wa}^{\iota_1\iota_2}}\big[\varphi_{k_1}\widehat{g_1},\varphi_{k_2}\widehat{g_2}\big]\Big\}}\\
    \ls&\sum_{k_1,k_2}\btnm{\varphi_{k_1}\widehat{g_1}}\blnm{\ue^{-\ui t\lakg}P_{k_2}g_2}
    \ls\sum_{k_1,k_2}\ep^2\br{t}^{-\frac{3}{2}+\frac{\sigma}{8}}2^{-N(n_1)k_1^+-N(n_2)k_2^+}2^{-k_1^-(\frac{1}{2}-\frac{\sigma}{4})}2^{k_2^-}\no\\
    \ls&\sum_{k_1,k_2}\ep^2\br{t}^{-\frac{3}{2}+\frac{\sigma}{8}}2^{-N(n_1)k_1^+-N(n_2)k_2^+}2^{k_2^-}
    \ls\;\ep^22^{\frac{1}{2}k^-}\br{t}^{-\frac{17}{16}+\frac{\sigma}{8}}2^{-\min\{N(n_1),N(n_2)\}k^+}.\no
\end{align}
If $2^{\min(k_1,k_2)}\geq \br{t}^{-\frac{1}{2}}$ and $\max(k_1,k_2)\leq0$, we integrate by parts in $\eta$ using \eqref{resonance-space}, and apply \eqref{ttt 02} and Lemma \ref{vtt lemma 14'}, Lemma \ref{vtt lemma 7} to obtain
\begin{align}
    &\sum_{k_1,k_2}\tnm{\varphi_k\Big\{\varphi_{\geq0}\big(\xi\br{t}^{\frac{7}{8}}\big)\!\cdot\bbi_{wa}^{\iota_1\iota_2}\big[\varphi_{k_1}\widehat{g_1},\varphi_{k_2}\widehat{g_2}\big]\Big\}}
    \ls\sum_{k_1,k_2}\br{t}^{-1}2^{-k}\btnm{\varphi_{k_1}\p_{\xi_{\ell}}\widehat{g_1}}\lnm{\ue^{-\ui t\lakg}P_{k_2}g_2}\\
    \ls&\sum_{k_1,k_2}\ep^2\br{t}^{-2}2^{-N(n_1)k_1^+-N(n_2)k_2^+}2^{\frac{1}{2}k_1^-}2^{k_2^-}
    \ls\;\ep^22^{\frac{1}{2}k^-}\br{t}^{-\frac{25}{16}}2^{-\min\{N(n_1),N(n_2)\}k^+}.\no
\end{align}
In total, we have
\begin{align}
    \tnm{\varphi_k\Big\{\varphi_{\geq0}\big(\xi\br{t}^{\frac{7}{8}}\big)\!\cdot{\bbi_{wa}^{\iota_1\iota_2}}\big[\varphi_{k_1}\widehat{g_1},\varphi_{k_2}\widehat{g_2}\big]\Big\}}\ls \ep^2\br{t}^{-\frac{17}{16}+\frac{\sigma}{8}}2^{-N(n)k}2^{\frac{1}{2}k^-}.
\end{align}

\noindent{\it \underline{Step\;2}. Non-Resonant Case: $(\iota_1,\iota_2)=(+,+)$ or $(\iota_1,\iota_2)=(-,-)$}

Using stationary phase argument in Lemma \ref{wtt prelim 3} (with $\eta_c = \xi/2$), we have 
\begin{align}
    &\btnm{\varphi_k{\bbi_{wa}^{\iota_1\iota_2}}\big[\widehat{g_1},\widehat{g_2}\big]}
    \ls\sum_{k_1,k_2} \Btnm{\varphi_k{\bbi_{wa}^{\iota_1\iota_2}}\big[\varphi_{k_1}\widehat{g_1},\varphi_{k_2}\widehat{g_2}\big]}
    \ls\sum_{k_1,k_2} 2^{\frac{3}{2}k}\Blnm{\varphi_k{\bbi_{wa}^{\iota_1\iota_2}}\big[\varphi_{k_1}\widehat{g_1},\varphi_{k_2}\widehat{g_2}\big]}\\
    \ls&\sum_{k_1,k_2} 2^{\frac{3}{2}k}\br{t}^{-\frac{3}{2}}2^{\frac{5}{2}k^+}\bnm{\varphi_{k_1}\widehat{g_1}}_{L^{\infty}_{\eta}}\bnm{\varphi_{k_2}\widehat{g_2}}_{L^{\infty}_{\eta}}
    \ls\;\ep^22^{\frac{3}{2}k}\br{t}^{-\frac{3}{2}}
    \sum_{k_1,k_2} 2^{\frac{5}{2}k^+-N(n_1)k_1^+-N(n_2)k_2^+}2^{k_1^-+k_2^-}\no\\
    \ls&\;\ep^22^{\frac{1}{2}k}\br{t}^{-\frac{19}{8}}2^{-(N(n)-5)k^+}.\no
\end{align}

\begin{remark}
The application of stationary phase argument highly depends on the non-resonant phase, which is not applicable to the resonant phase.
\end{remark}

\noindent{\it \underline{Step\;3}. $\bbf_{\lll_j,\infty}$ Terms}

Note that all terms related to $\bbf_{\lll_j,\infty}$ can be bounded in a similar manner as Lemma \ref{s lemma 1} since Lemma \ref{vtt lemma 18} and Lemma \ref{vtt lemma 18'} provide sufficient time decay. 
\end{proof}

Combining Lemma \ref{s lemma 1}, Lemma \ref{s lemma 2} and Lemma \ref{s lemma 3}, we get the following proposition: 

\begin{proposition}\label{nonlinear 1}
Assume \eqref{wtt-assumption} and \eqref{wtt-assumption'} hold. 
For any $\lll\in\vvv_n$ with $0\leq n\leq N_1$, we have
\begin{align}
    \big\|\varphi_k\dt\widehat{\llgwa}\big\|_{L^2_\xi}\ls\ep^2\br{t}^{-(1+H_{wa}''(n)\d)}2^{-N_{wa}''(n)k^+}2^{\frac{1}{2}k^-}.
\end{align}
\end{proposition}

\smallskip
\subsection{\texorpdfstring{$T_1'$}{} Estimates} 

\begin{lemma}\label{s lemma 4}
Assume \eqref{wtt-assumption} and \eqref{wtt-assumption'} hold. 
For any $\lll\in\vvv_n$ with $0\leq n\leq N_1-1$ and $\ell\in\{1,2,3\}$, we have
\begin{align}
    \bigg\|\varphi_k\p_{\xi_{\ell}}\Big\{\ue^{-\ui t\lawa(\xi)}\bbi_{wa}^{\iota_1\iota_2}\big[G^{kg,\iota_1}_{\lll_1},G^{kg,\iota_2}_{\lll_2}\big]\Big\}\bigg\|_{L^2_\xi}\ls\ep^2\br{t}^{-H_{wa}''(n)\d}2^{-N_{wa}''(n)k^+}2^{\frac{1}{2}k^-}.
\end{align}
\end{lemma}

\begin{proof}
The proof is similar to that of Lemma \ref{s lemma 1} and of \cite[Lemma 4.2]{Ionescu.Pausader2019}.
\begin{align}
    \bbi_{wa}^{\iota_1\iota_2}\big[G^{kg,\iota_1}_{\lll_1},G^{kg,\iota_2}_{\lll_2}\big](t,\xi)\sim \int_{\r^3}\ue^{\ui t\Phi_{wa}^{\iota_1\iota_2}}a_{\iota_1\iota_2}(\xi,\eta)\widehat{G^{kg,\iota_1}_{\lll_1}}(t,\xi\!-\!\eta)\widehat{G^{kg,\iota_2}_{\lll_2}}(t,\eta)\,\ud\eta\,.
\end{align}
Note that $\xi_{\ell}$ derivative may hit the phase $\ue^{-\ui t\lawa}\ue^{\ui t\Phi_{wa}^{\iota_1\iota_2}}$, the multiplier $a_{\iota_1\iota_2}$, or $G^{kg,\iota_1}_{\lll_1}(\xi\!-\!\eta)$.
\begin{itemize}
    \item 
    If $\xi_{\ell}$ derivative hits the phase $\ue^{-\ui t\lawa}\ue^{\ui t\Phi_{wa}^{\iota_1\iota_2}}$, then we have an extra $\br{t}$ term popping out. Following exactly the same argument as in the proof of Lemma \ref{s lemma 1}, we get the desired result.
    \item
     If $\xi_{\ell}$ derivative hits the multiplier $a_{\iota_1\iota_2}$, then nothing happens and we follow the same proof of Lemma \ref{s lemma 1}.
     \item
     If $\xi_{\ell}$ derivative hits $\widehat{G^{kg,\iota_1}_{\lll_1}}(\xi\!-\!\eta)$, then we simply use \eqref{ttt 01} and Lemma \ref{wtt lemma 1} to obtain
     \begin{align}
         \tnm{\varphi_k{\bbi_{wa}^{\iota_1\iota_2}}\Big[\varphi_{k_1}\big(\p_{\xi_{\ell}}\widehat{G^{kg,\iota_1}_{\lll_1}}\big),\varphi_{k_2}\widehat{G^{kg,\iota_2}_{\lll_2}}\Big]}
         \ls&\;2^{\frac{3}{2}\min\{k_1,k_2,k\}}\Btnm{\varphi_{k_1}\big(\p_{\xi_{\ell}}\widehat{G^{kg,\iota_1}_{\lll_1}}\big)}\Btnm{\varphi_{k_2}\widehat{G^{kg,\iota_2}_{\lll_2}}}\\
         \ls&\;\ep^2\br{t}^{-(H(n_1+1)+H(n_2))\d} 2^{\frac{3}{2}\min\{k_1,k_2,k\}}2^{-(N(n_1+1)+1)k_1^+-N(n_2)k_2^+}.\no
     \end{align}
     Here, through change of variables in the convolution, we can always assume $k_1\leq k_2$, so we know both $H_{wa}''(n)$ and $N_{wa}''(n)$ requirements can be met.
\end{itemize}
\end{proof}

\begin{lemma}\label{s lemma 5}
Assume \eqref{wtt-assumption} and \eqref{wtt-assumption'} hold. 
For any $\lll\in\vvv_n$ with $0\leq n\leq N_1-1$ and $\ell\in\{1,2,3\}$, we have
\begin{align}
    &\bigg\|\varphi_k\p_{\xi_{\ell}}\bigg\{\ue^{-\ui t\lawa(\xi)}\,{\bbi_{wa}^{\iota_1\iota_2}}\Big[\widehat{G^{kg,\iota_1}_{\lll_1}},\big(\ue^{\iota_2\ui\ddi^{\iota_2}}\widehat{V^{kg,\iota_2}_{\infty}}\big)_{\lll_2}+\bbf_{\lll_2,\infty}^{\iota_2}\Big]\bigg\}\bigg\|_{L^2_\xi}
    \ls \ep^{2}\br{t}^{-H_{wa}''(n)\d}2^{-N_{wa}''(n)k^+}2^{\frac{1}{2}k^-},
    \\
    &\bigg\|\varphi_k\p_{\xi_{\ell}}\bigg\{\ue^{-\ui t\lawa(\xi)}\,{\bbi_{wa}^{\iota_1\iota_2}}\Big[\big(\ue^{\iota_1\ui\ddi^{\iota_1}}\widehat{V^{kg,\iota_1}_{\infty}}\big)_{\lll_1}+\bbf_{\lll_1,\infty}^{\iota_1}\,,\widehat{G^{kg,\iota_2}_{\lll_2}}\Big]\bigg\}\bigg\|_{L^2_\xi}
    \ls \ep^{2}\br{t}^{-H_{wa}''(n)\d}2^{-N_{wa}''(n)k^+}2^{\frac{1}{2}k^-}.
\end{align}
\end{lemma}

\begin{proof}
It is similar to the proofs of Lemma \ref{s lemma 2} and Lemma \ref{s lemma 4}.
\end{proof}

\begin{lemma}\label{s lemma 6}
Assume \eqref{wtt-assumption} and \eqref{wtt-assumption'} hold. 
For any $\lll\in\vvv_n$ with $0\leq n\leq N_1-1$ and $\ell\in\{1,2,3\}$, we have
\begin{align}
    &\left\|\varphi_k\p_{\xi_{\ell}}\Bigg\{\ue^{-\ui t\lawa(\xi)}\bigg(\sum_{\lll_1,\lll_2}\sum_{\iota_1,\iota_2\in\{+,-\}}{\bbi_{wa}^{\iota_1\iota_2}}\Big[\big(\ue^{\iota_1\ui\ddi^{\iota_1}}\widehat{V^{kg,\iota_1}_{\infty}}\big)_{\lll_1}+\bbf_{\lll_1,\infty}^{\iota_1}\,,\big(\ue^{\iota_2\ui\ddi^{\iota_2}}\widehat{V^{kg,\iota_2}_{\infty}}\big)_{\lll_2}+\bbf_{\lll_2,\infty}^{\iota_2}\Big]-\hhh_{\lll,\infty}\bigg)\Bigg\}\right\|_{L^2_\xi}\no\\
    \ls&\;\ep^2\br{t}^{-H_{wa}''(n)\d}2^{-N_{wa}''(n)k^+}2^{\frac{1}{2}k^-}.
\end{align}
\end{lemma}

\begin{proof} 
It is similar to the proofs of Lemma \ref{s lemma 3} and Lemma \ref{s lemma 4}.
\end{proof}

Combining Lemma \ref{s lemma 4}, Lemma \ref{s lemma 5} and Lemma \ref{s lemma 6}, we get the following proposition:

\begin{proposition}\label{nonlinear 3}
Assume \eqref{wtt-assumption} and \eqref{wtt-assumption'} hold. 
For any $\lll\in\vvv_n$ with $0\leq n\leq N_1-1$ and $\ell\in\{1,2,3\}$, we have
\begin{align}
    \Big\|\varphi_k\p_{\xi_{\ell}} \Big\{\widehat{\mathcal{N}_{\lll}^{wa}}-\ue^{-\ui t\lawa(\xi)}\hhh_{\lll,\infty}\Big\}\Big\|_{L^2_\xi}\ls\ep^2\br{t}^{-H_{wa}''(n)\d}2^{-N_{wa}''(n)k^+}2^{\frac{1}{2}k^-}.
\end{align}
\end{proposition}

\bigskip
\section{Bounds on the Nonlinearities of Klein-Gordon Equation} \label{sec:s2t2}

Based on \eqref{wtt 06-diff}, in order to control $\dt\gkg_{\lll}$, it suffices to bound for any $\lll\in\vvv_n$ with $0\leq n\leq N_1$,
\begin{align}\label{decomposition 2}
&\ue^{\ui t\lakg(\xi)}\widehat{\mathcal{N}^{kg}_{\lll}}-\Big\{\ui\,\cci\big(\ue^{\ui\ddi}\widehat{V^{kg}_{\infty}}+\bbf_{\infty}\big)\Big\}_{\lll}-\bbb_{\lll,\infty}\\
    =&\sum_{\lll_1,\lll_2}\bigg\{\sum_{\iota_1,\iota_2\in\{+,-\}}\bbi_{kg}^{\iota_1\iota_2}\Big[G^{kg,\iota_1}_{\lll_1},G^{wa,\iota_2}_{\lll_2}\Big]+\sum_{\iota_1,\iota_2\in\{+,-\}}{\bbi_{kg}^{\iota_1\iota_2}}\Big[\widehat{G^{kg,\iota_1}_{\lll_1}},\widehat{V^{wa,\iota_2}_{\lll_2,\infty}}+\hhf^{\iota_2}_{\lll_2,\infty}\Big]\no\\
    &+\sum_{\iota_1,\iota_2\in\{+,-\}}{\bbi_{kg}^{\iota_1\iota_2}}\Big[\big(\ue^{\iota_1\ui\ddi^{\iota_1}}\widehat{V^{kg,\iota_1}_{\infty}}\big)_{\lll_1}+\bbf^{\iota_1}_{\lll_1,\infty}\,,\widehat{G^{wa,\iota_2}_{\lll_2}}\Big]\bigg\}\no\\
    &+\bigg\{\sum_{\lll_1,\lll_2}\sum_{\iota_1,\iota_2\in\{+,-\}}{\bbi_{kg}^{\iota_1\iota_2}}\Big[\big(\ue^{\iota_1\ui\ddi^{\iota_1}}\widehat{V^{kg,\iota_1}_{\infty}}\big)_{\lll_1}+\bbf^{\iota_1}_{\lll_1,\infty}\,,\widehat{V^{wa,\iota_2}_{\lll_2,\infty}}+\hhf^{\iota_2}_{\lll_2,\infty}\Big]\no\\
    &\quad-\sum_{\lll_1,\lll_2}\ui\,\ccc_{\lll_2,\infty}\big(\ue^{\ui\ddi}\widehat{V^{kg}_{\infty}}\big)_{\lll_1}-\sum_{\lll_1,\lll_2}\ui\,\ccc_{\lll_2,\infty}\bbf_{\lll_1,\infty}-\bbb_{\lll,\infty}\bigg\},\no
\end{align}
$\sum_{\lll_1,\lll_2}$ is a summation over all combinations of vector fields $\lll_1\in\vvv_{n_1}$ and $\lll_2\in\vvv_{n_2}$ with $n_1+n_2= n$ and $\lll=\lll_1\circ\lll_2$. Again, all variables here with $\lll_1$ or $\lll_2$ should be understood in the sense of \eqref{vector field 1} and \eqref{vector field 2}.

\smallskip
\subsection{\texorpdfstring{$S_2'$}{} Estimates} 

\begin{lemma}\label{s lemma 1'}
Assume \eqref{wtt-assumption} and \eqref{wtt-assumption'} hold. 
For any $\lll\in\vvv_n$ with $0\leq n\leq N_1$, we have
\begin{align}
    \bigg\|\sum_{\iota_1,\iota_2\in\{+,-\}}\varphi_k\,\bbi_{kg}^{\iota_1\iota_2}\big[G^{kg,\iota_1}_{\lll_1},G^{wa,\iota_2}_{\lll_2}\big]\bigg\|_{L^2_\xi}\ls\ep^2\br{t}^{-(1+H_{kg}''(n)\d)}2^{-N''_{kg}(n)k^+}.
\end{align}
\end{lemma}

\begin{proof}
For simplicity, we temporarily ignore $\iota_1$ and $\iota_2$ superscripts. Considering Lemma \ref{wtt prelim 1}, it suffices to discuss the cases $k_1\leq k_2$ and $k_2\leq k_1$.
\begin{itemize}
    \item 
    Case I: $k_1\leq k_2$ and $n_1\leq N_1-1$. 
    We use \eqref{ttt 02'} combined with Lemma \ref{wtt lemma 7} and Lemma \ref{wtt lemma 1} to get
    \begin{align}
        &\sum_{k_1,k_2}\tnm{\varphi_k\bbi_{kg}\big[P_{k_1}G^{kg}_{\lll_1},P_{k_2}G^{wa}_{\lll_2}\big]}
        \ls\sum_{k_1,k_2}2^{k_1}2^{-k_2}\bnm{\ue^{\ui t\lakg}P_{k_1}G^{kg}_{\lll_1}}_{L^{\infty}}\bnm{P_{k_2}G^{wa}_{\lll_2}}_{L^2}\\
        \ls&\sum_{k_1,k_2}\Big(\ep\br{t}^{-1}\br{t}^{-H(n_1+1)\d}2^{-(N(n_1+1)-\frac{5}{2})k_1^+}2^{\frac{1}{2}k_1^-}\Big)\Big(\ep\br{t}^{-H(n_2)\d}2^{-N(n_2)k_2^+}2^{\frac{1}{2}k_2^-}\Big)\no\\
        \ls&\sum_{k_1,k_2}\ep^2\br{t}^{-1}\br{t}^{-(H(n_1+1)+H(n_2))\d}2^{-(N(n_1+1)-\frac{5}{2})k_1^+-N(n_2+1)k_2^+}2^{\frac{1}{2}k_1^-+\frac{1}{2}k_2^-}\no\\
        \ls&\;\ep^2\br{t}^{-(1+H(n_1+1)\d+H(n_2)\d)}2^{-\min\{N(n_1+1)-\frac{5}{2}, N(n_2+1)\}k^+}.\no
    \end{align}
    Note that
    \begin{align}
        H(n_1+1)+H(n_2)\geq H''_{kg}(n),\qquad \min\Big\{N(n_1+1)-\frac{5}{2}, N(n_2+1)\Big\}\geq N''_{kg}(n).
    \end{align}
    \item 
    Case II: $k_1\leq k_2$ and $n_1=N_1$. 
    We use \eqref{ttt 02'} combined with Lemma \ref{wtt lemma 1} and Lemma \ref{wtt lemma 7} to obtain
    \begin{align}
        &\sum_{k_1,k_2}\nm{\varphi_k\bbi_{kg}\big[P_{k_1}G^{kg}_{\lll},P_{k_2}G^{wa}\big]}_{L^2}
        \ls\sum_{k_1,k_2}2^{k_1}2^{-k_2}\bnm{P_{k_1}G^{kg}_{\lll}}_{L^2}\bnm{\ue^{\ui t\lawa}P_{k_2}G^{wa}}_{L^{\infty}}\\
        \ls&\sum_{k_1,k_2}\Big(\ep\br{t}^{-H(N_1)\d}2^{-N(N_1)k_1^+}\Big)\Big(\ep\br{t}^{-1}\br{t}^{-H(1)\d}\ln\!\br{t}2^{-N(1)k_2^+}2^{k_2^-}\Big)\no\\
        \ls&\sum_{k_1,k_2}\ep^2\br{t}^{-1}\br{t}^{-(H(N_1)+H(1))\d}\ln\!\br{t}2^{-N(N_1)k_1^+-N(1)k_2^+}2^{k_2^-}\no\\
        \ls&\;\ep^2\br{t}^{-(1+H(N_1)\d+H(1)\d)}\ln\!\br{t}2^{-\min\{N(N_1), N(1)\}k^+}.\no
    \end{align}
    Note that
    \begin{align}
        H(N_1)+H(1)\geq H''_{kg}(N_1),\qquad \min\big\{N(N_1), N(1)\big\}\geq N''_{kg}(N_1).
    \end{align}
    \item 
    Case III: $k_2\leq k_1$ and $n_2\leq N_1-1$. 
    We use \eqref{ttt 02'} combined with Lemma \ref{wtt lemma 7} and Lemma \ref{wtt lemma 1} to get
    \begin{align}
        &\sum_{k_1,k_2}\tnm{\varphi_k\bbi_{kg}\big[P_{k_1}G^{kg}_{\lll_1},P_{k_2}G^{wa}_{\lll_2}\big]}
        \ls\sum_{k_1,k_2}2^{k_1}2^{-k_2}\bnm{P_{k_1}G^{kg}_{\lll_1}}_{L^2}\bnm{\ue^{\ui t\lawa}P_{k_2}G^{wa}_{\lll_2}}_{L^{\infty}}\\
        \ls&\sum_{k_1,k_2}2^{k_1}2^{-k_2}\Big(\ep\br{t}^{-H(n_1)\d}2^{-N(n_1)k_1^+}\Big)\Big(\ep\br{t}^{-1}\br{t}^{-H(n_2+1)\d}\ln\!\br{t}2^{-N(n_2+1)k_2^+}2^{k_2^-}\Big)\no\\
        \ls&\sum_{k_1,k_2}\ep^2\br{t}^{-1}\br{t}^{-(H(n_1)+H(n_2+1))\d}\ln\!\br{t}2^{-(N(n_1)-1)k_1^+-N(n_2+1)k_2^+}2^{k_1^-}\no\\
        \ls&\;\ep^2\br{t}^{-(1+H(n_1)\d+H(n_2+1)\d)}\ln\!\br{t}2^{-\min\{N(n_1)-1, N(n_2+1)\}k^+}.\no
    \end{align}
    Note that
    \begin{align}
        H(n_1)+H(n_2+1)\geq H''_{kg}(n),\qquad \min\big\{N(n_1)-1, N(n_2+1)\big\}\geq N''_{kg}(n).
    \end{align}
    \item 
    Case IV: $k_2\leq k_1$ and $n_2=N_1$. 
    This is the most complicated case. The key is to control the singular term $2^{-k_2^-}$. We need to bound
    \begin{align}
     \sum_{k_1,k_2}\Bnm{\varphi_k\bbi_{kg}\big[P_{k_1}G^{kg},P_{k_2}G^{wa}_{\lll}\big]}_{L^2}.
    \end{align}
    The only possible control of $P_{k_2}G^{wa}_{\lll}$ is through Lemma \ref{wtt lemma 1} as
    \begin{align}
        \btnm{P_{k_2}G^{wa}_{\lll}}\ls \ep\br{t}^{-H(N_1)\d}2^{-N(N_1)k_2^+}2^{\frac{1}{2}k_2^-}.
    \end{align}
    This is not enough and only provides $2^{\frac{1}{2}k_2^-}$. Then it remains to bound $P_{k_1}G^{kg}$ part and provide another $2^{\frac{1}{2}k_2^-}$. Notice that this has been achieved in the wave nonlinear estimate (Lemma \ref{s lemma 1}). We just need to discuss the four cases (replacing $k^-$ by $k_2^-$ there) to get the desired estimate
    \begin{align}
     &\sum_{k_1,k_2}\Bnm{\varphi_k\bbi_{kg}\big[P_{k_1}G^{kg},P_{k_2}G^{wa}_{\lll}\big]}_{L^2}
     \ls\ep^2\br{t}^{-(1+H_{kg}''(n)\d)}2^{-N''_{kg}(n)k^+}.
    \end{align}
    This suffices to close the proof.
\end{itemize}
\end{proof}

\begin{lemma}\label{s lemma 2'}
Assume \eqref{wtt-assumption} and \eqref{wtt-assumption'} hold. 
For any $\lll\in\vvv_n$ with $0\leq n\leq N_1$, we have
\begin{align}
    \bigg\|\sum_{\iota_1,\iota_2\in\{+,-\}}\varphi_k\,{\bbi_{kg}^{\iota_1\iota_2}}\Big[\widehat{G^{kg,\iota_1}_{\lll_1}},\widehat{V^{wa,\iota_2}_{\lll_2,\infty}}+\hhf^{\iota_2}_{\lll_2,\infty}\Big]\bigg\|_{L^2_\xi}&\ls\ep^{2}\br{t}^{-(1+H(n)\d)}2^{-N(n)k^+},\\
    \bigg\|\sum_{\iota_1,\iota_2\in\{+,-\}}\varphi_k\,{\bbi_{kg}^{\iota_1\iota_2}}\Big[\big(\ue^{\iota_1\ui\ddi^{\iota_1}}\widehat{V^{kg,\iota_1}_{\infty}}\big)_{\lll_1}+\bbf^{\iota_1}_{\lll_1,\infty}\,,\widehat{G^{wa,\iota_2}_{\lll_2}}\Big]\bigg\|_{L^2_\xi}&\ls\ep^{2}\br{t}^{-(1+H(n)\d)}2^{-N(n)k^+}.
\end{align}
\end{lemma}

\begin{proof}
The proof is similar to that of Lemma \ref{s lemma 1'}. The only difference is that we do not have $H(n_2)$ time decay, so we should always assign $L^2$ to $G^{kg,\iota_1}_{\lll_1}$ or $G^{wa,\iota_2}_{\lll_2}$. The other part $V^{wa,\iota_2}_{\lll_2,\infty}+\hhf^{\iota_2}_{\lll_2,\infty}$ or $\big(\ue^{\iota_1\ui\ddi^{\iota_1}}\widehat{V^{kg,\iota_1}_{\infty}}\big)_{\lll_1}+\bbf_{\lll_1,\infty}$ may take $L^2$ or $L^{\infty}$ following the argument in the proof of Lemma\;\ref{s lemma 1'}. Then based on the similar analysis as in Lemma \ref{s lemma 2} and using the results from Lemma \ref{vtt lemma 12}, Lemma \ref{vtt lemma 8'} and Lemma \ref{vtt lemma 18'}, we get the desired result.
\end{proof}

\begin{lemma}\label{s lemma 3'}
Assume \eqref{wtt-assumption} and \eqref{wtt-assumption'} hold. 
For any $\lll\in\vvv_n$ with $0\leq n\leq N_1$, we have
\begin{align}
    &\Bigg\|\varphi_k\bigg\{\sum_{\lll_1,\lll_2}\sum_{\iota_1,\iota_2\in\{+,-\}}{\bbi_{kg}^{\iota_1\iota_2}}\Big[\big(\ue^{\iota_1\ui\ddi^{\iota_1}}\widehat{V^{kg,\iota_1}_{\infty}}\big)_{\lll_1}+\bbf^{\iota_1}_{\lll_1,\infty}\,,\widehat{V^{wa,\iota_2}_{\lll_2,\infty}}+\hhf^{\iota_2}_{\lll_2,\infty}\Big]\\
    &\quad-\sum_{\lll_1,\lll_2}\ui\,\ccc_{\lll_2,\infty}\big(\ue^{\ui\ddi}\widehat{V^{kg}_{\infty}}\big)_{\lll_1}-\sum_{\lll_1,\lll_2}\ui\,\ccc_{\lll_2,\infty}\bbf_{\lll_1,\infty}-\bbb_{\lll,\infty}\bigg\}\Bigg\|_{L^2_\xi}\no\\
    \ls&\;\ep^{2}\br{t}^{-(1+H(n)\d)}2^{-N(n)k^+}.\no
\end{align}
\end{lemma}

\begin{proof}
Based on the proof of \cite[Lemma 4.2]{Deng.Pusateri2020}, we have (up to fast-decaying terms)
\begin{align}
    \ccc_{\lll,\infty}(t,\xi)\sim \frac{1}{2}(2\pi)^{-\frac{3}{2}}\frac{\abs{\xi}^2}{\br{\xi}}\cdot \Im\left\{\int_{\mathbb{R}^3} \text{\large$\ue$}^{\ui t\left(\frac{\xi\cdot\eta}{\br{\xi}}-\abs{\eta}\right)}\frac{1}{\abs{\eta}} H_{\lll,\infty}(t,\eta)\,\dd\eta\right\}.
\end{align}
We may decompose
\begin{align}
     &\;{\bbi_{kg}^{\iota_1\iota_2}}\Big[\big(\ue^{\iota_1\ui\ddi^{\iota_1}}\widehat{V^{kg,\iota_1}_{\infty}}\big)_{\lll_1}\,,\big(\widehat{V^{wa,\iota_2}_{\lll_2,\infty}}+{\hhf}^{\iota_2}_{\lll_2,\infty}\big)\Big]\\
     =&\;{\bbi_{kg}^{\iota_1\iota_2}}\Big[\big(\ue^{\iota_1\ui\ddi^{\iota_1}}\widehat{V^{kg,\iota_1}_{\infty}}\big)_{\lll_1}\,,\varphi_{\leq0}\big(\eta\br{t}^{\frac{7}{8}}\big)\big(\widehat{V^{wa,\iota_2}_{\lll_2,\infty}}+{\hhf}^{\iota_2}_{\lll_2,\infty}\big)\Big]+{\bbi_{kg}^{\iota_1\iota_2}}\Big[\big(\ue^{\iota_1\ui\ddi^{\iota_1}}\widehat{V^{kg,\iota_1}_{\infty}}\big)_{\lll_1}\,,\varphi_{\geq0}\big(\eta\br{t}^{\frac{7}{8}}\big)\big(\widehat{V^{wa,\iota_2}_{\lll_2,\infty}}+{\hhf}^{\iota_2}_{\lll_2,\infty}\big)\Big].\no
\end{align}

\noindent{\it \underline{Step\;1}. Low Frequency - Resonant Case $(\iota_1,\iota_2)=(+,+)$ or $(\iota_1,\iota_2)=(+,-)$}

Let $\widehat{g_1}=\big(\ue^{\iota_1\ui\ddi^{\iota_1}}\widehat{V^{kg,\iota_1}_{\infty}}\big)_{\lll_1}$ and $\widehat{g_2}=\varphi_{\leq0}\big(\eta\br{t}^{\frac{7}{8}}\big)\big(\widehat{V^{wa,\iota_2}_{\lll_2,\infty}}+{\hhf}^{\iota_2}_{\lll_2,\infty}\big)=H_{\lll_2,\infty}^{\iota_2}$.
Then we have
\begin{align}
    \sum_{\iota_2\in\{+,-\}}{\bbi_{kg}^{+\iota_2}}\big[\widehat{g_1},\widehat{g_2}\big]-\ui\,\ccc_{\lll_2,\infty}\widehat{g_1}
    \simeq\sum_{\iota_2\in\{+,-\}}{\bbi_{kg}^{+\iota_2}}\big[\widehat{g_1},\widehat{g_2}\big]-\widehat{g_1}\sum_{\iota_2\in\{+,-\}}\iota_2\int_{\r^3}\ue^{\ui t\left(\frac{\xi\cdot\eta}{\br{\xi}}-\iota_2\eta\right)}\frac{\abs{\xi}^2}{\br{\xi}}\frac{1}{\abs{\eta}}\widehat{g_2}(t,\eta)\,\ud\eta.
\end{align}
We only focus on the $(\iota_1,\iota_2)=(+,+)$ case. 
The $(\iota_1,\iota_2)=(+,-)$ case can be handled in a similar way.
It suffices to consider
\begin{align}
    &{\bbi_{kg}^{++}}\big[\widehat{g_1},\widehat{g_2}\big]-\widehat{g_1}\int_{\r^3}\ue^{\ui t\left(\frac{\xi\cdot\eta}{\br{\xi}}-\eta\right)}\frac{\abs{\xi}^2}{\br{\xi}}\frac{1}{\abs{\eta}}\widehat{g_2}(t,\eta)\,\ud\eta\\
    \simeq&\int_{\mathbb{R}^3} \ue^{\ui t\Phi_{\,kg}^{++}(\xi,\eta)} b_{++}(\xi,\eta) \widehat{g_1}(t,\xi\!-\!\eta)\widehat{g_2}(t,\eta)\,\dd\eta-\widehat{g_1}(t,\xi)\int_{\r^3}\ue^{\ui t\left(\frac{\xi\cdot\eta}{\br{\xi}}-\eta\right)}\frac{\abs{\xi}^2}{\br{\xi}}\frac{1}{\abs{\eta}}\widehat{g_2}(t,\eta)\,\ud\eta\no\\
    =&\;\rr_1+\rr_2+\rr_3,\no
\end{align}
where
\begin{align}
    \rr_1&:=\int_{\mathbb{R}^3} \ue^{\ui t\Phi_{\,kg}^{++}(\xi,\eta)} \bigg[b_{++}(\xi,\eta)-\frac{\abs{\xi}^2}{\br{\xi}}\frac{1}{\abs{\eta}}\bigg] \widehat{g_1}(t,\xi\!-\!\eta)\widehat{g_2}(t,\eta)\,\dd\eta,\\
    \rr_2&:=\int_{\mathbb{R}^3} \Big[\ue^{\ui t\Phi_{\,kg}^{++}(\xi,\eta)}-\ue^{\ui t\left(\frac{\xi\cdot\eta}{\br{\xi}}-\eta\right)}\Big] \frac{\abs{\xi}^2}{\br{\xi}}\frac{1}{\abs{\eta}}\widehat{g_1}(t,\xi\!-\!\eta)\widehat{g_2}(t,\eta)\,\dd\eta,\\
    \rr_3&:=\int_{\mathbb{R}^3} \ue^{\ui t\left(\frac{\xi\cdot\eta}{\br{\xi}}-\eta\right)} \frac{\abs{\xi}^2}{\br{\xi}}\frac{1}{\abs{\eta}}\Big[\widehat{g_1}(t,\xi\!-\!\eta)-\widehat{g_1}(t,\xi)\Big]\widehat{g_2}(t,\eta)\,\dd\eta.
\end{align}
Note that due to the cutoff in $\widehat{g_2}$, all integral over $\eta$ is restricted to $\varphi_{\leq0}\big(\eta\br{t}^{\frac{7}{8}}\big)$. For $\rr_1$, note that 
\begin{align}
    \abs{\frac{\abs{\xi}^2}{\br{\xi}}-\frac{\abs{\xi\!-\!\eta}^2}{\br{\xi\!-\!\eta}}}\ls \abs{\eta}.
\end{align}
Using \eqref{ttt 01'}, Lemma \ref{vtt lemma 14} and Lemma \ref{vtt lemma 10}, we obtain 
\begin{align}
    \btnm{\varphi_k\rr_1}\ls 2^{\frac{3}{2}k_2}\tnm{\widehat{g_1}}\tnm{\widehat{g_2}}\ls \ep^3\br{t}^{-\frac{21}{16}}.
\end{align}
For $\rr_2$,
note that 
\begin{align}
    \abs{\ue^{\ui t\Phi_{\,kg}^{++}(\xi,\eta)}-\ue^{\ui t\left(\frac{\xi\cdot\eta}{\br{\xi}}-\eta\right)}}\ls \abs{\eta}^2.
\end{align}
Using \eqref{ttt 01'}, Lemma \ref{vtt lemma 14} and Lemma \ref{vtt lemma 10}, we obtain 
\begin{align}
    \btnm{\varphi_k\rr_2}\ls 2^{\frac{3}{2}k_2}2^{k_2}\tnm{\widehat{g_1}}\tnm{\widehat{g_2}}\ls \ep^3\br{t}^{-\frac{35}{16}}.
\end{align}
$\rr_3$ is a bit more delicate. Minkowski's integral inequality implies
\begin{align}
    \bnm{\widehat{g_1}(\xi\!-\!\eta)-\widehat{g_1}(\xi)}_{L^2}
    \ls\bigg\|\int_0^{\abs{\eta}}\p_{\hat{\eta}}\widehat{g_1}(\xi-c\hat{\eta})\,\ud c\bigg\|_{L^2_{\xi}}
    \ls\int_0^{\abs{\eta}}\bnm{\p_{\hat{\eta}}\widehat{g_1}(\xi-c\hat{\eta})}_{L^2_{\xi}}\ud c\ls\int_0^{\abs{\eta}}\bnm{\p_{\xi}\widehat{g_1}}_{L^2_{\xi}}\ud c
    \,\ls\,\abs{\eta}\bnm{\p_{\xi}\widehat{g_1}}_{L^2_{\xi}}.
\end{align}
Then we may use \eqref{ttt 01'}, Lemma \ref{vtt lemma 13'} and Lemma \ref{vtt lemma 14} to obtain
\begin{align}
    \btnm{\varphi_k\rr_3}\ls 2^{\frac{3}{2}k}\varphi_{\leq0}\big(2^k\!\br{t}^{\frac{7}{8}}\big)\nm{\widehat{g_1}}_{H^1_{\eta}}\tnm{\widehat{g_2}}\ls \ep^3\br{t}^{-\frac{21}{16}}\big(\ln\!\br{t}\big)^2.
\end{align}
In total, we have 
\begin{align}
    \bigg\|{\bbi_{kg}^{++}}\big[\widehat{g_1},\widehat{g_2}\big]-\widehat{g_1}\int_{\r^3}\ue^{\ui t\left(\frac{\xi\cdot\eta}{\br{\xi}}-\eta\right)}\frac{\abs{\xi}^2}{\br{\xi}}\frac{1}{\abs{\eta}}\widehat{g_2}(t,\eta)\,\ud\eta\bigg\|_{L^2}\ls \ep^3\br{t}^{-\frac{21}{16}}\big(\ln\!\br{t}\big)^2.
\end{align}

\noindent{\it \underline{Step\;2}. Low Frequency - Non-Resonant Case $(\iota_1,\iota_2)=(-,+)$ or $(\iota_1,\iota_2)=(-,-)$}

By the definition of $\bbb_{\infty}$, we directly have
\begin{align}
    \sum_{\lll_1,\lll_2}\sum_{\iota_2\in\{+,-\}}{\bbi_{kg}^{-\iota_2}}\big[\widehat{g_1},\widehat{g_2}\big]-\bbb_{\lll,\infty}=0.
\end{align}

\noindent{\it \underline{Step\;3}. High Frequency}

Since the high-frequency cutoff will eliminate $\hhf_{\infty}$, we let $\widehat g_1=\ue^{\iota_1\ui\ddi^{\iota_1}}\widehat{V^{kg}_{\lll_1,\infty}}$ and $\widehat g_2=\widehat{V^{wa,\iota_2}_{\lll_2,\infty}}$. Then it suffices to bound 
\begin{align}
    &\bigg\|\varphi_k\int_{\mathbb{R}^3} \varphi_{\geq0}\big(\eta\br{t}^{\frac{7}{8}}\big)\ue^{\ui t(\br{\xi}-\iota_1\br{\xi-\eta}-\iota_2\abs{\eta})} \frac{\abs{\xi\!-\!\eta}^2}{\br{\xi\!-\!\eta}\abs{\eta}} \widehat{g_1}(t,\xi\!-\!\eta)\widehat{g_2}(\eta)\,\ud\eta\bigg\|_{L^2}\\
    \ls&\sum_{k_1,k_2}\bigg\|\varphi_k\int_{\mathbb{R}^3} \varphi_{\geq0}\big(\eta\br{t}^{\frac{7}{8}}\big)\ue^{\ui t(\br{\xi}-\iota_1\br{\xi-\eta}-\iota_2\abs{\eta})} \frac{\abs{\xi\!-\!\eta}^2}{\br{\xi\!-\!\eta}\abs{\eta}} \widehat{P_{k_1}g_1}(t,\xi\!-\!\eta)\widehat{P_{k_2}g_2}(\eta)\,\ud\eta\bigg\|_{L^2}.\no
\end{align}
\begin{itemize}
\item Case I: If $2^{2k_1^-}\!\!\br{t}^{\frac{1}{8}}\ls 1$, then using (\ref{ttt 02'}), we have
\begin{align}
    &\sum_{k_1,k_2}\bigg\|\varphi_k\int_{\mathbb{R}^3} \varphi_{\geq0}\big(\eta\br{t}^{\frac{7}{8}}\big)\ue^{\ui t(\br{\xi}-\iota_1\br{\xi-\eta}-\iota_2\abs{\eta})} \frac{\abs{\xi\!-\!\eta}^2}{\br{\xi\!-\!\eta}\abs{\eta}} \widehat{P_{k_1}g_1}(t,\xi\!-\!\eta)\widehat{P_{k_2}g_2}(\eta)\,\ud\eta\bigg\|_{L^2}\\
    \ls&\sum_{k_1,k_2}2^{k_1^++2k_1^-}2^{-k_2}\btnm{\widehat{P_{k_1}g_1}}\blnm{\ue^{-\iota_2\ui t\lawa}P_{k_2}g_2}\no\\
    \ls&\sum_{k_1,k_2}2^{k_1^++2k_1^-}2^{-k_2}\Big(\ep\big(\ln\!\br{t}\big)^2 2^{-N(n_1-4)k_1^+}2^{k_1^-}\Big)\Big(\ep\br{t}^{-1}2^{-N(n_2-2)k_2^+}2^{k_2}\Big)\no\\
    \ls& \sum_{k_1,k_2}\ep^2\br{t}^{-1}\!\big(\ln\!\br{t}\big)^2 2^{-N(n_1-4)k_1^+ +k_1^++3k_1^- -N(n_2-2)k_2^+}\no\\
    \ls&\; \ep^2\br{t}^{-\frac{9}{8}}\!\big(\ln\!\br{t}\big)^2\cdot \!\!\sum_{\max(k_1,k_2)\geq k}\!\! 2^{-N(n_1-4)k_1^+ +k_1 -N(n_2-2)k_2^+}
    \ls\; \ep^{2}\br{t}^{-(1+H(n)\d)}2^{-N(n)k^+} .\no
\end{align}
Here in the second inequality we use Lemma \ref{vtt lemma 1} and Lemma \ref{vtt lemma 8'}.
In the last inequality we may take the summation over $k_1,k_2$\, for \,$\max(k_1,k_2)\geq k$\, in view of Lemma\;\ref{wtt prelim 1}.

\item Case II: If $2^{2k_1^-}\!\!\br{t}^{\frac{1}{8}}\gs 1$ and $k_1\gs k_2$ and $k_2\leq 0$, then we use Lemma \ref{vtt lemma 14''} and (\ref{wtt 7-1}) to obtain
\begin{align}
    &\sum_{k_1,k_2}\bigg\|\varphi_k\int_{\mathbb{R}^3} \varphi_{\geq0}\big(\eta\br{t}^{\frac{7}{8}}\big)\ue^{\ui t(\br{\xi}-\iota_1\br{\xi-\eta}-\iota_2\abs{\eta})} \frac{\abs{\xi\!-\!\eta}^2}{\br{\xi\!-\!\eta}\abs{\eta}} \widehat{P_{k_1}g_1}(t,\xi\!-\!\eta)\widehat{P_{k_2}g_2}(\eta)\,\ud\eta\bigg\|_{L^2}\\
    \ls&\sum_{k_1,k_2}2^{k_1^+}\blnm{\ue^{-\ui t\lakg}P_{k_1}g_1}\left\|\frac{1}{\abs{\eta}}\widehat{P_{k_2}g_2}\right\|_{L^2_\eta}
    \ls\sum_{k_1,k_2}2^{k_1^+}\blnm{\ue^{-\ui t\lakg}P_{k_1}g_1}\left\|P_{k_2}\frac{1}{\abs{\eta}}\right\|_{L^2_\eta}\bnm{\widehat{P_{k_2}g_2}}_{L^{\infty}_{\eta}}\no\\
    \ls& \sum_{k_1,k_2}\ep^2\big(\ln\!\br{t}\big)^2\br{t}^{-\frac{3}{2}+\frac{\sigma}{8}}2^{-(N(n_1)-5)k_1^+-N(n_2)k_2^+}2^{-k_1^-(\frac{1}{2}-\frac{\sigma}{4})-\frac{1}{2}k_2^-}\no\\
    \ls& \sum_{k_1,k_2}\ep^2\big(\ln\!\br{t}\big)^2\br{t}^{-\frac{3}{2}+\frac{\sigma}{8}}2^{-(N(n_1)-5)k_1^+-N(n_2)k_2^+}\br{t}^{\frac{1}{16}(\frac{1}{2}-\frac{\sigma}{4})}\br{t}^{-\frac{7}{16}}\no\\
    \ls&\sum_{k_1,k_2}\ep^2\br{t}^{-\frac{33}{32}}\big(\ln\!\br{t}\big)^2\, 2^{-(N(n_1)-5)k_1^+-N(n_2)k_2^+}.\no
\end{align}
\item Case III: If $2^{2k_1^-}\!\!\br{t}^{\frac{1}{8}}\gs 1$ and $k_1\gs k_2$ and $k_2\geq 0$, then
\begin{align}
    &\sum_{k_1,k_2}\bigg\|\varphi_k\int_{\mathbb{R}^3} \varphi_{\geq0}\big(\eta\br{t}^{\frac{7}{8}}\big)\ue^{\ui t(\br{\xi}-\iota_1\br{\xi-\eta}-\iota_2\abs{\eta})} \frac{\abs{\xi\!-\!\eta}^2}{\br{\xi\!-\!\eta}\abs{\eta}} \widehat{P_{k_1}g_1}(t,\xi\!-\!\eta)\widehat{P_{k_2}g_2}(\eta)\,\ud\eta\bigg\|_{L^2}\\
    \ls&\sum_{k_1,k_2}2^{k_1^+}\blnm{\ue^{-\ui t\lakg}P_{k_1}g_1}\btnm{\widehat{P_{k_2}g_2}}
    \ls\sum_{k_1,k_2}\ep^2\big(\ln\!\br{t}\big)^2\br{t}^{-\frac{3}{2}+\frac{\sigma}{8}}2^{-(N(n_1)-5)k_1^+-N(n_2)k_2^+}2^{-k_1^-(\frac{1}{2}-\frac{\sigma}{4})}\no\\
    \ls&\sum_{k_1,k_2}\ep^2\br{t}^{-\frac{47}{32}}\big(\ln\!\br{t}\big)^2\, 2^{-(N(n_1)-5)k_1^+-N(n_2)k_2^+}.\no
\end{align}
\item Case IV: If $2^{2k_1^-}\!\!\br{t}^{\frac{1}{8}}\gs 1$ and $k_1\ls k_2$, then we have
\begin{align}
    \frac{\abs{\xi\!-\!\eta}^2}{\br{\xi\!-\!\eta}\abs{\eta}}\ls 1.
\end{align}
We first integrate by parts in $\eta$ using \eqref{resonance-space}, and then use Lemma \ref{vtt lemma 1}, Lemma \ref{vtt lemma 14} and (\ref{wtt 7-1}) to obtain that
\begin{align}
    &\sum_{k_1,k_2}\bigg\|\varphi_k\int_{\mathbb{R}^3} \varphi_{\geq0}\big(\eta\br{t}^{\frac{7}{8}}\big)\ue^{\ui t(\br{\xi}-\iota_1\br{\xi-\eta}-\iota_2\abs{\eta})} \frac{\abs{\xi\!-\!\eta}^2}{\br{\xi\!-\!\eta}\abs{\eta}} \widehat{P_{k_1}g_1}(t,\xi\!-\!\eta)\widehat{P_{k_2}g_2}(\eta)\,\ud\eta\bigg\|_{L^2}\\
    \ls&\sum_{k_1,k_2}\br{t}^{-1}\bigg\|\varphi_k\int_{\mathbb{R}^3} \varphi_{\geq0}\big(\eta\br{t}^{\frac{7}{8}}\big)\ue^{\ui t(\br{\xi}-\iota_1\br{\xi-\eta}-\iota_2\abs{\eta})} \frac{\abs{\xi\!-\!\eta}^2}{\br{\xi\!-\!\eta}\abs{\eta}^3} \widehat{P_{k_1}g_1}(t,\xi\!-\!\eta)\widehat{P_{k_2}g_2}(\eta)\,\ud\eta\bigg\|_{L^2}\no\\
    \ls&\sum_{k_1,k_2}\br{t}^{-1}\btnm{\widehat{P_{k_1}g_1}}2^{-k_2}\blnm{\ue^{-\ui t\lawa}P_{k_2}g_2}
    \ls\sum_{k_1,k_2}\ep^2\br{t}^{-2}\big(\ln\!\br{t}\big)^3\, 2^{-N(n_1)k_1^+-N(n_2)k_2^+}.\no
\end{align}
For the term with derivative after integration by parts, we always assign $L^{2}$, and the other term with $L^{\infty}$. Note that now $2^{-k_2}$, even present, is not a big deal since it grows at most $\br{t}^{\frac{1}{8}}$.
\end{itemize}

\noindent{\it \underline{Step\;4}. $\bbf_{\lll,\infty}$ Terms}

In bounding
\begin{align}
   \sum_{\lll_1,\lll_2}\varphi_k\bigg\{\sum_{\iota_1,\iota_2\in\{+,-\}}{\bbi_{kg}^{\iota_1\iota_2}}\Big[\bbf^{\iota_1}_{\lll_1,\infty}\,,\widehat{V^{wa,\iota_2}_{\lll_2,\infty}}+\hhf^{\iota_2}_{\lll_2,\infty}\Big]-\ui\,\ccc_{\lll_2,\infty}\bbf_{\lll_1,\infty}\bigg\},
\end{align}
all terms related to $\bbf_{\lll_1,\infty}$ can be controlled as in Lemma \ref{s lemma 1'} since $\bbf_{\lll_1,\infty}$ has sufficient time decay as in Lemma \ref{vtt lemma 18} and Lemma \ref{vtt lemma 18'}. 
\end{proof}

Combining Lemma \ref{s lemma 1'}, Lemma \ref{s lemma 2'} and Lemma \ref{s lemma 3'}, we get the following proposition:

\begin{proposition}\label{nonlinear 1'}
Assume \eqref{wtt-assumption} and \eqref{wtt-assumption'} hold. 
For any $\lll\in\vvv_n$ with $0\leq n\leq N_1$, we have
\begin{align}
    \big\|\varphi_k\dt\widehat{\llgkg}\big\|_{L^2_\xi}\ls\ep^2\br{t}^{-(1+H_{kg}''(n)\d)}2^{-N''_{kg}(n)k^+}.
\end{align}
\end{proposition}

\smallskip
\subsection{\texorpdfstring{$T_2'$}{} Estimates} 

\begin{lemma}\label{s lemma 4'}
Assume \eqref{wtt-assumption} and \eqref{wtt-assumption'} hold. 
For any $\lll\in\vvv_n$ with $0\leq n\leq N_1-1$ and $\ell\in\{1,2,3\}$, we have
\begin{align}
    \bigg\|\varphi_k\p_{\xi_{\ell}}\Big\{\ue^{-\ui t\lakg(\xi)}\,\bbi_{kg}^{\iota_1\iota_2}\big[G^{kg,\iota_1}_{\lll_1},G^{wa,\iota_2}_{\lll_2}\big]\Big\}\bigg\|_{L^2_\xi}\ls\ep^2\br{t}^{-H_{kg}''(n)\d}2^{-N''_{kg}(n)k^+}.
\end{align}
\end{lemma}

\begin{proof}
The proof is similar to that of Lemma \ref{s lemma 1'} and of \cite[Lemma 4.3]{Ionescu.Pausader2019}. Since
\begin{align}
    \bbi_{kg}^{\iota_1\iota_2}\big[G^{kg,\iota_1}_{\lll_1},G^{wa,\iota_2}_{\lll_2}\big](t,\xi)\sim \int_{\r^3}\ue^{\ui t\Phi_{kg}^{\iota_1\iota_2}}b_{\iota_1\iota_2}(\xi,\eta)\widehat{G^{kg,\iota_1}_{\lll_1}}(t,\xi\!-\!\eta)\widehat{G^{wa,\iota_2}_{\lll_2}}(t,\eta)\,\ud\eta,
\end{align}
$\xi_{\ell}$ derivative may hit the phase $\ue^{-\ui t\lakg}\ue^{\ui t\Phi_{kg}^{\iota_1\iota_2}}$, the multiplier $b_{\iota_1\iota_2}$, or $G^{kg,\iota_1}_{\lll_1}(\xi\!-\!\eta)$.
\begin{itemize}
    \item 
    If $\xi_{\ell}$ derivative hits the phase $\ue^{-\ui t\lakg}\ue^{\ui t\Phi_{kg}^{\iota_1\iota_2}}$, then we have an extra $\br{t}$ term popping out. Following exactly the same argument as in the proof of Lemma \ref{s lemma 1'}, we get the desired result.
    \item
     If $\xi_{\ell}$ derivative hits the multiplier $b_{\iota_1\iota_2}$, then nothing happens and we follow the same proof of Lemma \ref{s lemma 1'}.
     \item
     If $\xi_{\ell}$ derivative hits $\widehat{G^{kg,\iota_1}_{\lll_1}}(\xi\!-\!\eta)$, then we simply use \eqref{ttt 01'} and Lemma \ref{wtt lemma 1} to obtain
     \begin{align}
         \tnm{\varphi_k{\bbi_{kg}^{\iota_1\iota_2}}\Big[\varphi_{k_1}\p_{\xi_{\ell}}\widehat{G^{kg,\iota_1}_{\lll_1}},\varphi_{k_2}\widehat{G^{wa,\iota_2}_{\lll_2}}\Big]}
         \ls&\;2^{\frac{3}{2}\min\{k_1,k_2,k\}}2^{-k_2}\Btnm{\varphi_{k_1}\p_{\xi_{\ell}}\widehat{G^{kg,\iota_1}_{\lll_1}}}\tnm{\varphi_{k_2}\widehat{G^{wa,\iota_2}_{\lll_2}}}\\
         \ls&\;\ep^2\br{t}^{-(H(n_1+1)+H(n_2))\d} 2^{\frac{3}{2}\min\{k_1,k_2,k\}}2^{-\frac{1}{2}k_2}2^{-(N(n_1+1)+1)k_1^+-N(n_2)k_2^+}.\no
     \end{align}
     This works for $k_1\leq k_2$. If $k_1\geq k_2$, then through change of variables in the convolution, we can always transfer the derivative to $G^{wa,\iota_2}_{\lll_2}$, so we know both $H_{kg}''(n)$ and $N_{kg}''(n)$ requirements can be met.
\end{itemize}
\end{proof}

\begin{lemma}\label{s lemma 5'}
Assume \eqref{wtt-assumption} and \eqref{wtt-assumption'} hold. 
For any $\lll\in\vvv_n$ with $0\leq n\leq N_1-1$ and $\ell\in\{1,2,3\}$, we have
\begin{align}
    &\bigg\|\varphi_k\p_{\xi_{\ell}}\bigg\{\ue^{-\ui t\lakg(\xi)}\,{\bbi_{kg}^{\iota_1\iota_2}}\Big[\widehat{G^{kg,\iota_1}_{\lll_1}},\widehat{V^{wa,\iota_2}_{\lll_2,\infty}}+\hhf^{\iota_2}_{\lll_2,\infty}\Big]\bigg\}\bigg\|_{L^2_\xi} 
    \ls\ep^{2}\br{t}^{-H_{kg}''(n)\d}2^{-N''_{kg}(n)k^+},\\
    &\bigg\|\varphi_k\p_{\xi_{\ell}}\bigg\{\ue^{-\ui t\lakg(\xi)}\,{\bbi_{kg}^{\iota_1\iota_2}}\Big[\big(\ue^{\iota_1\ui\ddi^{\iota_1}}\widehat{V^{kg,\iota_1}_{\infty}}\big)_{\lll_1}+\bbf^{\iota_1}_{\lll_1,\infty}\,,\widehat{G^{wa,\iota_2}_{\lll_2}}\Big]\bigg\}\bigg\|_{L^2_\xi} 
    \ls\ep^{2}\br{t}^{-H_{kg}''(n)\d}2^{-N''_{kg}(n)k^+}.
\end{align}
\end{lemma}

\begin{proof}
It is similar to that of Lemma \ref{s lemma 2'} and Lemma \ref{s lemma 4'}.
\end{proof}

\begin{lemma}\label{s lemma 6'}
Assume \eqref{wtt-assumption} and \eqref{wtt-assumption'} hold. 
For any $\lll\in\vvv_n$ with $0\leq n\leq N_1-1$ and $\ell\in\{1,2,3\}$, we have
\begin{align}
    &\Bigg\|\varphi_k\p_{\xi_{\ell}}\bigg\{\ue^{-\ui t\lakg(\xi)}\bigg(\sum_{\lll_1,\lll_2}\sum_{\iota_1,\iota_2=\{+,-\}}{\bbi_{kg}^{\iota_1\iota_2}}\Big[\big(\ue^{\iota_1\ui\ddi^{\iota_1}}\widehat{V^{kg,\iota_1}_{\infty}}\big)_{\lll_1}+\bbf^{\iota_1}_{\lll_1,\infty}\,,\widehat{V^{wa,\iota_2}_{\lll_2,\infty}}+\hhf^{\iota_2}_{\lll_2,\infty}\Big]\\
    &\quad-\sum_{\lll_1,\lll_2}\ui\ccc_{\lll_2,\infty}\big(\ue^{\ui\ddi}\widehat{V^{kg}_{\infty}}\big)_{\lll_1}-\sum_{\lll_1,\lll_2}\ui\,\ccc_{\lll_2,\infty}\bbf_{\lll_1,\infty}-\bbb_{\lll,\infty}\bigg)\bigg\}\Bigg\|_{L^2_\xi}
    \ls\ep^2\br{t}^{-H_{kg}''(n)\d}2^{-N''_{kg}(n)k^+}.\no
\end{align}
\end{lemma}

\begin{proof} 
It is similar to that of Lemma \ref{s lemma 3'} and Lemma \ref{s lemma 4'}.
\end{proof}

Combining Lemma \ref{s lemma 4'}, Lemma \ref{s lemma 5'} and Lemma \ref{s lemma 6'}, we get the following proposition:

\begin{proposition}\label{nonlinear 3'}
Assume \eqref{wtt-assumption} and \eqref{wtt-assumption'} hold. 
For any $\lll\in\vvv_n$ with $0\leq n\leq N_1-1$ and $\ell\in\{1,2,3\}$, we have
\begin{align}
    \bigg\|\varphi_k\p_{\xi_{\ell}}\bigg\{\widehat{\mathcal{N}_{\lll}^{kg}}-\ue^{-\ui t\lakg(\xi)}\bigg(\sum_{\lll_1,\lll_2}\ui\,\ccc_{\lll_2,\infty}\big(\ue^{\ui\ddi}\widehat{V^{kg}_{\infty}}\big)_{\lll_1}+\sum_{\lll_1,\lll_2}\ui\,\ccc_{\lll_2,\infty}\bbf_{\lll_1,\infty}+\bbb_{\lll,\infty}\bigg)\bigg\}\bigg\|_{L^2_\xi}
    \ls&\;\ep^2\br{t}^{-H_{kg}''(n)\d}2^{-N''_{kg}(n)k^+}.
\end{align}
\end{proposition}

\bigskip
\section{Energy Estimates for Wave Equation} 

\smallskip
\subsection{\texorpdfstring{$S_1$}{} Estimates} 

We start from the equation \eqref{wtt 05-diff} with vector fields:
\begin{align} \label{wtt 05-diff-L}
    \dt\widehat{\gwa_{\lll}}=\ue^{\ui t\lawa(\xi)}\widehat{\mathcal{N}^{wa}_{\lll}}-\hhh_{\lll,\infty}.
\end{align}
Multiplying $\abs{\xi}^{-1}\!\br{\xi}^{2N(n)}\overline{\widehat{\gwa_{\lll}}}$ on both sides and integrating over $\xi\in\r^3$, we have
\begin{align}
    \dt\bbee^{\lll}_{wa}=\bbrr_{wa}^{\lll},
\end{align}
where
\begin{align}
    \bbee^{\lll}_{wa}(t):=\frac{1}{2}\nm{\abs{\nabla}^{-\frac{1}{2}}\llgwa(t)}_{H^{N(n)}}^2,
\end{align}
and
\begin{align}
    \bbrr_{wa}^{\lll}:=\int_{\r^3}\abs{\xi}^{-1}\!\br{\xi}^{2H(n)}\Big(\ue^{\ui t\lawa(\xi)}\widehat{\mathcal{N}^{wa}_{\lll}}-\hhh_{\lll,\infty}\Big)\overline{\widehat{\gwa_{\lll}}}\,\ud\xi.
\end{align}
Hence, we have
\begin{align}
    \bbee^{\lll}_{wa}(t)=-\int_t^{\infty}\bbrr_{wa}^{\lll}(s)\,\ud s.
\end{align}
Similar to the nonlinear analysis, in order to estimate $\bbee^{\lll}_{wa}(t)$, it suffices to bound the time integral of every term in \eqref{decomposition 1} for $\ue^{\ui t\lawa(\xi)}\widehat{\mathcal{N}^{wa}_{\lll}}-\hhh_{\lll,\infty}$. 

\begin{remark}
It seems that we can directly apply nonlinear estimates in Lemma \ref{s lemma 1}, Lemma \ref{s lemma 2} and Lemma \ref{s lemma 3} to obtain the desired result. However, this is not always workable since nonlinear estimates do not have sufficient time decay and $k^+$ weight, so we have to redo the estimates term by term. In particular, the time integral will play a key role.
\end{remark}

\begin{lemma}\label{t lemma 1}
Assume \eqref{wtt-assumption} and \eqref{wtt-assumption'} hold. 
For any $\lll\in\vvv_n$ with $0\leq n\leq N_1$, we have
\begin{align}
\abs{\int_t^{\infty}\!\!\int_{\r^3}\abs{\xi}^{-1}\!\br{\xi}^{2H(n)} \bbi_{wa}^{\iota_1\iota_2}\big[G^{kg,\iota_1}_{\lll_1},G^{kg,\iota_2}_{\lll_2}\big] \!\cdot\overline{\widehat{\gwa_{\lll}}}\,\ud\xi\ud s} \ls \ep^3\br{t}^{-2H(n)\d}.
\end{align}
\end{lemma}

\begin{proof}
Without loss of generality, we assume $n_1\leq n_2$ and $n_1+n_2=n$. Also, we ignore $\iota_1$ and $\iota_2$ when there is no confusion. 
We intend to bound
\begin{align}
    \bbii_{wa}\big[G_{\lll_1}^{kg},G_{\lll_2}^{kg},G_{\lll}^{wa}\big]
    :=&\int_t^{\infty}\!\!\iint_{\r^3\times\r^3}\ue^{\ui s\Phi_{wa}}\abs{\xi}^{-1}\!\br{\xi}^{2N(n)} a(\xi,\eta)\widehat{G_{\lll_1}^{kg}}(s,\xi\!-\!\eta)\widehat{G_{\lll_2}^{kg}}(s,\eta)\overline{\widehat{G_{\lll}^{wa}}(s,\xi)}\,\ud\eta\ud\xi\ud s\,.
\end{align}
Here $a(\xi,\eta)$ will not play a role in the estimate since $\abs{a(\xi,\eta)}\ls 1$.
It is convenient to write
\begin{align}
    \bbii_{wa}^{m,k,k_1,k_2}\big[G_{\lll_1}^{kg},G_{\lll_2}^{kg},G_{\lll}^{wa}\big]:=\int_t^{\infty}\!\!\iint_{\r^3\times\r^3}&\tau_m(s)\,\ue^{\ui s\Phi_{wa}}\abs{\xi}^{-1}\!\br{\xi}^{2N(n)}\widehat{P_{k_1}G_{\lll_1}^{kg}}(s,\xi\!-\!\eta)\widehat{P_{k_2}G_{\lll_2}^{kg}}(s,\eta)\overline{\widehat{P_{k}G_{\lll}^{wa}}(s,\xi)}\,\ud\eta\ud\xi\ud s\,.
\end{align}
Here $\tau_m(s)$ is a time cutoff function defined in \eqref{t-decomp}. 

If $1\leq n_1\leq n_2\leq n-1$,
then we integrate by parts in time using \eqref{resonance-time}
to get
\begin{align}
    \bbii_{wa}^{m,k,k_1,k_2}\big[G_{\lll_1}^{kg},G_{\lll_2}^{kg},G_{\lll}^{wa}\big] \simeq&\int_t^{\infty}\!\tau_m'\bbkk_{wa}^{k,k_1,k_2}\big[G_{\lll_1}^{kg},G_{\lll_2}^{kg},G_{\lll}^{wa}\big](s)\,\ud s+\int_t^{\infty}\!\tau_m\bbkk_{wa}^{k,k_1,k_2}\big[\p_sG_{\lll_1}^{kg},G_{\lll_2}^{kg},G_{\lll}^{wa}\big](s)\,\ud s\\
    &+\int_t^{\infty}\!\tau_m\bbkk_{wa}^{k,k_1,k_2}\big[G_{\lll_1}^{kg},\p_sG_{\lll_2}^{kg},G_{\lll}^{wa}\big](s)\,\ud s+\int_t^{\infty}\!\tau_m\bbkk_{wa}^{k,k_1,k_2}\big[G_{\lll_1}^{kg},G_{\lll_2}^{kg},\p_sG_{\lll}^{wa}\big](s)\,\ud s,\no
\end{align}
where
\begin{align}
    \bbkk_{wa}^{k,k_1,k_2}\big[G_{\lll_1}^{kg},G_{\lll_2}^{kg},G_{\lll}^{wa}\big]&:=\iint_{\r^3\times\r^3}\frac{\ue^{\ui s\Phi_{wa}}}{\Phi_{wa}}\abs{\xi}^{-1}\!\br{\xi}^{2N(n)}\widehat{P_{k_1}G_{\lll_1}^{kg}}(\xi\!-\!\eta)\widehat{P_{k_2}G_{\lll_2}^{kg}}(\eta)\overline{\widehat{P_{k}G_{\lll}^{wa}}(\xi)}\,\ud\eta\ud\xi\,.
\end{align}
Note that for $|\xi|\approx 2^k$ and $s\approx 2^m$,
\begin{align}
    \abs{\frac{\ue^{\ui s\Phi_{wa}}}{\Phi_{wa}}}\ls \abs{\xi}^{-1}\approx 2^{-k},
\end{align}
and that
\begin{align}
    \abs{\tau_m}\ls 1,\quad \abs{\tau_m'}\ls 2^{-m}\approx s^{-1}.
\end{align}
From \eqref{ttt 02.} in Lemma\;\ref{prelim: nonlinear 3}, we know
\begin{align}
    \lnm{\bbkk_{wa}^{k,k_1,k_2}[\mathfrak{f},\mathfrak{g},\mathfrak{h}]}&\ls 2^{-k}2^{2N(n)k^+}\abs{\frac{\ue^{\ui s\Phi_{wa}}}{\Phi_{wa}}}\abs{\iint_{\r^3\times\r^3}\widehat{P_{k_1}\mathfrak{f}}(\xi\!-\!\eta)\widehat{P_{k_2}\mathfrak{g}}(\eta)\overline{\widehat{P_{k}\mathfrak{h}}(\xi)}\,\ud\eta\ud\xi}\\
    &\ls2^{-2k}2^{2N(n)k^+}2^{\frac{3}{2}\min(k,k_1,k_2)}\tnm{P_{k_1}\mathfrak{f}}\tnm{P_{k_2}\mathfrak{g}}\tnm{P_{k}\mathfrak{h}}\no\\
    &\ls2^{-\frac{1}{2}k}2^{2N(n)k^+}\tnm{P_{k_1}\mathfrak{f}}\tnm{P_{k_2}\mathfrak{g}}\tnm{P_{k}\mathfrak{h}}.\no
\end{align}
Also, based on Lemma \ref{wtt lemma 1}, we have
\begin{align}
    \big\|P_{k_1}G_{\lll_1}^{kg}\big\|_{L^2}&\ls \ep\br{t}^{-H(n_1)\d}2^{-N(n_1)k_1^+},\\
    \big\|P_{k_2}G_{\lll_2}^{kg}\big\|_{L^2}&\ls \ep\br{t}^{-H(n_2)\d}2^{-N(n_2)k_2^+},\\
    \big\|P_{k}G_{\lll}^{wa}\big\|_{L^2}&\ls \ep\br{t}^{-H(n)\d}2^{-N(n)k^+}2^{\frac{1}{2}k},
\end{align}
and based on Proposition \ref{nonlinear 1} and Proposition \ref{nonlinear 1'}, we have
\begin{align}
    \big\|P_{k_1}\p_tG_{\lll_1}^{kg}\big\|_{L^2}&\ls \ep^2\br{t}^{-H''_{kg}(n_1)\d}\br{t}^{-1}2^{-N''_{kg}(n_1)k_1^+},\\
    \big\|P_{k_2}\p_tG_{\lll_2}^{kg}\big\|_{L^2}&\ls \ep^2\br{t}^{-H''_{kg}(n_2)\d}\br{t}^{-1}2^{-N''_{kg}(n_2)k_2^+},\\
    \big\|P_{k}\p_tG_{\lll}^{wa}\big\|_{L^2}&\ls \ep^2\br{t}^{-H''_{wa}(n)\d}\br{t}^{-1}2^{-N''_{kg}(n)k^+}2^{\frac{1}{2}k^-}.
\end{align}
Note the following facts, which suffice to close the estimate:
\begin{itemize}
    \item 
    We have enough time decay
    \begin{align}
        H(n_1)+H(n_2)+H(n)&\geq 2H(n),\\
        H''_{kg}(n_1)+H(n_2)+H(n)&\geq 2H(n),\\
        H(n_1)+H''_{kg}(n_2)+H(n)&\geq 2H(n),\\
        H(n_1)+H(n_2)+H''_{wa}(n)&\geq 2H(n).
    \end{align}
    This relies on the fact that $0<n_1,n_2<n$, and that terms with less derivatives have faster time decay.
    \item
    We have enough $k^+$ weight
    \begin{align}
        \min\big\{N(n_1),N(n_2)\big\}+N(n)-5&\geq 2N(n),\\
        \min\big\{N''_{kg}(n_1),N(n_2)\big\}+N(n)-5&\geq 2N(n),\\
        \min\big\{N(n_1),N''_{kg}(n_2)\big\}+N(n)-5&\geq 2N(n),\\
        \min\big\{N(n_1),N(n_2)\big\}+N''_{wa}(n)-5&\geq 2N(n).
    \end{align}
    This also relies on the fact that $0<n_1,n_2<n$, and that terms with less derivatives have better $k^+$ weight.
    \item
    We have enough $k^-$ weight. $\btnm{P_{k}G_{\lll}^{wa}}$ and $\btnm{P_{k}\p_tG_{\lll}^{wa}}$ always provide an extra $2^{\frac{1}{2}k^-}$.
\end{itemize}

Now the only remaining case is when $n_1=0$ and $n_2=n$. (If we continue using integration by parts in time, time decay is enough, but $k^+$ weight is not.) Instead, we use \eqref{ttt 01.}.
\begin{itemize}
    \item 
    Case $k_1\simeq k\geq k_2$: For $n\leq 2$ (time decay is delicate), we use Lemma \ref{wtt lemma 7} to get
    \begin{align}
        &\abs{\iint_{\r^3\times\r^3}\ue^{\ui t\Phi_{wa}}\abs{\xi}^{-1}\!\br{\xi}^{2N(n)}\widehat{P_{k_1}G_{\lll_1}^{kg}}(\xi\!-\!\eta)\widehat{P_{k_2}G_{\lll_2}^{kg}}(\eta)\overline{\widehat{P_{k}G_{\lll}^{wa}}(\xi)}\,\ud\eta\ud\xi}\\
        \ls&\;2^{-k}2^{2N(n)k^+}\btnm{P_{k_1}G_{\lll_1}^{kg}}\blnm{\ue^{-\ui t\lakg}P_{k_2}G_{\lll_2}^{kg}}\btnm{P_{k}G_{\lll}^{wa}}\no\\
        \ls&\;2^{-k}2^{2N(n)k^+}\Big(\ep\br{t}^{-H(n_1)\d}2^{-N(n_1)k_1^+}\Big)\Big(\ep\br{t}^{-1}\br{t}^{-H(n_2+1)\d}2^{-(N(n_2+1)-\frac{5}{2})k_2^+}2^{\frac{1}{2}k_2^-}\Big)\Big(\ep\br{t}^{-H(n)\d}2^{-N(n)k^+}2^{\frac{1}{2}k}\Big)\no\\
        \ls&\;\ep^3\br{t}^{-1}\br{t}^{-(H(n_1)+H(n_2+1)+H(n))\d}2^{-N(n_1)k_1^+-(N(n_2+1)-\frac{5}{2})k_2^++\frac{1}{2}k_2^- +N(n)k^++\frac{1}{2}k^- -k}.\no
    \end{align}
    For $n\geq 2$ ($k^+$ weight is delicate), we use Lemma \ref{wtt lemma 7} to get
    \begin{align}
        &\abs{\iint_{\r^3\times\r^3}\ue^{\ui t\Phi_{wa}}\abs{\xi}^{-1}\!\br{\xi}^{2N(n)}\widehat{P_{k_1}G_{\lll_1}^{kg}}(\xi\!-\!\eta)\widehat{P_{k_2}G_{\lll_2}^{kg}}(\eta)\overline{\widehat{P_{k}G_{\lll}^{wa}}(\xi)}\,\ud\eta\ud\xi}\\
        \ls&\;2^{-k}2^{2N(n)k^+}\blnm{\ue^{-\ui t\lakg}P_{k_1}G_{\lll_1}^{kg}}\btnm{P_{k_2}G_{\lll_2}^{kg}}\btnm{P_{k}G_{\lll}^{wa}}\no\\
        \ls&\;2^{-k}2^{2N(n)k^+}\Big(\ep\br{t}^{-1}\!\br{t}^{-H(n_1+1)\d}2^{-(N(n_1+1)-\frac{5}{2})k_1^+}2^{\frac{1}{2}k_1^-}\Big)\Big(\ep\br{t}^{-H(n_2)\d}2^{-N(n_2)k_2^+}\Big)\Big(\ep\br{t}^{-H(n)\d}2^{-N(n)k^+}2^{\frac{1}{2}k}\Big)\no\\
        \ls&\;\ep^3\br{t}^{-1}\!\br{t}^{-(H(n_1+1)+H(n_2)+H(n))\d}2^{-(N(n_1+1)-\frac{5}{2})k_1^++\frac{1}{2}k_1^- -N(n_2)k_2^++N(n)k^++\frac{1}{2}k^-}.\no
    \end{align}
    \item
    Case $k_2\simeq k\geq k_1$: We choose $(\infty,2,2)$ in \eqref{ttt 01.} and use Lemma \ref{wtt lemma 7} as above to close the estimate.
    \item
    Case $k_1\simeq k_2\geq k$: We may use integration by parts in time as in the case $0<n_1,n_2<n$ (note that now $k^+$ weight is sufficient).
\end{itemize}

\end{proof}

\begin{lemma}\label{t lemma 2}
Assume \eqref{wtt-assumption} and \eqref{wtt-assumption'} hold. 
For any $\lll\in\vvv_n$ with $0\leq n\leq N_1$, we have
\begin{align}
\abs{\int_t^{\infty}\!\!\int_{\r^3}\abs{\xi}^{-1}\!\br{\xi}^{2H(n)} {\bbi_{wa}^{\iota_1\iota_2}}\Big[\widehat{G^{kg,\iota_1}_{\lll_1}},\big(\ue^{\iota_2\ui\ddi^{\iota_2}}\widehat{V^{kg,\iota_2}_{\infty}}\big)_{\lll_2}+\bbf^{\iota_2}_{\lll_2,\infty}\Big] \cdot\overline{\widehat{\gwa_{\lll}}}\,\ud\xi\ud s} \ls \ep^3\br{t}^{-2H(n)\d},\\
\abs{\int_t^{\infty}\!\!\int_{\r^3}\abs{\xi}^{-1}\!\br{\xi}^{2H(n)} {\bbi_{wa}^{\iota_1\iota_2}}\Big[\big(\ue^{\iota_1\ui\ddi^{\iota_1}}\widehat{V^{kg,\iota_1}_{\infty}}\big)_{\lll_1}+\bbf^{\iota_1}_{\lll_1,\infty}\,,\widehat{G^{kg,\iota_2}_{\lll_2}}\Big] \cdot\overline{\widehat{\gwa_{\lll}}}\,\ud\xi\ud s} \ls \ep^3\br{t}^{-2H(n)\d}.
\end{align}
\end{lemma}
\begin{proof}
By Cauchy-Schwartz inequality, we have 
\begin{align}
    &\abs{\int_t^{\infty}\!\!\int_{\r^3}\abs{\xi}^{-1}\!\br{\xi}^{2H(n)} {\bbi_{wa}^{\iota_1\iota_2}}\Big[\widehat{G^{kg,\iota_1}_{\lll_1}},\big(\ue^{\iota_2\ui\ddi^{\iota_2}}\widehat{V^{kg,\iota_2}_{\infty}}\big)_{\lll_2}+\bbf^{\iota_2}_{\lll_2,\infty}\Big] \cdot\overline{\widehat{\gwa_{\lll}}}\,\ud\xi\ud s}\\
    \ls&\int_t^{\infty}\nm{\abs{\xi}^{-\frac{1}{2}}\!\br{\xi}^{H(n)}{\bbi_{wa}^{\iota_1\iota_2}}\Big[\widehat{G^{kg,\iota_1}_{\lll_1}},\big(\ue^{\iota_2\ui\ddi^{\iota_2}}\widehat{V^{kg,\iota_2}_{\infty}}\big)_{\lll_2}+\bbf^{\iota_2}_{\lll_2,\infty}\Big]}_{L^2_\xi}\nm{\abs{\xi}^{-\frac{1}{2}}\!\br{\xi}^{H(n)}\widehat{\gwa_{\lll}}}_{L^2_\xi} \ud s\,,\no
\end{align}
and
\begin{align}
    &\abs{\int_t^{\infty}\!\!\int_{\r^3}\abs{\xi}^{-1}\!\br{\xi}^{2H(n)} {\bbi_{wa}^{\iota_1\iota_2}}\Big[\big(\ue^{\iota_1\ui\ddi^{\iota_1}}\widehat{V^{kg,\iota_1}_{\infty}}\big)_{\lll_1}+\bbf^{\iota_1}_{\lll_1,\infty}\,,\widehat{G^{kg,\iota_2}_{\lll_2}}\Big] \cdot\overline{\widehat{\gwa_{\lll}}}\,\ud\xi\ud s}\\
    \ls&\int_t^{\infty}\nm{\abs{\xi}^{-\frac{1}{2}}\!\br{\xi}^{H(n)}{\bbi_{wa}^{\iota_1\iota_2}}\Big[\big(\ue^{\iota_1\ui\ddi^{\iota_1}}\widehat{V^{kg,\iota_1}_{\infty}}\big)_{\lll_1}+\bbf^{\iota_1}_{\lll_1,\infty}\,,\widehat{G^{kg,\iota_2}_{\lll_2}}\Big]}_{L^2_\xi}\nm{\abs{\xi}^{-\frac{1}{2}}\!\br{\xi}^{H(n)}\widehat{\gwa_{\lll}}}_{L^2_\xi} \ud s\,.\no
\end{align}
Then Lemma \ref{s lemma 2} and Lemma \ref{wtt lemma 1} will close the proof.
\end{proof}

\begin{lemma}\label{t lemma 3}
Assume \eqref{wtt-assumption} and \eqref{wtt-assumption'} hold. 
For any $\lll\in\vvv_n$ with $0\leq n\leq N_1$, we have
\begin{align}
&\Bigg|\int_t^{\infty}\!\!\int_{\r^3}\abs{\xi}^{-1}\!\br{\xi}^{2H(n)}\overline{\widehat{\gwa_{\lll}}}\bigg\{\sum_{\lll_1,\lll_2}\sum_{\iota_1,\iota_2\in\{+,-\}}{\bbi_{wa}^{\iota_1\iota_2}}\Big[\big(\ue^{\iota_1\ui\ddi^{\iota_1}}\widehat{V^{kg,\iota_1}_{\infty}}\big)_{\lll_1}+\bbf^{\iota_1}_{\lll_1,\infty}\,,\big(\ue^{\iota_2\ui\ddi^{\iota_2}}\widehat{V^{kg,\iota_2}_{\infty}}\big)_{\lll_2}+\bbf^{\iota_2}_{\lll_2,\infty}\Big]-\hhh_{\lll,\infty}\bigg\}\,\,\ud\xi\ud s\Bigg|\no\\
\ls&\; \ep^3\br{t}^{-2H(n)\d}.
\end{align}
\end{lemma}
\begin{proof}
Similar to the proof of Lemma \ref{t lemma 2}, we may use Lemma \ref{s lemma 3} and Lemma \ref{wtt lemma 1} to get the bound.
\end{proof}

\begin{remark} 
Note that the nonlinear estimates in Lemma \ref{s lemma 2} and Lemma \ref{s lemma 3} are better than that of Lemma \ref{s lemma 1}. This makes the proofs of Lemma\;\ref{t lemma 2} and Lemma\;\ref{t lemma 3} much shorter than that of Lemma\;\ref{t lemma 1}.
\end{remark}

Combining Lemma \ref{t lemma 1}, Lemma \ref{t lemma 2} and Lemma \ref{t lemma 3}, we get the following proposition:

\begin{proposition}\label{energy 1}
Assume \eqref{wtt-assumption} and \eqref{wtt-assumption'} hold. 
For any $\lll\in\vvv_n$ with $0\leq n\leq N_1$, we have
\begin{align}
\big|\mathcal{E}^{\lll}_{wa}[G_{\lll}^{wa},G_{\lll}^{wa}]\big|\ls \ep^3\br{t}^{-2H(n)\d}.
\end{align}
\end{proposition}
This concludes the estimate of $\big\|\gwa\big\|_{S_1}$ in Proposition\;\ref{main proposition}.

\smallskip
\subsection{\texorpdfstring{$T_1$}{} Estimates} 

We now estimate $\xi$ derivatives of $\hlgwa$. 
Deriving from \eqref{wtt 05-diff-L}, we get the equation 
\begin{align}
    \big[\dt+\ui\lawa(\xi)\big]\Big(\ue^{-\ui t\lawa(\xi)}\widehat{\gwa_{\lll}}\Big)=\widehat{\mathcal{N}_{\lll}^{wa}}-\ue^{-\ui t\lawa(\xi)}\hhh_{\lll,\infty}.
\end{align}
Taking $\mu=wa$ in Lemma \ref{wtt prelim 16}, we know
\begin{align} \label{menghu}
    \widehat{\Gamma_{\ell}\big(\ue^{-\ui t\lawa}G_{\lll}^{wa}\big)}(t,\xi)=\ui\p_{\xi_{\ell}}\Big[\,\widehat{\mathcal{N}_{\lll}^{wa}}-\ue^{-\ui t\lawa(\xi)}\hhh_{\lll,\infty}\Big](t,\xi)+\ue^{-\ui t\lawa(\xi)}\p_{\xi_{\ell}}\Big[\lawa(\xi)\widehat{G_{\lll}^{wa}}(t,\xi)\Big].
\end{align}
Note that $\lawa(\xi)=\abs{\xi}$. 
Inserting 
\begin{align}
    \p_{\xi_{\ell}}\Big[\lawa(\xi)\widehat{G_{\lll}^{wa}}(t,\xi)\Big]=\lawa(\xi)\p_{\xi_{\ell}}\widehat{G_{\lll}^{wa}}(t,\xi)+\frac{\xi_{\ell}}{\abs{\xi}}\widehat{G_{\lll}^{wa}}(t,\xi)
\end{align}
into \eqref{menghu} yields
\begin{align}
    \ue^{-\ui t\lawa(\xi)}\lawa(\xi)\,\p_{\xi_{\ell}}\widehat{G_{\lll}^{wa}}(t,\xi)=
    \widehat{\Gamma_{\ell}\big(\ue^{-\ui t\lawa}G_{\lll}^{wa}\big)}(t,\xi)-\ui\p_{\xi_{\ell}}\Big[\,\widehat{\mathcal{N}_{\lll}^{wa}}-\ue^{-\ui t\lawa(\xi)}\hhh_{\lll,\infty}\Big](t,\xi)-\ue^{-\ui t\lawa(\xi)}\frac{\xi_{\ell}}{\abs{\xi}}\widehat{G_{\lll}^{wa}}(t,\xi).
\end{align}
Multiplying the above equality by $2^{-\frac{1}{2}k}\varphi_k$ and estimating in $L^2_\xi$, we obtain
\begin{align} \label{xiaomeng}
    2^{\frac{1}{2}k}\btnm{\varphi_k\p_{\xi_{\ell}}\widehat{G_{\lll}^{wa}}}\ls\;& 2^{-\frac{1}{2}k}\tnm{\varphi_k\widehat{\Gamma_{\ell}\big(\ue^{-\ui t\lawa}G_{\lll}^{wa}\big)}}+2^{-\frac{1}{2}k}\tnm{\varphi_k\p_{\xi_{\ell}}\Big[\,\widehat{\mathcal{N}_{\lll}^{wa}}-\ue^{-\ui t\lawa(\xi)}\hhh_{\lll,\infty}\Big]}+2^{-\frac{1}{2}k}\btnm{\widehat{P_k G_{\lll}^{wa}}}.
\end{align}
Based on the energy estimate in Proposition \ref{energy 1}, we have
\begin{align}
    2^{-\frac{1}{2}k}\tnm{\varphi_k\widehat{\Gamma_{\ell}\big(\ue^{-\ui t\lawa}G_{\lll}^{wa}\big)}}&\ls \e^{\frac{3}{2}}\br{t}^{-H(n+1)\d}2^{-N(n+1)k^+},\\
    2^{-\frac{1}{2}k}\btnm{\widehat{P_k G_{\lll}^{wa}}}&\ls\e^{\frac{3}{2}}\br{t}^{-H(n)\d}2^{-N(n)k^+}.
\end{align}
Based on the nonlinear estimate in Proposition \ref{nonlinear 3}, we have
\begin{align} \label{xiaohu}
    2^{-\frac{1}{2}k}\tnm{\varphi_k\p_{\xi_{\ell}}\Big[\,\widehat{\mathcal{N}_{\lll}^{wa}}-\ue^{-\ui t\lawa(\xi)}\hhh_{\lll,\infty}\Big]}\ls \e^2\br{t}^{-H''_{wa}(n)\d}2^{-N''_{wa}(n)k^+}.
\end{align}
Combining (\ref{xiaomeng})\,--\,(\ref{xiaohu}),
we get the following proposition:

\begin{proposition}\label{energy 3}
Assume \eqref{wtt-assumption} and \eqref{wtt-assumption'} hold. 
For any $\lll\in\vvv_n$ with $0\leq n\leq N_1-1$, we have
\begin{align}
2^{\frac{1}{2}k}\btnm{\varphi_k\p_{\xi_{\ell}}\widehat{G_{\lll}^{wa}}}\ls \e^{\frac{3}{2}}\br{t}^{-H(n+1)\d}2^{-N(n+1)k^+}.
\end{align}
\end{proposition}
This concludes the estimate of $\big\|\gwa\big\|_{T_1}$ in Proposition\;\ref{main proposition}.

\bigskip
\section{Energy Estimates for Klein-Gordon Equation} 

\smallskip
\subsection{\texorpdfstring{$S_2$}{} Estimates} 

We start from the equation \eqref{wtt 06-diff} with vector fields:
\begin{align} \label{wtt 06-diff-L}
    \dt\widehat{\gkg_{\lll}}=\ue^{\ui t\lakg(\xi)}\widehat{\mathcal{N}^{kg}_{\lll}}-\sum_{\lll_1,\lll_2}\ui\,\ccc_{\lll_2,\infty}\big(\ue^{\ui\ddi}\widehat{V^{kg}_{\infty}}\big)_{\lll_1}-\sum_{\lll_1,\lll_2}\ui\,\ccc_{\lll_2,\infty}\bbf_{\lll_1,\infty}-\bbb_{\lll,\infty}.
\end{align}
Multiplying $\br{\xi}^{2N(n)}\overline{\widehat{\gkg_{\lll}}}$ on both sides and integrating over $\xi\in\r^3$, we have
\begin{align}
    \dt\bbee^{\lll}_{kg}=\bbrr_{kg}^{\lll},
\end{align}
where
\begin{align}
    \bbee^{\lll}_{kg}(t):=\frac{1}{2}\big\|\llgkg(t)\big\|_{H^{N(n)}}^2,
\end{align}
and
\begin{align}
    \bbrr_{kg}^{\lll}:=\int_{\r^3}\br{\xi}^{2H(n)}\bigg\{\ue^{\ui t\lakg(\xi)}\widehat{\mathcal{N}^{kg}_{\lll}}-\sum_{\lll_1,\lll_2}\ui\,\ccc_{\lll_2,\infty}\big(\ue^{\ui\ddi}\widehat{V^{kg}_{\infty}}\big)_{\lll_1}-\sum_{\lll_1,\lll_2}\ui\,\ccc_{\lll_2,\infty}\bbf_{\lll_1,\infty}-\bbb_{\lll,\infty}\bigg\}\,\overline{\widehat{\gkg_{\lll}}}(t,\xi)\,\ud\xi\,.
\end{align}
Hence, we have
\begin{align}
    \bbee^{\lll}_{kg}(t)=-\int_t^{\infty}\bbrr_{kg}^{\lll}(s)\,\ud s.
\end{align}
Similar to the nonlinear analysis, in order to estimate $\bbee^{\lll}_{kg}(t)$, it suffices to bound the time integral of every term in \eqref{decomposition 2} for $\ue^{\ui t\lakg(\xi)}\widehat{\mathcal{N}^{kg}_{\lll}}-\Big\{\ui\,\cci\big(\ue^{\ui\ddi}\widehat{V^{kg}_{\infty}}+\bbf_{\infty}\big)\Big\}_{\lll}-\bbb_{\lll,\infty}$. 

\begin{lemma}\label{t lemma 1'}
Assume \eqref{wtt-assumption} and \eqref{wtt-assumption'} hold. 
For any $\lll\in\vvv_n$ with $0\leq n\leq N_1$, we have
\begin{align}
\abs{\int_t^{\infty}\!\!\int_{\r^3}\br{\xi}^{2H(n)} \bbi_{kg}^{\iota_1\iota_2}\big[G^{kg,\iota_1}_{\lll_1},G^{wa,\iota_2}_{\lll_2}\big]\cdot\overline{\widehat{\gkg_{\lll}}}(s,\xi)\,\ud\xi\ud s} \ls \ep^3\br{t}^{-2H(n)\d}.
\end{align}
\end{lemma}

\begin{proof}
We ignore $\iota_1$ and $\iota_2$ when there is no confusion. 
We focus on the bound of
\begin{align}
    \bbii_{kg}\big[G_{\lll_1}^{kg},G_{\lll_2}^{wa},G_{\lll}^{kg}\big]
    :=&\int_t^{\infty}\!\!\iint_{\r^3\times\r^3}\ue^{\ui s\Phi_{kg}}\br{\xi}^{2N(n)} b(\xi,\eta)\widehat{G_{\lll_1}^{kg}}(s,\xi\!-\!\eta)\widehat{G_{\lll_2}^{wa}}(s,\eta)\overline{\widehat{G_{\lll}^{kg}}(s,\xi)}\,\ud\eta\ud\xi\ud s\,.
\end{align}
Denote by
\begin{align}
    \bbii_{kg}^{m,k,k_1,k_2}\big[G_{\lll_1}^{kg},G_{\lll_2}^{wa},G_{\lll}^{kg}\big]:=\int_t^{\infty}\!\!\iint_{\r^3\times\r^3}&\tau_m(s)\,\ue^{\ui s\Phi_{kg}}\br{\xi}^{2N(n)} b(\xi,\eta)\widehat{P_{k_1}G_{\lll_1}^{kg}}(s,\xi\!-\!\eta)\widehat{P_{k_2}G_{\lll_2}^{wa}}(s,\eta)\overline{\widehat{P_{k}G_{\lll}^{kg}}(s,\xi)}\,\ud\eta\ud\xi\ud s\,.
\end{align}

For the case $1\leq n_1\leq n_2\leq n-1$, we integrate by parts in time using \eqref{resonance-time} to get
\begin{align}
    \bbii_{kg}^{m,k,k_1,k_2}\big[G_{\lll_1}^{kg},G_{\lll_2}^{wa},G_{\lll}^{kg}\big]\simeq&\int_t^{\infty}\tau_m'\bbkk_{kg}^{k,k_1,k_2}\big[G_{\lll_1}^{kg},G_{\lll_2}^{wa},G_{\lll}^{kg}\big](s)\,\ud s+\int_t^{\infty}\tau_m\bbkk_{kg}^{k,k_1,k_2}\big[\p_sG_{\lll_1}^{kg},G_{\lll_2}^{wa},G_{\lll}^{kg}\big](s)\,\ud s\\
    &+\int_t^{\infty}\tau_m\bbkk_{kg}^{k,k_1,k_2}\big[G_{\lll_1}^{kg},\p_sG_{\lll_2}^{wa},G_{\lll}^{kg}\big](s)\,\ud s+\int_t^{\infty}\tau_m\bbkk_{kg}^{k,k_1,k_2}\big[G_{\lll_1}^{kg},G_{\lll_2}^{wa},\p_sG_{\lll}^{kg}\big](s)\,\ud s,\no
\end{align}
where
\begin{align}
    \bbkk_{kg}^{k,k_1,k_2}\big[G_{\lll_1}^{kg},G_{\lll_2}^{wa},G_{\lll}^{kg}\big]&:=\iint_{\r^3\times\r^3}\frac{\ue^{\ui s\Phi_{kg}}}{\Phi_{kg}}\br{\xi}^{2N(n)} b(\xi,\eta)\widehat{P_{k_1}G_{\lll_1}^{kg}}(\xi\!-\!\eta)\widehat{P_{k_2}G_{\lll_2}^{wa}}(\eta)\overline{\widehat{P_{k}G_{\lll}^{kg}}(\xi)}\,\ud\eta\ud\xi\,.
\end{align}
Note that for $|\xi|\approx 2^k$ and $s\approx 2^m$,
\begin{align}
    \abs{\frac{\ue^{\ui s\Phi_{kg}}}{\Phi_{kg}}}&\ls \abs{\eta}^{-1}\approx 2^{-k_2},\\
    \big|b(\xi,\eta)\big|&\ls\frac{|\xi\!-\!\eta|}{|\eta|}\approx 2^{k_1}2^{-k_2},
\end{align}
and that
\begin{align}
    \abs{\tau_m}\ls 1,\qquad \abs{\tau_m'}\ls 2^{-m}\approx s^{-1}.
\end{align}
From \eqref{ttt 02.} in Lemma\;\ref{prelim: nonlinear 3}, we know
\begin{align}
    \lnm{\bbkk_{kg}^{k,k_1,k_2}[\mathfrak{f},\mathfrak{g},\mathfrak{h}]}&\ls 2^{2N(n)k^+}2^{k_1}2^{-k_2}\abs{\frac{\ue^{\ui s\Phi_{kg}}}{\Phi_{kg}}}\abs{\iint_{\r^3\times\r^3}\widehat{P_{k_1}\mathfrak{f}}(\xi\!-\!\eta)\widehat{P_{k_2}\mathfrak{g}}(\eta)\overline{\widehat{P_{k}\mathfrak{h}}(\xi)}\,\ud\eta\ud\xi}\\
    &\ls2^{k_1}2^{-2k_2}2^{2N(n)k^+}2^{\frac{3}{2}\min(k,k_1,k_2)}\tnm{P_{k_1}\mathfrak{f}}\tnm{P_{k_2}\mathfrak{g}}\tnm{P_{k}\mathfrak{h}}\no\\
    &\ls2^{k_1}2^{-\frac{1}{2}k_2}2^{2N(n)k^+}\tnm{P_{k_1}\mathfrak{f}}\tnm{P_{k_2}\mathfrak{g}}\tnm{P_{k}\mathfrak{h}}.\no
\end{align}
Also, based on Lemma \ref{wtt lemma 1}, we have
\begin{align}
    \big\|P_{k_1}G_{\lll_1}^{kg}\big\|_{L^2}&\ls \ep\br{t}^{-H(n_1)\d}2^{-N(n_1)k_1^+},\\
    \tnm{P_{k_2}G_{\lll_2}^{wa}}&\ls \ep\br{t}^{-H(n_2)\d}2^{-N(n_2)k_2^+}2^{\frac{1}{2}k_2},\\
    \big\|P_{k}G_{\lll}^{kg}\big\|_{L^2}&\ls \ep\br{t}^{-H(n)\d}2^{-N(n)k^+},
\end{align}
and based on Proposition \ref{nonlinear 1} and Proposition \ref{nonlinear 1'}, we have
\begin{align}
    \big\|P_{k_1}\p_tG_{\lll_1}^{kg}\big\|_{L^2}&\ls \ep^2\br{t}^{-1-H''_{kg}(n_1)\d} 2^{-N''_{kg}(n_1)k_1^+},\\
    \tnm{P_{k_2}\p_tG_{\lll_2}^{wa}}&\ls \ep^2\br{t}^{-1-H''_{wa}(n_2)\d} 2^{-N''_{kg}(n_1)k^+}2^{\frac{1}{2}k_2^-},\\
    \big\|P_{k}\p_tG_{\lll}^{kg}\big\|_{L^2}&\ls \ep^2\br{t}^{-1-H''_{kg}(n)\d} 2^{-N''_{kg}(n_1)k^+}.
\end{align}
Then following a similar argument as in the proof of Lemma \ref{t lemma 1}, we may close the estimate.

When $n_1=0$ and $n_2=n$, we use \eqref{ttt 01..} and discuss the following cases:
\begin{itemize}
    \item 
    Case $k_1\simeq k_2\geq k$ or $k_2\simeq k\geq k_1$: We use Lemma \ref{wtt lemma 7} to get
    \begin{align}
        &\abs{\iint_{\r^3\times\r^3}\ue^{\ui t\Phi_{kg}}\br{\xi}^{2N(n)} b(\xi,\eta)\widehat{P_{k_1}G_{\lll_1}^{kg}}(\xi\!-\!\eta)\widehat{P_{k_2}G_{\lll_2}^{wa}}(\eta)\overline{\widehat{P_{k}G_{\lll}^{kg}}(\xi)}\,\ud\eta\ud\xi}\\
        \ls&\;2^{k_1}2^{-k_2}2^{2N(n)k^+}\blnm{P_{k_1}\ue^{-\ui t\lakg}G_{\lll_1}^{kg}}\btnm{P_{k_2}G_{\lll_2}^{wa}}\big\|P_{k}G_{\lll}^{kg}\big\|_{L^2}\no\\
        \ls&\;2^{k_1}2^{-k_2}2^{2N(n)k^+}\Big(\ep\br{t}^{-1-H(n_1+1)\d}2^{-(N(n_1+1)-\frac{5}{2})k_1^+}2^{\frac{1}{2}k_1^-}\Big)\Big(\ep\br{t}^{-H(n_2)\d}2^{-N(n_2)k_2^+}2^{\frac{1}{2}k_2^-}\Big)\Big(\ep\br{t}^{-H(n)\d}2^{-N(n)k^+}\Big)\no\\
        \ls&\;\ep^3\br{t}^{-1-(H(n_1+1)+H(n_2)+H(n))\d}2^{-(N(n_1+1)-\frac{5}{2})k_1^+-N(n_2)k_2^++N(n)k^+}2^{\frac{1}{2}k_1^-+\frac{1}{2}k_2^-}2^{k_1}2^{-k_2}.\no
    \end{align}
    \item 
    Case $k_1\simeq k\geq k_2$: If $n\leq 2$ (time decay is delicate), we use Lemma \ref{wtt lemma 7} to get
    \begin{align}
        &\abs{\iint_{\r^3\times\r^3}\ue^{\ui t\Phi_{kg}}\br{\xi}^{2N(n)} b(\xi,\eta)\widehat{P_{k_1}G_{\lll_1}^{kg}}(\xi\!-\!\eta)\widehat{P_{k_2}G_{\lll_2}^{wa}}(\eta)\overline{\widehat{P_{k}G_{\lll}^{kg}}(\xi)}\,\ud\eta\ud\xi}\\
        \ls&\;2^{k_1}2^{-k_2}2^{2N(n)k^+}\btnm{P_{k_1}G_{\lll_1}^{kg}}\big\|P_{k_2}\ue^{-\ui t\lawa}G_{\lll_2}^{wa}\big\|_{L^\infty}\btnm{P_{k}G_{\lll}^{kg}}\no\\
        \ls&\;2^{k_1}2^{-k_2}2^{2N(n)k^+}\Big(\ep\br{t}^{-H(n_1)\d}2^{-N(n_1)k_1^+}\Big)\Big(\ep\br{t}^{-1}\ln\!\br{t}\br{t}^{-H(n_2+1)\d}2^{-(N(n_2+1)-1)k_2^+}2^{k_2^-}\Big)\Big(\ep\br{t}^{-H(n)\d}2^{-N(n)k^+}\Big)\no\\
        \ls&\;\ep^3\ln\!\br{t}\br{t}^{-1-(H(n_1)+H(n_2+1)+H(n))\d}2^{-N(n_1)k_1^+-(N(n_2+1)-1)k_2^++N(n)k^+}2^{k_2^-}2^{k_1}2^{-k_2}.\no
    \end{align}
    If $n\geq2$ ($N(n)$ is relatively small, but we must find a term providing $2^{k_2}$; however, $\gwa$ in $L^2$ only provides $2^{\frac{1}{2}k_2}$), then we use integration by parts in time to close the estimate.
\end{itemize}

When $n_1=n$ and $n_2=0$, we use \eqref{ttt 01..} and discuss the following cases:
\begin{itemize}
    \item 
    Case $k_1\simeq k\geq k_2$ or $k_1\simeq k_2\geq k$: We choose $(2,\infty,2)$ in \eqref{ttt 01..} and use Lemma \ref{wtt lemma 7} as above to close the estimate.
    \item
    Case $k_2\simeq k\geq k_1$: If $n\leq 2$, we choose $(\infty,2,2)$ in \eqref{ttt 01..} and use Lemma \ref{wtt lemma 7}. If $n\geq 2$, we may use integration by parts in time to close the proof.
\end{itemize}
\end{proof}

\begin{lemma}\label{t lemma 2'}
Assume \eqref{wtt-assumption} and \eqref{wtt-assumption'} hold. 
For any $\lll\in\vvv_n$ with $0\leq n\leq N_1$, we have
\begin{align}
\abs{\int_t^{\infty}\!\!\int_{\r^3}\br{\xi}^{2H(n)} {\bbi_{kg}^{\iota_1\iota_2}}\Big[\widehat{G^{kg,\iota_1}_{\lll_1}},\widehat{V^{wa,\iota_2}_{\lll_2,\infty}}+\hhf^{\iota_2}_{\lll_2,\infty}\Big] \!\cdot\overline{\widehat{\gkg_{\lll}}}\,\ud\xi\ud s}&\ls \ep^3\br{t}^{-2H(n)\d},\\
\abs{\int_t^{\infty}\!\!\int_{\r^3}\br{\xi}^{2H(n)} {\bbi_{kg}^{\iota_1\iota_2}}\Big[\big(\ue^{\iota_1\ui\ddi^{\iota_1}}\widehat{V^{kg,\iota_1}_{\infty}}\big)_{\lll_1}+\bbf^{\iota_1}_{\lll_1,\infty}\,,\widehat{G^{wa,\iota_2}_{\lll_2}}\Big] \!\cdot\overline{\widehat{\gkg_{\lll}}}\,\ud\xi\ud s}&\ls \ep^3\br{t}^{-2H(n)\d}.
\end{align}
\end{lemma}

\begin{proof}
By Cauchy-Schwartz inequality, we have 
\begin{align}
    &\abs{\int_t^{\infty}\!\!\int_{\r^3}\br{\xi}^{2H(n)} {\bbi_{kg}^{\iota_1\iota_2}}\Big[\widehat{G^{kg,\iota_1}_{\lll_1}},\widehat{V^{wa,\iota_2}_{\lll_2,\infty}}+\hhf^{\iota_2}_{\lll_2,\infty}\Big] \!\cdot\overline{\widehat{\gkg_{\lll}}}\,\ud\xi\ud s}\\
    \ls&\int_t^{\infty}\Big\|\br{\xi}^{H(n)}{\bbi_{kg}^{\iota_1\iota_2}}\Big[\widehat{G^{kg,\iota_1}_{\lll_1}},\widehat{V^{wa,\iota_2}_{\lll_2,\infty}}+\hhf^{\iota_2}_{\lll_2,\infty}\Big]\Big\|_{L^2_\xi}\Bnm{\br{\xi}^{H(n)}\widehat{\gkg_{\lll}}}_{L^2_\xi}\!\ud s\,,\no
\end{align}
and
\begin{align}
    &\abs{\int_t^{\infty}\!\!\int_{\r^3}\br{\xi}^{2H(n)} {\bbi_{kg}^{\iota_1\iota_2}}\Big[\big(\ue^{\iota_1\ui\ddi^{\iota_1}}\widehat{V^{kg,\iota_1}_{\infty}}\big)_{\lll_1}+\bbf^{\iota_1}_{\lll_1,\infty}\,,\widehat{G^{wa,\iota_2}_{\lll_2}}\Big] \!\cdot\overline{\widehat{\gkg_{\lll}}}\,\ud\xi\ud s}\\
    \ls&\int_t^{\infty}\Bnm{\br{\xi}^{H(n)}{\bbi_{kg}^{\iota_1\iota_2}}\Big[\big(\ue^{\iota_1\ui\ddi^{\iota_1}}\widehat{V^{kg,\iota_1}_{\infty}}\big)_{\lll_1}+\bbf^{\iota_1}_{\lll_1,\infty}\,,\widehat{G^{wa,\iota_2}_{\lll_2}}\Big]}_{L^2_\xi}\Bnm{\br{\xi}^{H(n)}\widehat{\gkg_{\lll}}}_{L^2_\xi}\!\ud s\,.\no
\end{align}
Then Lemma \ref{s lemma 2'} and Lemma \ref{wtt lemma 1} will close the proof.
\end{proof}

\begin{lemma}\label{t lemma 3'}
Assume \eqref{wtt-assumption} and \eqref{wtt-assumption'} hold. 
For any $\lll\in\vvv_n$ with $0\leq n\leq N_1$, we have
\begin{align}
&\Bigg|\int_t^{\infty}\!\!\int_{\r^3}\br{\xi}^{2H(n)}\overline{\widehat{\gkg_{\lll}}}\bigg\{\sum_{\lll_1,\lll_2}\sum_{\iota_1,\iota_2=\{+,-\}}{\bbi_{kg}^{\iota_1\iota_2}}\Big[\big(\ue^{\iota_1\ui\ddi^{\iota_1}}\widehat{V^{kg,\iota_1}_{\infty}}\big)_{\lll_1}+\bbf^{\iota_1}_{\lll_1,\infty}\,,\widehat{V^{wa,\iota_2}_{\lll_2,\infty}}+\hhf^{\iota_2}_{\lll_2,\infty}\Big]\\
&-\sum_{\lll_1,\lll_2}\ui\,\ccc_{\lll_2,\infty}\big(\ue^{\ui\ddi}\widehat{V^{kg}_{\infty}}\big)_{\lll_1}-\sum_{\lll_1,\lll_2}\ui\,\ccc_{\lll_2,\infty}\bbf_{\lll_1,\infty}-\bbb_{\lll,\infty}\bigg\}\,\ud\xi\ud s\Bigg|
\ls\;\ep^3\br{t}^{-2H(n)\d},\no
\end{align}
\end{lemma}
\begin{proof}
Similar to the proof of Lemma \ref{t lemma 2'}, we may use Lemma \ref{s lemma 3'} and Lemma \ref{wtt lemma 1} to get the bound.
\end{proof}

Combining Lemma \ref{t lemma 1'}, Lemma \ref{t lemma 2'} and Lemma \ref{t lemma 3'}, we get the following proposition:

\begin{proposition}\label{energy 1'}
Assume \eqref{wtt-assumption} and \eqref{wtt-assumption'} hold. 
For any $\lll\in\vvv_n$ with $0\leq n\leq N_1$, we have
\begin{align}
\abs{\mathcal{E}^{\lll}_{kg}[G_{\lll}^{kg},G_{\lll}^{kg}]}\ls \ep^3\br{t}^{-2H(n)\d}.
\end{align}
\end{proposition}
This concludes the estimate of $\big\|\gkg\big\|_{S_2}$ in Proposition\;\ref{main proposition}.

\smallskip
\subsection{\texorpdfstring{$T_2$}{} Estimates} 

We now estimate $\xi$ derivatives of $\hlgkg$. 
Deriving from \eqref{wtt 06-diff-L}, we get the equation 
\begin{align}
    \big[\dt+\ui\lakg(\xi)\big]\Big(\ue^{-\ui t\lakg(\xi)}\widehat{\gkg_{\lll}}\Big) 
    = \widehat{\mathcal{N}^{kg}_{\lll}}
    -\ue^{-\ui t\lakg(\xi)}  \bigg\{\sum_{\lll_1,\lll_2}\ui\,\ccc_{\lll_2,\infty}\big(\ue^{\ui\ddi}\widehat{V^{kg}_{\infty}}\big)_{\lll_1}+\sum_{\lll_1,\lll_2}\ui\,\ccc_{\lll_2,\infty}\bbf_{\lll_1,\infty}+\bbb_{\lll,\infty}\bigg\}.
\end{align}
Taking $\mu=kg$ in Lemma \ref{wtt prelim 16}, we know
\begin{align} \label{feimeng}
    \widehat{\Gamma_{\ell}\big(\ue^{-\ui t\lakg}G_{\lll}^{kg}\big)}(t,\xi)=\;&\ui\p_{\xi_{\ell}}\bigg\{\widehat{\mathcal{N}_{\lll}^{kg}}-\ue^{-\ui t\lakg(\xi)}\bigg[\sum_{\lll_1,\lll_2}\ui\,\ccc_{\lll_2,\infty}\big(\ue^{\ui\ddi}\widehat{V^{kg}_{\infty}}\big)_{\lll_1}+\sum_{\lll_1,\lll_2}\ui\,\ccc_{\lll_2,\infty}\bbf_{\lll_1,\infty}+\bbb_{\lll,\infty}\bigg]\bigg\}(t,\xi)\no\\
    &+\ue^{-\ui t\lakg(\xi)}\p_{\xi_{\ell}}\Big[\lakg(\xi)\widehat{G_{\lll}^{kg}}(t,\xi)\Big].
\end{align}
Note that $\lakg(\xi)=\br{\xi}$. 
Inserting
\begin{align}
    \p_{\xi_{\ell}}\Big[\lakg(\xi)\widehat{G_{\lll}^{kg}}(t,\xi)\Big]=\lakg(\xi)\p_{\xi_{\ell}}\widehat{G_{\lll}^{kg}}(t,\xi)+\frac{\xi_{\ell}}{\br{\xi}}\widehat{G_{\lll}^{kg}}(t,\xi)
\end{align}
into \eqref{feimeng} yields
\begin{align}
    \ue^{-\ui t\lakg(\xi)}\lakg(\xi)\,\p_{\xi_{\ell}}\widehat{G_{\lll}^{kg}}(t,\xi)=&\;
    \widehat{\Gamma_{\ell}\big(\ue^{-\ui t\lakg}G_{\lll}^{kg}\big)}(t,\xi)
    -\ue^{-\ui t\lakg(\xi)}\frac{\xi_{\ell}}{\br{\xi}}\widehat{G_{\lll}^{kg}}(t,\xi)\\
    &-\ui\p_{\xi_{\ell}}\bigg\{\widehat{\mathcal{N}_{\lll}^{kg}}-\ue^{-\ui t\lakg(\xi)}\bigg[\sum_{\lll_1,\lll_2}\ui\,\ccc_{\lll_2,\infty}\big(\ue^{\ui\ddi}\widehat{V^{kg}_{\infty}}\big)_{\lll_1}\!+\!\sum_{\lll_1,\lll_2}\ui\,\ccc_{\lll_2,\infty}\bbf_{\lll_1,\infty}+\bbb_{\lll,\infty}\bigg]\bigg\}(t,\xi)\,.\no
\end{align}
Multiplying the above equality by $\varphi_k$ and estimating in $L^2$, we obtain
\begin{align} \label{dashen}
    2^{k^+}\Big\|\varphi_k\p_{\xi_{\ell}}\widehat{G_{\lll}^{kg}}\Big\|_{L^2}\ls\,& \Big\|\widehat{\varphi_k\Gamma_{\ell}\big(\ue^{-\ui t\lakg}G_{\lll}^{kg}\big)}\Big\|_{L^2}+\big\|\widehat{P_k G_{\lll}^{kg}}\big\|_{L^2}\\
    &+\bigg\|\varphi_k\p_{\xi_{\ell}}\bigg\{\widehat{\mathcal{N}_{\lll}^{kg}}-\ue^{-\ui t\lakg(\xi)}\bigg[\sum_{\lll_1,\lll_2}\ui\,\ccc_{\lll_2,\infty}\big(\ue^{\ui\ddi}\widehat{V^{kg}_{\infty}}\big)_{\lll_1}+\sum_{\lll_1,\lll_2}\ui\,\ccc_{\lll_2,\infty}\bbf_{\lll_1,\infty}+\bbb_{\lll,\infty}\bigg]\bigg\}\bigg\|_{L^2}.\no
\end{align}
Based on the energy estimate in Proposition \ref{energy 1'}, we have
\begin{align}
    \Big\|\widehat{\varphi_k\Gamma_{\ell}\big(\ue^{-\ui t\lakg}G_{\lll}^{kg}\big)}\Big\|_{L^2}&\ls \e^{\frac{3}{2}}\br{t}^{-H(n+1)\d}2^{-N(n+1)k^+},\\
    \big\|\widehat{P_k G_{\lll}^{kg}}\big\|_{L^2}&\ls\e^{\frac{3}{2}}\br{t}^{-H(n)\d}2^{-N(n)k^+}.
\end{align}
Based on the nonlinear estimate in Proposition \ref{nonlinear 3'}, we have
\begin{align} \label{dapanghu}
    &\bigg\|\varphi_k\p_{\xi_{\ell}}\bigg\{\widehat{\mathcal{N}_{\lll}^{kg}}-\ue^{-\ui t\lakg(\xi)}\bigg[\sum_{\lll_1,\lll_2}\ui\,\ccc_{\lll_2,\infty}\big(\ue^{\ui\ddi}\widehat{V^{kg}_{\infty}}\big)_{\lll_1}+\sum_{\lll_1,\lll_2}\ui\,\ccc_{\lll_2,\infty}\bbf_{\lll_1,\infty}+\bbb_{\lll,\infty}\bigg]\bigg\}\bigg\|_{L^2} 
    \ls\;\e^2\br{t}^{-H''_{kg}(n)\d}2^{-N''_{kg}(n)k^+}.
\end{align}
Combining (\ref{dashen})\,--\,(\ref{dapanghu}),
we get the following proposition:

\begin{proposition}\label{energy 3'}
Assume \eqref{wtt-assumption} and \eqref{wtt-assumption'} hold. 
For any $\lll\in\vvv_n$ with $0\leq n\leq N_1-1$, we have
\begin{align}
2^{k^+}\Big\|\varphi_k\p_{\xi_{\ell}}\widehat{G_{\lll}^{kg}}\Big\|_{L^2}\ls \e^{\frac{3}{2}}\br{t}^{-H(n+1)\d}2^{-N(n+1)k^+}.
\end{align}
\end{proposition}
This concludes the estimate of $\big\|\gkg\big\|_{T_2}$ in Proposition\;\ref{main proposition}.

\bigskip
\appendix
\section{Technical Lemmas}

\makeatletter
\renewcommand \theequation {%
A.%
\ifnum\c@subsection>\z@\@arabic\c@subsection.%
\fi\@arabic\c@equation} \@addtoreset{equation}{section}
\@addtoreset{equation}{subsection} \makeatother

\smallskip
\subsection{Basic Analytic Tools} 

\begin{lemma}[Stationary Phase] \label{wtt prelim 3}
Let $\phi(x)\in C^2(\r^n)$ and $a(x)\in C(\r^n)$. $x_0$ is the only  non-degenerate critical point of $\phi$, that is $\nabla_x\phi(x_0)=0$ and $\det\big(\nabla_x^2\phi(x_0)\big)\neq0$. Then for $t\in\r^+$, we have
\begin{align}
    \int_{\r^n}\ue^{\ui t\phi(x)}a(x)\ud x=\Big(\frac{2\pi}{t}\Big)^{\frac{n}{2}}\frac{\ue^{\ui t\phi(x_0)}}{\abs{\det\big(\nabla_x^2\phi(x_0)\big)}^{\frac{1}{2}}}\ue^{\frac{\ui\pi}{4}\text{sgn}\big(\nabla_x^2\phi(x_0)\big)}\Big(a(x_0)+O(t^{-1})\Big)\ \ \text{as}\ \ t\rt\infty.
\end{align}
Here $\rm{sgn}(A)$ represents the number of positive eigenvalues minus the number of negative eigenvalues.
\end{lemma}

\begin{proof}
This is basically an integration by parts argument.
See Evans \cite[Section 4.5.3]{Evans2010} and Stein-Shakarchi \cite[Section 8.2]{Stein.Shakarchi2011}.
\end{proof}

\begin{lemma}[Non-Stationary Phase] \label{wtt prelim 4}
\begin{enumerate}
    \item[(i)]
    For $\psi(\eta)\in L^2(\r^n)$ with $\psi\neq0$, we have
    \begin{align}\label{resonance-time}
    \ue^{\ui s\psi}=\frac{1}{\ui\psi}\p_s(\ue^{\ui s\psi}).
    \end{align}
    \item[(ii)]
    For $\psi(\eta)\in H^1(\r^n)$ with $\nabla_{\!\eta}\psi\neq0$, we have
    \begin{align}\label{resonance-space}
        \ue^{\ui s\psi}=\frac{1}{\ui s\abs{\nabla_{\!\eta}\psi}^2}\nabla_{\!\eta}\psi\cdot\nabla_{\!\eta}(\ue^{\ui s\psi}),
    \end{align}
\end{enumerate}
\end{lemma}
\begin{proof}
These rely on direct computation, so we omit the proof.
\end{proof}

\begin{remark}
The identity \eqref{resonance-time} is mainly used to perform integration by parts in the temporal variable, and \eqref{resonance-space} is mainly used to perform integration by parts in the frequency variables.
\end{remark}

\begin{lemma} \label{wtt prelim 1}
For $f,g\in L^2$ satisfying 
\begin{align}
    f(x)=\sum_{k_1=-\infty}^{\infty}P_{k_1}f(x),\qquad g(x)=\sum_{k_2=-\infty}^{\infty}P_{k_2}g(x),
\end{align}
we have that for $k\in\zz$,
\begin{align}
    P_k(fg)\simeq&\sum_{k_1\simeq k_2\geq k}P_k\big[(P_{k_1}f)(P_{k_2}g)\big]+\sum_{k_1\simeq k\geq k_2+2}P_k\big[(P_{k_1}f)(P_{k_2}g)\big]+\sum_{k_2\simeq k\geq k_1+2}P_k\big[(P_{k_1}f)(P_{k_2}g)\big].
\end{align}
The summation is over all possible combination of $k_1,k_2\in\zz$.
\end{lemma}

\begin{proof}
Let $h(x)=f(x)g(x)$. Hence,
\begin{align}
    \widehat h(\xi)=\big(\widehat f\ast \widehat g\big)(\xi)=\int_{\r^3}\widehat f(\xi-\eta)\widehat g(\eta)\,\ud\eta.
\end{align}
Hence, we actually want to find the region so that
\begin{align}
    \abs{\xi}\in\big[2^{k-1},2^{k+1}\big],\qquad \abs{\xi-\eta}\in \big[2^{k_1-1},2^{k_1+1}\big],\qquad \abs{\eta}\in \big[2^{k_2-1},2^{k_2+1}\big].
\end{align}
Then if $k_1\simeq k_2$, we know 
\begin{align}
    \abs{\xi-\eta}+\abs{\eta}\sim \big[2^{\max(k_1,k_2)+1},0\big],
\end{align}
which means $k\simeq k_1\simeq k_2$. Otherwise, if $\abs{k_1-k_2}\geq2$, we have 
\begin{align}
    \abs{\xi-\eta}+\abs{\eta}\sim \big[2^{\max(k_1,k_2)+1},2^{\max(k_1,k_2)-1}\big],
\end{align}
which means $k\simeq \max(k_1,k_2)$.
\end{proof}

\begin{remark}
This lemma provides the effective region of convolution in frequency space.
For convenience, we denote this set for fixed $k\in\mathbb{Z}$\, by
\begin{align*}
    \chi_k := \Big\{(k_1,k_2)\in\mathbb{Z}^2: k_1\simeq k_2\geq k \,\textrm{ or }\, k_1\simeq k\geq k_2+2 \,\textrm{ or }\, k_2\simeq k\geq k_1+2 \Big\} \,,
\end{align*}
and we will frequently take the summation over $(k_1,k_2)\in\chi_k$ (sometimes we simply write for short as $\sum_{k_1,k_2}$).
\end{remark}

We record the following basic inequalities used repeatedly in the nonlinear estimates.

\begin{lemma}[Young's Inequality]
Suppose that $f\in L^p(\r^3)$ and $g\in L^q(\r^3)$ and
\begin{align}
    \frac{1}{p}+\frac{1}{q}=1+\frac{1}{r},
\end{align}
with $1\leq p,q,r\leq\infty$. Then we have
\begin{align}
    \nm{f\ast g}_{L^r}\ls \nm{f}_{L^p}\nm{q}_{L^q}.
\end{align}
\end{lemma}
\begin{proof}
See Duoandikoetxea \cite[Corollary 1.21]{Duoandikoetxea2000}.
\end{proof}

\begin{lemma}[Minkowski's Integral Inequality]
Given measure spaces $(X,\mu)$ and $(Y,\nu)$ with $\sigma$-finite measures. Then for $1\leq p\leq \infty$,
\begin{align}
    \bigg(\int_X\abs{\int_Yf(x,y)\ud \nu(y)}^p\ud \mu(x)\bigg)^{\frac{1}{p}}\leq \int_Y\bigg(\int_X\abs{f(x,y)}^p\ud\mu(x)\bigg)^{\frac{1}{p}}\ud\nu(y).
\end{align}
\end{lemma}
\begin{proof}
See Duoandikoetxea \cite[Page xviii]{Duoandikoetxea2000}.
\end{proof}


\smallskip
\subsection{Preliminary Estimates} 

\begin{lemma}[Bilinear Estimates] \label{wtt prelim 9}
(i) Assume that $\ell\geq2$, $f_1,\cdots,f_{\ell},f_{\ell+1}\in L^2(\r^3)$, and $M: (\r^3)^{\ell}\rt\mathbb{C}$ is a continuously compactly supported function. Then
\begin{align}
    &\bigg|\int_{(\r^3)^{\ell}}M(\xi_1,\cdots,\xi_{\ell})\widehat{f_1}(\xi_1)\cdots\widehat{f_{\ell}}(\xi_{\ell})\widehat{f}_{\ell+1}(-\xi_1\cdots-\xi_{\ell})\,\ud\xi_1\cdots\ud\xi_{\ell}\bigg|
    \ls\nm{\mathscr{F}^{-1}M}_{L^1((\r^3)^{\ell})}\nm{f_1}_{L^{p_1}}\cdots\nm{f_{\ell}}_{L^{p_{\ell}}}\nm{f_{\ell+1}}_{L^{p_{\ell+1}}},
\end{align}
for any exponent $p_1,\cdots,p_{\ell},p_{\ell+1}\in[1,\infty]$ satisfying $\frac{1}{p_1}+\cdots\frac{1}{p_{\ell}}+\frac{1}{p_{\ell+1}}=1$.

(ii) If $q,p_1,p_2\in[1,\infty]$ satisfying $\frac{1}{q}=\frac{1}{p_1}+\frac{1}{p_2}$, then
\begin{align}
    \nm{\mathscr{F}^{-1}\bigg(\int_{\r^3}M(\xi,\eta)\widehat f(\eta)\widehat g(-\xi-\eta)\,\ud\eta\bigg)}_{L^q}\ls \nm{\mathscr{F}^{-1}M}_{L^1(\r^3\times\r^3)}\nm{f}_{L^{p_1}}\nm{g}_{L^{p_2}}.
\end{align}
\end{lemma}
\begin{proof}
This is Ionescu-Pausader \cite[Lemma 3.2]{Ionescu.Pausader2019}.
\end{proof}

\begin{lemma}[Lower Bound for Quadratic Phases] \label{wtt prelim 11}
Let $\text{\large$\Phi$}_{\,wa}^{\iota_1\iota_2}$ and $\text{\large$\Phi$}_{\,kg}^{\iota_1\iota_2}$ be the quadratic phases given in \eqref{Duhamel-wa} and \eqref{Duhamel-kg}.
Assume $\abs{\xi},\abs{\eta},\abs{\xi-\eta}\in[0,b]$ for $b\geq1$. Then we have
\begin{align}\label{phase-elliptic-2}
    \big|\text{\large$\Phi$}_{\,wa}^{\iota_1\iota_2}(\xi,\eta)\big|\geq \frac{\abs{\xi}}{4b^2},\qquad
    \big|\text{\large$\Phi$}_{\,kg}^{\iota_1\iota_2}(\xi,\eta)\big|\geq \frac{\abs{\eta}}{4b^2}.
\end{align}
\end{lemma}
\begin{proof}
This is Ionescu-Pausader \cite[Lemma 3.3]{Ionescu.Pausader2019}.
\end{proof}

\begin{lemma}[Lorentz Vector Fields]\label{wtt prelim 16}
Assume $\mu\in\{wa,kg\}$ and 
\begin{align}
    (\dt+\ui\Lambda_{\mu})U=\mathcal{N},
\end{align}
on $[1,\infty)\times\r^3$. If $V(t)=\ue^{\ui t\Lambda_{\mu}}U$ and $\ell\in\{1,2,3\}$, then for any $t\in[1,\infty)$, we have
\begin{align}
    \widehat{\Gamma_{\ell}U}(t,\xi)=\ui\big(\p_{\xi_{\ell}}\widehat{\mathcal{N}}\big)(t,\xi)+\ue^{-\ui t\Lambda_{\mu}(\xi)}\p_{\xi_{\ell}}\big[\Lambda_{\mu}(\xi)\widehat{V}(t,\xi)\big].
\end{align}
\end{lemma}
\begin{proof}
This is Ionescu-Pausader \cite[Lemma 6.1]{Ionescu.Pausader2019}.
\end{proof}


\smallskip
\subsection{Linear Estimates} 

This subsection is about the linear estimates either in the physical space or in the frequency space. 
In the following lemmas, $\sigma>0$ denotes a sufficiently small constant. When it is present, the inequality constant might depend on $\sigma$. Also, for all rotation vector fields $\Omega$,
\begin{align}
    \nm{f}_{H_{\Omega}^{0,1}}:=\sum_{\abs{\alpha}\leq 1}\bnm{\Omega^{\alpha}f}_{L^2}.
\end{align}

\begin{lemma}\label{wtt prelim 13}
For any $f\in L^2(\r^3)$, $(k,j)\in\jj$ and $\alpha\in(\mathbb{Z}_+)^3$, we have
\begin{align}
    \bnm{D_{\xi}^{\alpha}\widehat{\qq_{jk}f}}_{L^{2}_{\xi}}\ls 2^{\abs{\alpha}j}\bnm{\widehat{\qq_{jk}f}}_{L^{2}_{\xi}},\qquad \bnm{D_{\xi}^{\alpha}\widehat{\qq_{jk}f}}_{L^{\infty}_{\xi}}\ls 2^{\abs{\alpha}j}\bnm{\widehat{\qq_{jk}f}}_{L^{\infty}_{\xi}}.
\end{align}
\end{lemma}

\begin{proof}
This is Ionescu-Pausader \cite[(3.17)]{Ionescu.Pausader2019}.
\end{proof}

\begin{lemma}\label{wtt prelim 14}
For any $f\in L^2(\r^3)$ and $(k,j)\in\jj$, we have
\begin{align}
    \bnm{\widehat{Q_{jk}f}-\widehat{\qq_{jk}f}}_{L^{\infty}_{\xi}}\ls 2^{\frac{3}{2}j}2^{-4(j+k)}\tnm{P_kf}.
\end{align}
\end{lemma}

\begin{proof}
This is Ionescu-Pausader \cite[(3.21)]{Ionescu.Pausader2019}.
\end{proof}

\begin{lemma}\label{wtt prelim 2}
For any $f\in L^2(\r^3)$ and $(k,j)\in\jj$, we have
\begin{align}\label{wtt 16}
    \blnmx{\widehat{\qq_{jk}f}}\ls \min\left\{2^{\frac{3j}{2}}\btnm{Q_{jk}f},\,2^{\frac{j}{2}}2^{-k}2^{\frac{\sigma(j+k)}{8}}\bnm{Q_{jk}f}_{H_{\Omega}^{0,1}}\right\}.
\end{align}
\end{lemma}

\begin{proof}
This is Ionescu-Pausader \cite[(3.18)]{Ionescu.Pausader2019}.
\end{proof}

\begin{lemma}[Dispersive Estimates]\label{wtt prelim 5}
For any $f\in L^2(\r^3)$ and $(k,j)\in\jj$, we have
\begin{align}
    \nm{\ue^{-\ui t\lawa}\qq_{jk}f}_{L^{\infty}}&\ls \min\Big\{2^{\frac{3k}{2}},\,2^{\frac{3k}{2}}2^j\br{t}^{-1}\Big\}\bnm{Q_{jk}f}_{L^2},\\
    \nm{\ue^{-\ui t\lakg}\qq_{jk}f}_{L^{\infty}}&\ls \min\Big\{2^{\frac{3k}{2}},\,2^{3k^+}2^{\frac{3j}{2}}\br{t}^{-\frac{3}{2}}\Big\}\bnm{Q_{jk}f}_{L^2}.
\end{align}
\end{lemma}

\begin{proof}
This is Ionescu-Pausader \cite[(3.24),(3.28)]{Ionescu.Pausader2019}.
\end{proof}

\begin{lemma}\label{wtt prelim 6}
For any $f\in L^2(\r^3)$ and $(k,j)\in\jj$, if $\abs{t}\geq 1$, we have 
\begin{align}
    \nm{\ue^{-\ui t\lawa}\qq_{jk}f}_{L^{\infty}}&\ls \br{t}^{-1}2^{\frac{k}{2}}\big(1+2^{k}\br{t}\big)^{\frac{\sigma}{8}}\bnm{Q_{jk}f}_{H^{0,1}_{\Omega}},\quad \text{if}\ \ 2^j\leq \br{t}, 
    \label{wtt prelim 6-1}\\
     \nm{\ue^{-\ui t\lawa}\qq_{\leq jk}f}_{L^{\infty}}&\ls\br{t}^{-1}2^{2k}\bnm{\widehat{Q_{\leq jk}f}}_{L^{\infty}_{\xi}},\quad \text{if}\ \ 2^j\lesssim \br{t}^{\frac{1}{2}}2^{-\frac{k}{2}}. \label{wtt prelim 6-2}
\end{align}
\end{lemma}

\begin{proof}
This is Ionescu-Pausader \cite[(3.26),(3.27)]{Ionescu.Pausader2019}.
\end{proof}

\begin{lemma}\label{wtt prelim 7}
For any $f\in L^2(\r^3)$ and $(k,j)\in\jj$, if $1\leq 2^{2k^--20}\br{t}$, we have
\begin{align}
\label{wtt prelim 7-1}
    \bnm{\ue^{-\ui t\lakg}\qq_{jk}f}_{L^{\infty}}&\ls \br{t}^{-\frac{3}{2}}2^{5k^+}2^{-k^-}2^{\frac{j}{2}}\big(1+2^{2k^-}\br{t}\big)^{\frac{\sigma}{8}}\bnm{Q_{jk}f}_{H^{0,1}_{\Omega}},\quad \text{if}\ \ 2^j\leq 2^{k^-}\br{t},\\
     \bnm{\ue^{-\ui t\lakg}\qq_{\leq jk}f}_{L^{\infty}}&\ls\br{t}^{-\frac{3}{2}}2^{5k^+}\bnm{\widehat{Q_{\leq jk}f}}_{L^{\infty}_{\xi}},\quad \text{if}\ \ 2^j\lesssim \br{t}^{\frac{1}{2}}.\label{wtt prelim 7-2}
\end{align}
\end{lemma}

\begin{proof}
This is Ionescu-Pausader \cite[(3.29),(3.30)]{Ionescu.Pausader2019}.
\end{proof}

\begin{lemma}\label{wtt prelim 8}
For any $1\leq p\leq \infty$, we have
\begin{align}
    \bnm{\widehat{Q_{jk}f}}_{L^p}&\ls \bnm{\widehat{P_kf}}_{L^p},\qquad
    \bnm{\widehat{Q_{\leq jk}f}}_{L^p}\ls \bnm{\widehat{P_kf}}_{L^p}.
\end{align}
\end{lemma}

\begin{proof}
Since $Q_{jk}f=\varphi_j^{(k)}\cdot P_kf$, we know
\begin{align}
    \widehat{Q_{jk}f}(\xi)&=\int_{\r^3}\widehat{\varphi_j^{(k)}}(\xi-\eta)\widehat{P_kf}(\eta)\,\ud\eta.
\end{align}
Hence,
\begin{align}
    \bnm{\widehat{Q_{jk}f}}_{L^p}\ls\bnm{\widehat{P_kf}}_{L^p}\big\|\widehat{\varphi_j^{(k)}}\big\|_{L^1}.
\end{align}
In particular, $\big\|\widehat{\varphi_j^{(k)}}\big\|_{L^1}\ls 1$ uniformly in $j$ by scaling invariance.
\end{proof}

\begin{remark}
This lemma indicates that we can use $P_k$ to bound $Q_{jk}$. A natural corollary is that we can use $Q_{jk}$ to bound $\qq_{jk}$.
\end{remark}

\begin{lemma}\label{wtt lemma 2}
For $f\in L^2(\r^3)$ and $k\in\zz$, let
\begin{align}
    A_k&:=\bnm{P_kf}_{L^2}+\sum_{\ell=1}^3\bnm{\varphi_k\big(\p_{\xi_{\ell}}\hat f\big)}_{L^2_{\xi}},\qquad
    B_k:=\bigg(\sum_{j\geq- k^-}2^{2j}\bnm{Q_{jk}f}_{L^2}^2\bigg)^{\frac{1}{2}}.
\end{align}
Then for any $k\in\zz$,
\begin{align}
    A_k\ls \sum_{\abs{k'-k}\leq 4}B_{k'},
\end{align}
and
\begin{align}\label{wtt 07}
    B_k\ls\left\{
    \begin{array}{ll}
    \sum_{\abs{k'-k}\leq 4}A_{k'}&\quad \text{if}\quad k\geq0,\\
    \sum_{k'\in\zz}A_{k'}2^{-\frac{\abs{k-k'}}{2}}&\quad \text{if}\quad k\leq0.
    \end{array}
    \right.
\end{align}
\end{lemma}

\begin{proof}
This is Ionescu-Pausader \cite[Lemma 3.5]{Ionescu.Pausader2019}. 
\end{proof}

\begin{remark}
Roughly speaking, this lemma means that $A_k\sim B_k$, i.e. they are comparable.
\end{remark}


\smallskip
\subsection{Nonlinear Estimates} 

\begin{lemma}\label{prelim: nonlinear 1}
For $f,g\in L^2(\r^3)$, we have
\begin{align}\label{ttt 01}
    \bnm{\varphi_k\bbi_{wa}^{\iota_1\iota_2}\big[P_{k_1}f,P_{k_2}g\big]}_{L^2_\xi}\ls2^{\frac{3}{2}\min\{k_1,k_2,k\}}\nm{P_{k_1}f}_{L^2}\nm{P_{k_2}g}_{L^2},
\end{align}
and
\begin{align}\label{ttt 02}
    \nm{\varphi_k\bbi_{wa}^{\iota_1\iota_2}\big[P_{k_1}f,P_{k_2}g\big]}_{L^2_\xi}\ls\min\Big\{\nm{\ue^{-\ui t\lakg}P_{k_1}f}_{L^{\infty}}\nm{P_{k_2}g}_{L^2},\,\nm{P_{k_1}f}_{L^2}\nm{\ue^{-\ui t\lakg}P_{k_2}g}_{L^{\infty}}\Big\}.
\end{align}
\end{lemma}
\begin{proof}
Note that
\begin{align}
    \bbi_{wa}^{\iota_1\iota_2}\big[P_{k_1}f,P_{k_2}g\big]\simeq \int_{\mathbb{R}^3} \text{\large$\ue$}^{\ui t\Phi_{\,wa}^{\iota_1\iota_2}(\xi,\eta)} \text{\large$a$}_{\iota_1\iota_2}(\xi,\eta) \Big(\varphi_{k_1}(\xi\!-\!\eta)\widehat{f}(t,\xi\!-\!\eta)\Big)\Big(\varphi_{k_2}(\eta)\widehat{g}(t,\eta)\Big)\,\dd\eta.
\end{align}
Since $\babs{\text{\large$a$}_{\iota_1\iota_2}(\xi,\eta)}\ls 1$, $\text{\large$a$}_{\iota_1\iota_2}$ will not play a role in the estimate.

Using Young's inequality with $(p,q,r)=(2,2,\infty)$, we have
\begin{align}
    \nm{\varphi_k\bbi_{wa}^{\iota_1\iota_2}\big[P_{k_1}f,P_{k_2}g\big]}_{L^2_\xi}&\ls \nm{P_{k_1}f}_{L^2}\nm{P_{k_2}g}_{L^2}\left(\int_{\r^3}\varphi_k(\xi)\ud\xi\right)^{\frac{1}{2}}\ls2^{\frac{3}{2}k}\nm{P_{k_1}f}_{L^2}\nm{P_{k_2}g}_{L^2},
\end{align}
and with $(p,q,r)=(2,1,2)$, we have
\begin{align}
    \nm{\varphi_k\bbi_{wa}^{\iota_1\iota_2}\big[P_{k_1}f,P_{k_2}g\big]}_{L^2_\xi}&\ls\nm{P_{k_1}f}_{L^2}\nm{P_{k_2}g}_{L^1}\ls2^{\frac{3}{2}k_2} \nm{P_{k_1}f}_{L^2}\nm{P_{k_2}g}_{L^2}.
\end{align}
Similarly, we may show that
\begin{align}
    \nm{\varphi_k\bbi_{wa}^{\iota_1\iota_2}\big[P_{k_1}f,P_{k_2}g\big]}_{L^2_\xi}&\ls2^{\frac{3}{2}k_1} \nm{P_{k_1}f}_{L^2}\nm{P_{k_2}g}_{L^2}.
\end{align}
Hence, \eqref{ttt 01} is justified. On the other hand, \eqref{ttt 02} follows from Lemma \ref{wtt prelim 9} with $(p_1,p_2,q)=(\infty,2,2)$ or $(p_1,p_2,q)=(2,\infty,2)$.
\end{proof}

\begin{lemma}\label{prelim: nonlinear 2}
For $f,g\in L^2(\r^3)$, we have
\begin{align}\label{ttt 01'}
    \nm{\varphi_k\bbi_{kg}^{\iota_1\iota_2}\big[P_{k_1}f,P_{k_2}g\big]}_{L^2_\xi}\ls2^{\frac{3}{2}\min\{k_1,k_2,k\}}2^{k_1}2^{-k_2}\nm{P_{k_1}f}_{L^2}\nm{P_{k_2}g}_{L^2},
\end{align}
and
\begin{align}\label{ttt 02'}
    \nm{\varphi_k\bbi_{kg}^{\iota_1\iota_2}\big[P_{k_1}f,P_{k_2}g\big]}_{L^2_\xi}\ls2^{k_1}2^{-k_2}\min\Big\{\nm{\ue^{-\ui t\lakg}P_{k_1}f}_{L^{\infty}}\nm{P_{k_2}g}_{L^2},\,\nm{P_{k_1}f}_{L^2}\nm{\ue^{-\ui t\lawa}P_{k_2}g}_{L^{\infty}}\Big\}.
\end{align}
\end{lemma}
\begin{proof}
Note that for $|\xi\!-\!\eta|\approx 2^{k_1}$ and $|\eta|\approx 2^{k_2}$, $\babs{b_{\iota_1\iota_2}(\xi,\eta)}\ls 2^{k_1}2^{-k_2}$. Then the rest is similar to the proofs of \eqref{ttt 01} and \eqref{ttt 02}.
\end{proof}

\begin{lemma}\label{prelim: nonlinear 3}
For $\mathfrak{f},\mathfrak{g},\mathfrak{h}\in L^2(\r^3)$, we have
\begin{align}\label{ttt 02.}
    \abs{\iint_{\r^3\times\r^3}\ue^{\ui t\Phi_{wa}}a(\xi,\eta)\widehat{P_{k_1}\mathfrak{f}}(\xi\!-\!\eta)\widehat{P_{k_2}\mathfrak{g}}(\eta)\widehat{P_{k}\mathfrak{h}}(\xi)\,\ud\eta\ud\xi}
    \ls 2^{\frac{3}{2}\min(k,k_1,k_2)}\tnm{P_{k_1}\mathfrak{f}}\tnm{P_{k_2}\mathfrak{g}}\tnm{P_{k}\mathfrak{h}},
\end{align}
and
\begin{align}\label{ttt 01.}
    &\abs{\iint_{\r^3\times\r^3}\ue^{\ui t\Phi_{wa}}a(\xi,\eta)\widehat{P_{k_1}\mathfrak{f}}(\xi\!-\!\eta)\widehat{P_{k_2}\mathfrak{g}}(\eta)\widehat{P_{k}\mathfrak{h}}(\xi)\,\ud\eta\ud\xi}\\
    \ls&\, \min\Big\{\nm{\ue^{-\ui t\lakg}P_{k_1}\mathfrak{f}}_{L^{\infty}}\nm{P_{k_2}\mathfrak{g}}_{L^{2}}\nm{P_{k}\mathfrak{h}}_{L^{2}},\, \nm{P_{k_1}\mathfrak{f}}_{L^{2}}\nm{\ue^{-\ui t\lakg}P_{k_2}\mathfrak{g}}_{L^{\infty}}\nm{P_{k}\mathfrak{h}}_{L^{2}},\,\nm{P_{k_1}\mathfrak{f}}_{L^{2}}\nm{P_{k_2}\mathfrak{g}}_{L^{2}}\nm{\ue^{-\ui t\lawa}P_{k}\mathfrak{h}}_{L^{\infty}}\Big\}.\no
\end{align}
\end{lemma}
\begin{proof}
This is a direct consequence of Lemma \ref{prelim: nonlinear 1}.
\end{proof}

\begin{lemma}\label{prelim: nonlinear 4}
For $\mathfrak{f},\mathfrak{g},\mathfrak{h}\in L^2(\r^3)$, we have
\begin{align}\label{ttt 02..}
    \abs{\iint_{\r^3\times\r^3}\ue^{\ui t\Phi_{kg}}b(\xi,\eta)\widehat{P_{k_1}\mathfrak{f}}(\xi\!-\!\eta)\widehat{P_{k_2}\mathfrak{g}}(\eta)\widehat{P_{k}\mathfrak{h}}(\xi)\,\ud\eta\ud\xi}
    \ls 2^{\frac{3}{2}\min(k,k_1,k_2)}2^{k_1}2^{-k_2}\tnm{P_{k_1}\mathfrak{f}}\tnm{P_{k_2}\mathfrak{g}}\tnm{P_{k}\mathfrak{h}},
\end{align}
and
\begin{align}\label{ttt 01..}
    &\abs{\iint_{\r^3\times\r^3}\ue^{\ui t\Phi_{kg}}b(\xi,\eta)\widehat{P_{k_1}\mathfrak{f}}(\xi\!-\!\eta)\widehat{P_{k_2}\mathfrak{g}}(\eta)\widehat{P_{k}\mathfrak{h}}(\xi)\,\ud\eta\ud\xi}\\
    \ls&\, 2^{k_1}2^{-k_2}\min\Big\{\nm{\ue^{-\ui t\lakg}P_{k_1}\mathfrak{f}}_{L^{\infty}}\nm{P_{k_2}\mathfrak{g}}_{L^{2}}\nm{P_{k}\mathfrak{h}}_{L^{2}},\, \nm{P_{k_1}\mathfrak{f}}_{L^{2}}\nm{\ue^{-\ui t\lakg}P_{k_2}\mathfrak{g}}_{L^{\infty}}\nm{P_{k}\mathfrak{h}}_{L^{2}},\,\nm{P_{k_1}\mathfrak{f}}_{L^{2}}\nm{P_{k_2}\mathfrak{g}}_{L^{2}}\nm{\ue^{-\ui t\lakg}P_{k}\mathfrak{h}}_{L^{\infty}}\Big\}.\no
\end{align}
\end{lemma}
\begin{proof}
This is a direct consequence of Lemma \ref{prelim: nonlinear 2}.
\end{proof}

\bigskip
\section*{Acknowledgements}

The author would like to thank 
Professor Benoit Pausader for many helpful discussions and valuable comments on the manuscript. 

\bigskip

\bibliographystyle{siam}
\bibliography{Reference}

\end{document}